\documentclass[aos,preprint,reqno]{imsart}

\RequirePackage[OT1]{fontenc}
\RequirePackage{amsthm,amsmath}
\RequirePackage[numbers]{natbib}
\RequirePackage[colorlinks,linkcolor=blue,citecolor=blue,urlcolor=blue]{hyperref}
\usepackage{amssymb}
\usepackage{hyperref}
\usepackage{extarrows}
\usepackage{multirow}
\usepackage{natbib}
\usepackage{stmaryrd}
\usepackage{color}
\usepackage{amsmath}
\usepackage{enumitem}
\usepackage{graphicx}
\usepackage{amstext}
\usepackage{bbm}
\usepackage{caption}
\usepackage{subcaption}
 \usepackage{floatrow}
 
\numberwithin{equation}{section}

\theoremstyle{plain} 
\newtheorem{thm}{Theorem}[section]
\newtheorem{lem}[thm]{Lemma}
\newtheorem{cor}[thm]{Corollary}
\newtheorem{pro}[thm]{Proposition}

\newtheorem{defn}[thm]{Definition}

\theoremstyle{remark}
\newtheorem{rem}[thm]{Remark}

\renewcommand{\Re}{\mathrm{Re}\,}
\renewcommand{\Im}{\mathrm{Im}\,}

\newcommand{\R}{{\mathbb R }}

\newcommand{\Z}{{\mathbb Z}}
\renewcommand{\P}{{\mathbb P}}
\newcommand{\C}{{\mathbb C}}

\newcommand{\ii}{\mathrm{i}}
\newcommand{\deq}{\mathrel{\mathop:}=}

\newcommand{\ie}{\emph{i.e., }}

\newcommand{\wt}{\widetilde}
\newcommand{\wh}{\widehat}
\newcommand{\mf}{\mathfrak}
\newcommand{\mc}{\mathcal}
\newcommand{\mr}{\mathrm}
\newcommand{\mb}{\mathbf}

\newcommand{\bs}{\boldsymbol}

\newcommand*{\Cdot}[1][1.25]{%
  \mathpalette{\CdotAux{#1}}\cdot%
}
\newdimen\CdotAxis
\newcommand*{\CdotAux}[3]{%
  {%
    \settoheight\CdotAxis{$#2\vcenter{}$}%
    \sbox0{%
      \raisebox\CdotAxis{%
        \scalebox{#1}{%
          \raisebox{-\CdotAxis}{%
            $\mathsurround=0pt #2#3$%
          }%
        }%
      }%
    }%
    \dp0=0pt %
    \sbox2{$#2\bullet$}%
    \ifdim\ht2<\ht0 %
      \ht0=\ht2 %
    \fi
    \sbox2{$\mathsurround=0pt #2#3$}%
    \hbox to \wd2{\hss\usebox{0}\hss}%
  }%
}

\renewcommand{\mathbf}[1]{\bs{#1}}

\startlocaldefs
\endlocaldefs

\begin{document}

\begin{frontmatter}

\title{Tracy-Widom limit for Kendall's tau}
\runtitle{TW law for Kendall tau}

\author{\fnms{Zhigang} \snm{ Bao}\ead[label=e1]{mazgbao@ust.hk}\thanksref{t1}}
\address{\printead{e1}}
\thankstext{t1}{The author is partially supported by Hong Kong RGC GRF Grant 16300618}
\address{Department of Mathematics\\ Hong Kong University of Science and Technology}
\affiliation{Hong Kong University of Science and Technology}

\runauthor{Zhigang Bao}

\begin{abstract}
In this paper, we study a high-dimensional random matrix model from nonparametric statistics called the Kendall rank  correlation matrix, which is a natural multivariate extension of the Kendall rank correlation coefficient. We establish the Tracy-Widom law for its largest eigenvalue. It is the first Tracy-Widom law  for a nonparametric random matrix model, and also the first  Tracy-Widom law  for a high-dimensional U-statistic.
\end{abstract}



\end{frontmatter}

\section{Introduction}
Let $\mb{w}=(w_1,\ldots, w_p)'$ be a $p$-dimensional random vector. We assume that all the components of $\mb{w}$ are independent continuous random variables. We do not require the components to be identically distributed, and no moment assumption on the components of $\mb{w}$ is needed. Let  $\mathbf{w}_j=(w_{1j}, \ldots, w_{pj})', j\in \llbracket1, n\rrbracket$ be $n$ i.i.d. samples of $\mathbf{w}$. Hereafter we use the notation $\llbracket a, b\rrbracket:=[a, b]\cap \mathbb{Z}$.  We also denote by $W=(w_{ij})_{p,n}$ the data matrix.  In the paper, we assume that  $p$ and $n$ are comparable. More specifically, we assume
\begin{align}
p=p(n), \qquad c_n:=\frac{p}{n}\to c\in (0,\infty), \qquad \text{if} \quad n\to \infty,  \label{assump on dim}
\end{align}
for some positive constant $c$. 

From  the data matrix $W$, we can further construct a matrix model called {\it Kendall  rank correlation matrix}, originating from nonparametric statistics. The definition is detailed as follows.

\subsection { Kendall rank correlation matrix}

Recall the data matrix $W=(w_{ij})_{p,n}$.  For any given $k\in \llbracket 1, p\rrbracket$, we denote 
\begin{align}
v_{k,(ij)}:=\text{sign}(w_{ki}-w_{kj}), \qquad \forall i\neq j  \label{17091820}
\end{align}
and let 
\begin{align}
\bs{\theta}_{(ij)}:=\frac{1}{\sqrt{M}} (v_{1,(ij)}, \ldots, v_{p,(ij)})', \label{17120201}
\end{align}
 where  for brevity we set 
\begin{align*}
M\equiv M(n):=\frac{n(n-1)}{2}. 
\end{align*}
The  Kendall  rank correlation matrix is defined as the following sum of $M$ rank-one matrices 
\begin{align}
K\equiv K_n:= \sum_{i< j} \bs{\theta}_{(ij)} \bs{\theta}_{(ij)}'=\Theta\Theta'. \label{17101901}
\end{align}
Here we denote by 
\begin{align}
\Theta:=(\bs{\theta}_{(12)},\ldots, \bs{\theta}_{ (1n)}, \bs{\theta}_{(23)},\ldots, \bs{\theta}_{(2n)}  \ldots, \bs{\theta}_{(n-1,n)}). \label{def of Theta}
\end{align} 
Observe that the rank-one matrices $\bs{\theta}_{(ij)} \bs{\theta}_{(ij)}'$'s are not independent. For instance, $\bs{\theta}_{(ij)} \bs{\theta}_{(ij)}'$ and $\bs{\theta}_{(ik)} \bs{\theta}_{(ik)}'$ are correlated even if $j\neq k$.  Moreover, $K$ is a $p\times p$ matrix, and its $(a,b)$-entry is
\begin{align*}
K_{ab}=\frac{1}{M} \sum_{i<j} v_{a, (ij)} v_{b, (ij)}= \frac{1}{M}\sum_{i<j}  \text{sign}(w_{ai}-w_{aj}) \text{sign}(w_{bi}-w_{bj}),
\end{align*}
which is exactly the Kendall rank correlation coefficient between the samples of $w_a$ and those of $w_b$.  Hence, the matrix $K$ is a natural multivariate extension of the Kendall rank correlation coefficient. 

\subsection{Motivation} Since the seminal work of Marchenko and Pastur \cite{MP}, the spectral properties of large dimensional  sample covariance matrix and its variants  have attracted enormous attention. In \cite{MP}, the famous Marchenko-Pastur law (MP-law) for the global spectral distribution of the sample covariance matrices  has been raised.  On the local scale,  Johnstone \cite{Johnstone} proved the Tracy-Widom law (TW law) for the largest eigenvalue of the real Gaussian sample covariance matrix (Wishart matrix)  in the null case, i.e., the population covariance matrix is $I_p$.  Since the largest eigenvalue plays a fundamental role in principal component analysis (PCA), the TW law can be applied to many  PCA-related problems in high-dimensional scenarios.  The TW law was then shown to be universal for sample covariance matrices in the null case, even under more general distribution assumptions; see \cite{Wang, PY}.   In \cite{BPZ1, PY1}, it was also shown that the TW law holds for the (Pearson) sample correlation matrix in the null case.  We also mention \cite{Johansson, Karoui, Onatski} as they give related results for complex sample covariance matrices.  Recently, the universality was further established for more general population; see \cite{BPZ, LS, KY, FJ}. 

Both the sample covariance matrix and (Pearson) sample correlation matrix are parametric models.   Many spectral statistics such as the largest eigenvalue of the sample covariance matrix or correlation matrix are used for testing the hypothesis of {\it independence} among the entries of a random vector. The strategy is certainly feasible for Gaussian vectors. However, for non-Gaussian vectors, even in the classical large $n$ and fixed $p$ case, the idea of comparing population covariance matrix with diagonal matrix cannot be used for an independence test involving uncorrelated but dependent variables. 
On the other hand,  although the TW law was shown to be universal for sample covariance matrices, assumptions on the distribution of the matrix entries are still required to a certain extent; see for instance, the minimal moment condition in \cite{DY}. This moment requirement certainly excludes all heavy-tailed data sets. For the above reasons, a more robust nonparametric approach is needed. 

In classical  nonparametric statistics, the most famous statistics concerning the statistical dependence between two random variables are the Spearman rank correlation coefficient and the Kendall rank correlation coefficient, also known as Spearman's $\rho$ and Kendall's $\tau$. Both of them have natural 
multivariate extensions, which are called Spearman rank correlation matrix and Kendall  rank correlation matrix (c.f. (\ref{17101901})), respectively. Since these models are nonparametric, all the hypothesis tests based on statistics of these models are distribution-free.  However, in contrast to the parametric models, the study on the spectral properties of the high-dimensional nonparametric matrices is much less. Under the null hypothesis, i.e., the components of $\mathbf{w}$ are independent,  the global spectral distributions for the Spearman rank  correlation matrix and Kendall rank correlation matrix  have been derived in \cite{BZ} and \cite{BLR}, respectively.  A CLT for the linear eigenvalue statistics of the Spearman rank  correlation matrix has  been considered in \cite{BLPZ}. However, so far, there is no result on the local eigenvalue statistics such as the largest eigenvalue of these two nonparametric models.
In this work, our aim is to establish the TW law for the Kendall rank correlation matrix. In a companion paper \cite{BaoSpearman}, we show that the TW law also holds for the Spearman rank correlation matrix.   

 Moreover, it is also well-known that Kendall's tau is a U-statistic. The spectral theory on general high-dimensional U-statistics is still unexplored, except for the global law of Kendall's tau in \cite{BLR}. The result in this paper can also be regarded as the first TW law established for a high-dimensional U-statistic. Furthermore, we expect that the method developed in this paper will, to a certain extent, have potential applications to other high-dimensional U-statistics. 

\subsection{Global behavior of the spectrum} \label{subsect. global laws} In this subsection, we first review the result on the global law from \cite{BLR}.  
Let  $\lambda_{1}(K)\geq \ldots\geq \lambda_{p}(K)$ be $p$ ordered eigenvalues of $K$.  Denote the empirical spectral distribution (ESD) of $K$ by 
\begin{align*}
F_{n}^K:= \frac{1}{p}\sum_{i=1}^p \delta_{\lambda_i(K)}. 
\end{align*}
  In \cite{BLR}, it is proved the $F_n^K$ is asymptotically given by a scaled and shifted MP law.   To state the result in \cite{BLR}, we first introduce the Marchencko Pastur law $F_c$ (with parameter $c$), whose density function is given by 

\begin{align*}
\rho_c(x)= \frac{1}{2\pi c}\frac{\sqrt{({d}_{+,c}-x)(x-{d}_{-,c})}}{x}\mathbbm{1}({d}_{-,c}\leq x\leq {d}_{+,c})
\end{align*}
where
$
{d}_{\pm,c}=(1\pm \sqrt{c})^2. 
$
In case $c>1$, in addition, $F_{c}$  has a singular part: a point mass $(1-c^{-1})\delta_0$.

\begin{thm}[Theorem 1 of \cite{BLR}] \label{thm. global law of K} Under the assumption (\ref{assump on dim}), we have that  $F_n^K$ converges weakly (in probability) to  $F^K_c$ whose density is given by 
\begin{align*}
\rho^K_c(x)=\frac32\rho_c(\frac32x-\frac12).
\end{align*}
Hence, $F^K_c(x)= F_c(\frac32x-\frac12)$. 
\end{thm}

 Further, replacing $c$ by $c_n$,  we denote by  $\rho_{c_n}$,  $\rho^K_{c_n}$, $F_{c_n}$, $F^K_{c_n}$, ${d}_{\pm,c_n}$  the analogues of  $\rho_{c}$,  $\rho^K_{c}$, $F_{c}$, $F^K_{c}$, ${d}_{\pm,c}$, respectively.  Further, we introduce the shorthand notation
\begin{align}
{\lambda}_{\pm, c_n}:=\frac23 {d}_{\pm,c_n}+\frac13.  \label{17110101}
\end{align}

\subsection{Main results}\label{subsec. main results}

To state our main results, we denote by $Q:=\frac{1}{n} \mathcal{X}\mathcal{X}'$ a Wishart matrix, where $\mc{X}$ is a $p\times n$ data matrix with i.i.d. $N(0,1)$ variables. Let $\lambda_{i}(Q)$ be the $i$-th largest eigenvalue of $Q$.  Our main results are as follows.  
\begin{thm}[Edge universality of Kendall rank correlation matrix] \label{thm. universality for Kendall} Suppose that  the assumption (\ref{assump on dim}) holds. There exist positive constants $\varepsilon$ and $\delta$ such that for any  $s\in \mathbb{R}$, the following holds for all sufficiently large $n$
\begin{align}
\mathbb{P} \Big(\frac32n^{\frac23}(\lambda_1(K)- {\lambda}_{+,c_n})\leq s-n^{-\varepsilon}\Big)-n^{-\delta}\leq \mathbb{P} \Big(n^{\frac23}(\lambda_1(Q)- d_{+,c_n})\leq s\Big)\nonumber\\
\leq  \mathbb{P} \Big(\frac32 n^{\frac23}(\lambda_1(K)- {\lambda}_{+,c_n})\leq s+n^{-\varepsilon}\Big)+n^{-\delta}. 
\label{17120265}
\end{align}
\end{thm}

\begin{rem} The above theorem can be extended to the joint distribution for the first $k$ leading eigenvalues. 
We refer to Remark 1.4 of \cite{PY} for a similar extension for the sample covariance matrix. The extension here can be done in the same way. 
\end{rem}
From Theorem \ref{thm. universality for Kendall}, we can get the following corollary.
\begin{cor}[Tracy-Widom law for $\lambda_1(K)$]\label{cor. TW for K} Under the assumption of Theorem \ref{thm. universality for Kendall}, we have 
\begin{align*}
\frac32 n^{\frac23}c_n^{\frac16} d_{+,c_n}^{-\frac23}\big(\lambda_1(K)- {\lambda}_{+,c_n}\big)\Longrightarrow \mathrm{TW}_1,
\end{align*}
where $\mathrm{TW}_1$ stands for the Tracy-Widom law of type I.
\end{cor}

\subsection{Proof strategy} 

In the sequel, we summarize our proof strategy with a highlight on the novelties.  
Our proof strategy traces back to the seminal works of Erd\H{o}s, Yau and Yin \cite{EYY12,EYY},  where a general framework to prove the universality of local eigenvalue statistics  has been raised.   Roughly speaking,  the strategy in \cite{EYY} for proving the edge universality consists of two major steps. First, one needs to prove a local law for the spectral distribution, from which one can get a control on the location of the eigenvalues on an optimal local scale. Second, with the aid of the local law,  one needs to  perform a Green function comparison between the matrix of interest and a certain reference matrix ensemble, whose edge spectral behavior is already known. 
In the Green function comparison step, one translates the comparison between the distributions of the largest eigenvalues of two random matrices to  a comparison of  their Green functions.   The Green function turns out to be a more convenient object to look into, due to the simple resolvent expansion mechanism. 
An adaptation of this general strategy was used by Pillai and Yin in \cite{PY} to show both the bulk and edge universality of the sample covariance matrices.  Especially, in \cite{PY}, an extended criterion of the local law for covariance type of matrices with independent columns (or rows) was given; see Theorem 3.6 of \cite{PY}. It allows one to relax the independence assumption on the entries within each single column (or row) to a certain extent, as long as some large deviation estimates hold for certain linear and quadratic forms of  each column (or row) of the data matrix; see Lemma 3.4 of \cite{PY}.    This general criterion was then used in \cite{PY1} and \cite{BPZ1} to establish the edge universality of the sample correlation matrices. 
 
In order to illustrate the new ingredients in applying the above general strategy to our model, we first introduce some notations. For any parameter $z\in \mathbb{C}^+$, we denote by $G(z)=(G_{k\ell}(z)):=(K-z)^{-1}$ the Green function of $K$ and  by $m(z):=\frac{1}{p}\text{Tr} G(z)$ the normalized trace of the Green function, which is also the Stieltjes transform of the ESD $F_n^K$. Let $\underline{m}(z)$ be the Stieltjes transform of $F_c^K$.  For our matrix $K$, 
in the step of local law, one needs to establish the following estimates
\begin{align}
|G_{k\ell}(z)- \delta_{k\ell} \underline{m}(z)|\prec \Psi(z),\label{18080703}\\
|m(z)-\underline{m}(z)|\prec\frac{1}{n\Im z} \label{18080704}
\end{align}
in the domain ${\mathcal{D}}(\epsilon)$  (c.f. (\ref{18080701})). We also refer to (\ref{18080702}) and Definition \ref{definition of stochastic domination} for the definition of $\Psi(z)$ and the notation $\prec$, respectively. It is now well understood that a large deviation estimate of $\lambda_i(K)$ around its classical location can be derived from the local law. However,  the large deviation estimate does not tell the $\mathrm{TW}$ law of $\lambda_1(K)$ directly, although  together with (\ref{18080703}) and (\ref{18080704}) it will serve as an important input for the proof of the $\mathrm{TW}$ law.  As we mentioned above, for  $\mathrm{TW}$ law, as the next step,  we need to conduct a Green function comparison. In this step, we will compare the distribution function of $\lambda_1(K)$ with that of $\lambda_1(\widetilde{K})$,  where $\widetilde{K}$ (c.f. (\ref{18080710})) is a shifted covariance matrix and the law of $\lambda_1(\widetilde{K})$ is known to be $\mathrm{TW}_1$. The comparison of the distributions  can be translated into the comparison of the Green functions, and it suffices to show 
\begin{align}
&\Big| \mathbb{E} F\Big( n \int_{E_1}^{E_2} \Im m (x+{\lambda}_{+,c_n}+\mathrm{i}\eta){\rm d} x\Big)\nonumber\\
&\qquad\qquad-\mathbb{E} F\Big( n \int_{E_1}^{E_2} \Im \wt{m} (x+{\lambda}_{+,c_n}+\mathrm{i}\eta){\rm d} x\Big)\Big|\leq n^{-\delta}, \label{18080720}
\end{align}
where $F$ is a smooth test function and $\wt{m}$ stands for the Stieltjes transform of the ESD of $\wt{K}$. We refer to Proposition \ref{lem.17100801} for the setting of $\eta$, $E_1$ and $E_2$. The proof of (\ref{18080720}) will heavily rely on (\ref{18080703}) and (\ref{18080704}).

As we mentioned above, the Kendall rank correlation matrix is a multivariate U-statistic. Its structure is significantly different from the sample covariance matrix or correlation matrix.   Although the rows of $\Theta$ are mutually independent,  there is a strong dependence structure among the  entries   within each row.  Consequently, both the proofs of the two steps, i.e., local law and Green function comparison, require novel ideas. 

The starting point of the whole proof is (a variant of) Hoeffding decomposition \cite{Hoeffding},  which is already used  for the global law in \cite{BLR}. Specifically, for Kendall rank correlation, we can decompose $v_{k,(ij)}$ (c.f. (\ref{17091820})) as 
\begin{align}
v_{k,(ij)}=u_{k,(ij)}+\bar{v}_{k,(ij)},  \label{17111810}
\end{align}
where
\begin{align}
u_{k,(ij)}:=\mathbb{E}\big( \text{sign}(w_{ki}-w_{kj})|w_{ki}\big)+\mathbb{E}\big( \text{sign}(w_{ki}-w_{kj})|w_{kj}\big),  \label{17111811}
\end{align} 
and we take the above as the definition of $\bar{v}_{k,(ij)}$. It is easy to check that $u_{k,(ij)}$ and $\bar{v}_{k,(ij)}$ are uncorrelated.  Correspondingly, we set the $p\times M$ matrices $U=\frac{1}{\sqrt{M}}(u_{k,(ij)})_{k,(ij)}$ and $\bar{V}=\frac{1}{\sqrt{M}}(\bar{v}_{k,(ij)})_{k,(ij)}$. Hence, we have the decomposition $\Theta=U+\bar{V}$. In the sequel, we will call $U$ the {\it linear } part of $\Theta$, and $\bar{V}$ the {\it nonlinear } part of $\Theta$.  It will be seen that $UU'$  is indeed a covariance type of matrix and its spectral property can be obtained from the results on sample covariance matrices easily.  However, in $K=\Theta\Theta'=(U+\bar{V})(U+\bar{V})'$, we also have the crossing parts $\bar{V}U'$, $U\bar{V}'$ and the purely nonlinear part $\bar{V}\bar{V}'$. The nonlinear term $\bar{V}$ couples the columns of $\Theta$ together, and makes the structure of $K$  different from the covariance matrix.  

For the step of local law, recall our tasks (\ref{18080703}) and (\ref{18080704}). We take the estimate of the diagonal entries $G_{kk}$'s as an example. By Schur complement, one can write $G_{kk}$  in terms of a quadratic form $\mb{v}_k B^{(k)}\mb{v}_k'$; see (\ref{17091701}) for more details. Here $\mb{v}_k$ is the $k$-th row of $\Theta$ and it is independent of $B^{(k)}$. Hence, an estimate of $G_{kk}$ essentially boils down to a large deviation estimate of  the quadratic form of $\mb{v}_k$. It turns out that although a direct large deviation estimate is enough for  (\ref{18080703}), it is not sufficient for later use in the  Green function comparison. 
 With Hoeffding decomposition, we can write $\mb{v}_k B^{(k)}\mb{v}_k'$ as a linear combination of the linear part $\mb{u}_k B^{(k)}\mb{u}_k'$, crossing part  $\mb{u}_k B^{(k)}\bar{\mb{v}}_k'$ and the nonlinear part $\bar{\mb{v}}_k B^{(k)}\bar{\mb{v}}_k'$, where $\mb{u}_k$ and $\bar{\mb{v}}_k$ are the $k$-th rows of $U$ and $\bar{V}$, respectively. 
We establish the large deviation estimates  for three parts separately; see Propositions \ref{lem. large deviation linear part} and \ref{lem. large deviation nonlinear part}.  
It turns out that the large deviations of the last two parts are much sharper than the first part, although the sharpness for the crossing part can been seen only a posteriori.  The sharper large deviation estimates for the crossing part and nonlinear part  will be crucial in  Green function comparison.  The proof of Proposition \ref{lem. large deviation nonlinear part} will be the major task in this step. The matrices $U$ and $\bar{V}$ are only uncorrelated rather than independent, and so are  the entries within $\bar{V}$. To prove Proposition \ref{lem. large deviation nonlinear part}, we need to perform a martingale concentration argument. With these large deviation estimates, 
we then prove the local law, by pursuing the strategy in \cite{EYY} and \cite{PY}.  

For  Green function comparison (\ref{18080720}), we further decompose it into two steps. We call the first step as {\it decoupling}, and the second step as {\it first-order approximation}.  In the {\it decoupling } step, we compare $K=(U+\bar{V})(U+\bar{V})'$ with  $\wh{K}=(U+H)(U+H)'$, where $H=(h_{k,(ij)})$ is a $p\times M$ Gaussian matrix with i.i.d. $h_{k,(ij)}\sim N(0,\frac{1}{3M})$ and it is independent of $U$. This step allows us to decouple the dependent (although uncorrelated) pair $(U,\bar{V})$ by studying the independent pair $(U, H)$ instead. For the Green function comparison between $K$ and $\wh{K}$, we use a swapping strategy via replacing one row of $\bar{V}$ by that of $H$ at each time and compare the Green functions step by step. Such a replacement strategy has been previously used in \cite{PY}, and also \cite{PY1, BPZ1, BPZ}.  However, such a comparison involves high order moments of the quadratic forms of $\mathbf{v}_k$ and $\hat{\mathbf{v}}_k$, where $\hat{\mathbf{v}}_k$ represents the $k$-th row of $U+H$.  Roughly speaking, the comparison requires  the first three moments of $\mathbf{v}_k B\mathbf{v}_k'$ and $\hat{\mathbf{v}}_k B\hat{\mathbf{v}}_k'$ and their variants to match, up to sufficiently small errors. Here $B$ is certain matrix independent of both  $\mathbf{v}_k$ and $\hat{\mathbf{v}}_k$.  Although  the entries in $\bar{V}$ and those in $H$ have the same covariance structure, their higher order moments do not match. In addition, although the entries in $U$ and those in $\bar{V}$  are uncorrelated, they are dependent at high orders.   One key point in the comparison of the moments of $\mathbf{v}_k B\mathbf{v}_k'$ and those of $\hat{\mathbf{v}}_k B\hat{\mathbf{v}}_k'$ is to show that the high order correlation between the entries in $U$ and $\bar{V}$ is negligible. This fact heavily relies on the sharper large deviations for the crossing part and nonlinear part in Proposition \ref{lem. large deviation nonlinear part}.  In the {\it first-order approximation} step, we further compare $\wh{K}=(U+H)(U+H)'$ with the random matrix $\wt{K}$. In this step, we approximate all the terms with the matrix $H$ involved by the deterministic $\frac{1}{3}I_p$. The Green function comparison between $\wh{K}$ and $\wt{K}$ will be done with a continuous interpolation between two matrices.   Similar idea of  continuous interpolation  was previously used for the Green function comparison in \cite{LS1, LS}.

\subsection{Notation and organization}

We first need the following definition  from~\cite{EKY}. 

\begin{defn}\label{definition of stochastic domination}
Let $\mathsf{X}\equiv \mathsf{X}^{(n)}$ and $\mathsf{Y}\equiv \mathsf{Y}^{(n)}$ be two sequences of
 nonnegative random variables. We say that~$\mathcal{Y}$ stochastically dominates~$\mathsf{X}$ if, for all (small) $\epsilon>0$ and (large)~$D>0$,
\begin{align}
\P\big(\mathsf{X}^{(n)}>n^{\epsilon} \mathsf{Y}^{(n)}\big)\le n^{-D},
\end{align}
for sufficiently large $n\ge n_0(\epsilon,D)$, and we write $\mathsf{X} \prec \mathsf{Y}$ or $\mathsf{X}=O_\prec(\mathsf{Y})$.
 When
$\mathsf{X}^{(n)}$ and $\mathsf{Y}^{(n)}$ depend on a parameter $v\in \mathsf{V}$ (typically an index label or a spectral parameter), then $\mathsf{X}(v) \prec \mathsf{Y} (v)$, uniformly in $v\in \mathsf{V}$, means that the threshold $n_0(\epsilon,D)$ can be chosen independently of $v$. We also use the notation $\mathsf{X}^{(n)}\prec \mathsf{Y}^{(n)}$ if $\mathsf{X}^{(n)}\leq n^{\epsilon} \mathsf{Y}^{(n)}$ deterministically for any given (small) $\epsilon>0$. Finally, we say that an event $\mathcal{E}\equiv \mathcal{E}_n$ holds with high probability if: for any fixed $D>0$, there exists $n_0(D)>0$, such that for all $n\geq n_0(D)$ we have
\begin{align*}
\mathbb{P}(\mathcal{E})\geq 1-n^{-D}.
\end{align*} 
\end{defn}
In the case that the nonnegative random variable $\mathsf{X}$ satisfies the stochastic bound $\mathsf{X}\prec \mathsf{Y}$ and the deterministic bound $\mathsf{X}\leq N^k\mathsf{Y}$ for some nonnegative integer $k$ and nonnegative $Y$, we can easily conclude that $\mathbb{E}X^p\prec  \mathbb{E} Y^p$ for any given $p\geq 0$. We use the symbols $O(\,\cdot\,)$ and $o(\,\cdot\,)$ for the standard big-O and little-o notation.  We use~$C$ to denote strictly positive constant that does not depend on~$N$. Its value may change from line to line. For any matrix $A$, we denote by $\|A\|$ its operator norm, while for any vector $\mathbf{a}$, we use $\|\mathbf{a}\|$ to denote its $\ell^2$-norm. Further, we use $\|\mathbf{a}\|_\infty$ to represent the $\ell^\infty$-norm of a vector.  In addition, 
we use double brackets to denote index sets, \ie for $n_1, n_2\in\R$, $\llbracket n_1,n_2\rrbracket\deq [n_1, n_2] \cap\Z$. The notation $\mathbbm{1}(\cdot)$ will be used to denote the indicator function. We  also use $\mathbf{1}$ to represent the all-one vector, whose dimension may change from one to another.

The paper is organized as follows:  In Section \ref{s.simulation}, we will present a simulation study to show that the testing statistic of the largest eigenvalue of the Kendall rank correlation matrix has good performance in the independence test.  In Section \ref{s.large deviation}, we will state some large deviation estimates which will be used in the later sections. In Section \ref{s.local law} we will state a local law of $K$. In Section \ref{s. decoupling}, we will compare the Green functions of $K$ and $\wh{K}$, where the latter has independent linear and ``nonlinear" parts. In Section \ref{s.first order appro},  we further compare the Green functions of $\wh{K}$ and $\wt{K}$, where the latter is a shift of the linear part only.    Section \ref{s. universality for K} will be devoted to the final proof of  Theorem \ref{thm. universality for Kendall} and Corollary \ref{cor. TW for K}.  The proofs of the large deviation bounds, the local law, and some technical lemmas will be stated in the supplementary material \cite{BaoSupp}. In addition, we also present more simulation results in  \cite{BaoSupp}. 

\section{Application and simulation study}  \label{s.simulation}
In this section, we  apply the $\textrm{TW}_1$ law for $K$ to test the complete independence of the components of the random vector $\mathbf{w}=(w_1,\ldots,w_p)'$. We also compare the performance of our statistic, i.e., $\lambda_1(K)$, with some other statistics in the literature. From  the $n$ samples of $\mathbf{w}$, i.e. $\mathbf{w}_1, \ldots, \mathbf{w}_n$, we can define three types of correlation matrices: Pearson correlation matrix ($R$), Spearman rank correlation matrix ($S$), and Kendall rank correlation matrix (K). By definition, the matrix entries $R_{ij}$, $S_{ij}$ and $K_{ij}$ are the Pearson, Spearman and Kendall correlation coefficient between samples of $w_i$ and $w_j$, respectively.  Denote by $\lambda_1 (A)$ the largest eigenvalue of $A$, for $A=R,S$ and $K$.  We will  consider  $7$  statistics constructed from $R,S$ and $K$. They are defined as follows:

\begin{itemize}
\item[(i)] $\hspace*{5ex} \displaystyle T_1 = \frac{\text{Tr} R^2-a_R }{b_R}$ (see \cite{GHPY});  \vspace{1ex}
  \item[(ii)] $\hspace*{5ex} \displaystyle  T_2= \frac{\text{Tr} S^2-a_S}{b_S}$ (see \cite{BLPZ})  ;\vspace{1ex}
  \item[(iii)] $\hspace*{5ex} \displaystyle T_3 = n\Big(\max_{1\leq i < j \leq p} \left| R_{ij} \right|\Big)^2 -4\log n+ \log\log n$ (see \cite{Jiang});  \vspace{1ex}
    \item[(iv)] $\hspace*{5ex} \displaystyle T_4= n \Big(\max_{1\leq i < j \leq p} \left| \frac{p}{n}S_{ij} \right|\Big)^2 -4\log p+ \log\log p $ (see \cite{Zhou});  \vspace{1ex}
  \item[(v)] $\hspace*{5ex} \displaystyle T_5 =n^{\frac23}c_n^{\frac16} d_{+,c_n}^{-\frac23}(\lambda_1(R)-d_{+,c_n}) $ (see \cite{BPZ1,PY1} ); \vspace{1ex}
 \item[(vi)] $\hspace*{5ex} \displaystyle  T_6=n^{\frac23}c_n^{\frac16} d_{+,c_n}^{-\frac23}(\lambda_1(S)-d_{+,c_n})$ (see \cite{BaoSpearman});
   \vspace{1ex}
   \item[(vii)] $\hspace*{5ex} \displaystyle  T_7= \frac32n^{\frac23}c_n^{\frac16} d_{+,c_n}^{-\frac23}(\lambda_1(K)-\lambda_{+,c_n})$ (see Corollary \ref{cor. TW for K}),  
\end{itemize}
where the parameters $a_R,b_R,a_S$ and $b_S$ will be explained later.  We briefly describe the limiting distributions of the above statistics under the null hypothesis, i.e., $w_1, \ldots, w_p$ are independent.  The limiting null distributions of $T_1$ and $T_2$ are both $N(0,1)$.   The CLT for $T_1$ is derived in \cite{GHPY} under a four moment assumption, and that for $T_2$ is established in \cite{BLPZ} for arbitrary random vector with continuous distribution. We mention that both  \cite{GHPY} and \cite{BLPZ} give CLT of linear eigenvalue statistics for more general test functions. Here we choose the test function  $f(x)=x^2$ for simplicity. The explicit forms of the centering constants $a_R$ and $a_S$ and also those for the scaling constants $b_R$ and $b_S$ can be found in Theorem 3.1 of \cite{GHPY} and Theorem 1.1 of \cite{BLPZ}.
Under a moment condition $\mathbb{E}|w_i|^{30-\varepsilon}<\infty$ with some small constant $\varepsilon>0$, the limiting null distribution of $T_3$ is derived in \cite{Jiang}, and it admits the following c.d.f.:
$
F_{T_3}(x)=\exp(-(c^2\sqrt{8\pi})^{-1}e^{-y/2}).
$
Similarly, the limiting null distribution of $T_4$ (c.f. \cite{Zhou}) is given by 
$
F_{T_4}(x)= \exp(-(8\pi)^{-1/2}e^{-y/2}).
$
Since $T_4$ is nonparametric, the above limiting law does not require moment assumption.  The limiting null distributions of $T_5,T_6,T_7$ are all given by $\textrm{TW}_1$ law. In \cite{BPZ1, PY1}, the $\textrm{TW}_1$ law is established for $R$, assuming that $w_i$'s have sub-exponential tails. Again, since $T_6$ and $T_7$ are constructed from nonparametric matrices, their limiting laws do not require any moment assumption on  $w_i$'s. 

In the sequel, we denote by $\text{Cauchy}(0,1)$ the Cauchy distribution with location parameter $0$ and scale parameter $1$. We further denote by $t(4)$ the student's $t$-distribution with degrees of freedom $4$. We will consider three null hypotheses with the nominal significance level $\alpha=5\% $, for $N(0,1)$, $\text{Cauchy}(0,1)$ and $t(4)$ variables, respectively:
\begin{itemize}
\item $\mathrm{H}_{0,1}$:  $w_i$'s are i.i.d. $N(0,1)$ variables;

\item $\mathrm{H}_{0,2}$: $w_i$'s are i.i.d. $\text{Cauchy}(0,1)$ variables;

\item $\mathrm{H}_{0,3}$: $w_i$'s are i.i.d. $t(4)$ variables.
\end{itemize}

For each null hypothesis $\mathrm{H}_{0,i},i=1,2,3$, we consider two types of alternatives: (i) the alternative of one large disturbance, denoted by $\mathrm{H}_{a,i-1}$; (ii) the alternative of  many small disturbances, denoted by $\mathrm{H}_{a,i-2}$. Specifically, for some parameters $\delta\in (0,1]$ and $\tau_1,\tau_2,\tau_3>0$, we set
\begin{itemize}
\item $\mathrm{H}_{a,1-1}$: $\mathbf{w}\sim N_p(0, I_p+A)$, where $A=(a_{ij})_{p\times p}$ with $a_{ij}=0$ for all $i,j$ except for $a_{12}=a_{21}=\delta$.
\item $\mathrm{H}_{a,1-2}$: $\mathbf{w}\sim N_p(0, I_p+B)$, where $B=(b_{ij})_{p\times p}$ with $b_{ij}= \frac{\tau_1}{p}$ for all $i,j$.
\item $\mathrm{H}_{a,2-1}$: Let $\{x_{i}\}_{i=1}^p$ be i.i.d. $\text{Cauchy}(0,1)$. We set $w_1=x_1+\delta x_2$, $w_2=\delta x_1+x_2$ and $w_i=x_i$ for all $i\neq 1,2$.
\item $\mathrm{H}_{a,2-2}$: Let $\{x_{i}\}_{i=1}^p$ be i.i.d. $\text{Cauchy}(0,1)$. We set $w_i=x_i+\frac{\tau_2}{p}\sum_{j\neq i} x_j$ for all $i$. 
\item $\mathrm{H}_{a, 3-1}$: Let $\{x_{i}\}_{i=1}^p$ be i.i.d. $t(4)$. We set $w_1=x_1+\delta x_2$, $w_2=\delta x_1+x_2$ and $w_i=x_i$ for all $i\neq 1,2$.
\item $\mathrm{H}_{a, 3-2}$: Let $\{x_{i}\}_{i=1}^p$ be i.i.d. $t(4)$. We set $w_i=x_i+\frac{\tau_3}{p}\sum_{j\neq i} x_j$ for all $i$.
\end{itemize}

Here we give more explanation on the above two types of alternatives. Let us take the Gaussian case as an example.  Notice that  $A=\delta(\mathbf{e}_1\mathbf{e}_2^*+\mathbf{e}_2\mathbf{e}_1^*)$ is rank-two and $B=\frac{\tau_1}{p}\mathbf{1}\mathbf{1}'$ is rank-one, where $\mathbf{1}$ represents the all-one vector.  It is easy to see that the two non-zero eigenvalues of $A$ are $\delta$ and $-\delta$, while the nonzero eigenvalue of $B$ is $\tau_1$. Hence,  the population covariance matrix $I_p+A$ (resp. $I_p+B$) has a spike with strength  $1+\delta$ (resp. $1+\tau_1$).  Since the seminal work of Baik, Ben-Arous and P\'{e}ch\'{e} \cite{BBP}, it is now well-known that there is a phase transition called BBP-transition for the largest eigenvalue of the sample covariance  matrix when the population covariance matrix has a spike. Very  roughly speaking,  we can effectively detect the spike using the largest eigenvalue of the sample covariance matrix, only when the spike is larger than the threshold $1+\sqrt{\frac{p}{n}}$. Although here we are considering correlation type of matrices, simulation shows that there is a similar effect. 
Further, although there is no concept of population covariance matrix for $\text{Cauchy}(0,1)$ and $t(4)$ variables, the alternatives $\mathrm{H}_{a,i-1}$ and $\mathrm{H}_{a,i-2}$ for $i=2,3$ are constructed in a similar vein.

The results of sizes and powers stated in Table \ref{results} are obtained under the choices  $p=200, 400,560,800$ with the same $n=600$. The results are based on $1000$ replications. The parameters are chosen to be $\delta=1$, $\tau_1=\tau_3=\frac32$ and $\tau_2=\frac{1}{40}$.  We also refer to Tables \ref{results2} and \ref{results3} in the supplementary material \cite{BaoSupp} for the results under different choices of $p$ and $n$. In addition, we depict the powers for different choices of the parameters $\delta,\tau_1,\tau_2,\tau_3$ in Fig \ref{fig1}-\ref{fig6} in \cite{BaoSupp}, under the setting $(p,n)=(400,600)$.  

\begin{table}[h]
\renewcommand{\arraystretch}{0.95}
\begin{center}
 \begin{tabular}{ccccccccccccccccccc} \hline
$p$  &$T_1$&$T_2$&$T_3$&$T_4$&$T_5$&$T_6$&$T_7$&  &$T_2$ &$T_4$&$T_6$&$T_7$&&$T_2$ &$T_4$ &$T_6$&$T_7$\\ \hline\\
   &\multicolumn{7}{c}{$\mathrm{H}_{0,1}$}&  & \multicolumn{4}{c}{$\mathrm{H}_{0,2}$} & & \multicolumn{4}{c}{$\mathrm{H}_{0,3}$} \\ \cline{2-8} \cline{10-13} \cline{15-18}\\
 200  &3.7    &5.4    &2.6   &3.8   &0.5   &1.5   &1.6&   &4.6   &3.9    &1.5   &1.7&   &6     &3.2  &2.3   &2.3\\
  400 &2.2    &4.9    &2.9   &2.5   &1.8   &3.1   &3.6&   &5.1   &4.3    &2.3   &2.9&   &4.9  &3.2  &1.7   &2.5\\
 560  &1.6    &5.2    &2.9   &3.5   &1.8   &2.2   &2.8&   &4.2   &5       &1.6   &2.5&   &5.5  &5.2   &1.5  &1.8\\
 800  &1.3    &5.4    &4.2   &5      &1.7   &2.1   &2.4&   &5.5   &4.2    &2.1   &2.6&   &4     &3.4   &3.1  &3.6\\ \\
  &\multicolumn{7}{c}{$\mathrm{H}_{a,1-1}$}&  & \multicolumn{4}{c}{$\mathrm{H}_{a,2-1}$} & & \multicolumn{4}{c}{$\mathrm{H}_{a,3-1}$} \\ \cline{2-8} \cline{10-13} \cline{15-18} \\
 200  &87.3  &90.9  &100   &100  &99.4  &99.7 &100 &  &92      &100    &99.9   &100   &&90.6  &100  &99.6  &100\\
  400 &31     &44     &100   &100  &38.1  &40.9 &99.8&  &44.3   &100    &38.8   &99.8  &&42.2  &100  &36.2  &99.8 \\ \\ 
  &\multicolumn{7}{c}{$\mathrm{H}_{a,1-2}$}  && \multicolumn{4}{c}{$\mathrm{H}_{a,2-2}$} & & \multicolumn{4}{c}{$\mathrm{H}_{a,3-2}$} \\ \cline{2-8} \cline{10-13} \cline{15-18}\\
 200  &89.7    &99.5   &5    &5.7    &100  &100   &100&   &93.3   &8.5  &95.5   &95.7   &&100   &6.1   &100   &100\\
400 &39.7     &69.2   &3.5  &3.6   &99.3  &97.1  &97.6&   &97.2  &9.9  &98.7   &98.8   &&91.9  &4.2   &100  &100\\
 560  &18.5   &44.2   &3.3  &4.3   &89.7  &83.3  &84.9&   &98.5  &8.2  &99.5   &99.5   &&70     &5.3   &99.4  &99.5\\
 800  &8.1    &24.7   &3      &4.9   &55.6  &47     &49.1&  &99.6   &8.4  &100    &100    &&47     &3.7   &88.7  &89.7\\ \hline
  \end{tabular}
  \caption{
\footnotesize{The sizes and powers (percentage) of $T_1$ to $T_7$ under different hypotheses and dimension $p$. Here we chose sample size $n=600$, $\delta=1$, $\tau_1=\tau_3=\frac32$ and $\tau_2=\frac{1}{40}$.   \label{results}}
}
\end{center}
\end{table}

Since  $T_1$, $T_3$ and $T_5$ are parametric and the limiting theorems of them in \cite{GHPY, Jiang, BPZ1, PY1} do not apply to the $\text{Cauchy}(0,1)$  and $t(4)$ variables, we omit the simulation results from the tables in these cases. Observe that for the first type of alternatives $\mathrm{H}_{a,i-1}$ for $i=1, 2, 3$, we only consider the case when  $p$ is sufficiently smaller than $n$.  We take $\mathrm{H}_{a,1-1}$ to explain such a choice. In $\mathrm{H}_{a,1-1}$, we consider a Gaussian vector with a population covariance matrix $I_p+A$. On one hand, $\delta$ has to be no larger than $1$ to guarantee the non-negative definiteness of $I_p+A$.  On the other hand, as we mentioned previously, heuristically, due to the BBP transition, one needs $\delta>\sqrt{\frac{p}{n}}$ to get effective information about the existence of $\delta$ from the largest eigenvalue of the sample covariance matrix. Hence, in case that $p$ is close to or larger than $n$, our spike $1+\delta$ would not be large enough to be detected. Simulation shows that a similar effect exists for all three types of correlation matrices considered here.   So we omit the simulation results in those regimes where all the largest eigenvalue statistics will essentially fail. 

Below we summarize our findings from the simulation study.

(1) From Table \ref{results}, and also Table \ref{results2} and Table \ref{results3} in the supplementary material \cite{BaoSupp}, we see that the sizes of $T_2$ are close to the nominal size $5\%$. The sizes of all the other statistics  tend to be smaller than $5\%$. However, for the statistics of the largest eigenvalue  $T_5,T_6$ and $T_7$,  it is possible to modify the centering and scaling constants for the largest eigenvalues to improve the convergence rate of the weak convergence to the $\mathrm{TW}_1$ law such that better sizes can be achieved. Some important works have been done along this line, but only for Gaussian ensembles; see \cite{Karoui1, JM, Ma}. The extension of the results in \cite{Karoui1, JM, Ma} to other random matrix ensembles is still an open question. We do not pursue this direction in the current paper.

(2) From Table \ref{results}, and also Table \ref{results2} and Table \ref{results3} in the supplementary material \cite{BaoSupp}, we see that the statistics of the largest off-diagonal entry, i.e. $T_3, T_4$,  outperform the other statistics in the case of one large disturbance ( $\mathrm{H}_{a, i-1}, i=1,2,3 $). However, $T_3, T_4$ perform quite poorly in the case of many small disturbances ($\mathrm{H}_{a, i-2}, i=1,2,3 $). In general, the other statistics perform well in both types of alternatives. In addition, $T_7$ outperforms the others in most of the cases. For all statistics, the performance deteriorates when $\frac{p}{n}$ increases.  That can be again understood as an effect of the BBP transition. We also refer to Fig \ref{fig1}-\ref{fig6}  in \cite{BaoSupp} for more information about the powers for  different choices of the parameters. 

(3) In the Supplementary material \cite{BaoSupp}, we also consider another type of alternative hypothesis,  denoted by $\mathrm{H}_{a,4}$. For this alternative hypothesis, we consider a random vector $\mathbf{w}$ which has uncorrelated but dependent components. We refer to \cite{BaoSupp} for the detailed definition. The simulation results are stated in Table \ref{results4}. One can see that $T_4$ and $T_7$ outperform the other statistics in general. 

Overall, our statistic $T_7$ has the following advantages. First, it is nonparametric and thus can be used for the heavy-tailed variables, for which $T_1$, $T_3$ and $T_5$ cannot be applied. Second, among all nonparametric statistics  $T_2$, $T_4$, $T_6$ and $T_7$, only $T_2$ performs better than $T_7$ for  the first type of alternatives, but $T_2$ completely fails for  the second type of alternatives.  In a nutshell, $T_7$ is the most robust among all $7$ statistics 
for the cases considered in this simulation study.

\section{Hoeffding decomposition and large deviation} \label{s.large deviation} In this section, we state some key  large deviation estimates; see Propositions \ref{lem. large deviation linear part} and \ref{lem. large deviation nonlinear part}.  We start with (a variant of) Hoeffding decomposition for $v_{k,(ij)}$'s. 
\subsection{Hoeffding decomposition} Let
\begin{align}
v_{k,(i\cdot)}:=\mathbb{E}\big( \text{sign}(w_{ki}-w_{kj})|w_{ki}\big), \quad v_{k,(\cdot j)}:=\mathbb{E}\big( \text{sign}(w_{ki}-w_{kj})|w_{kj}\big).  \label{170920100}
\end{align} 
Observe that 
$
v_{k,(\cdot i)}=-v_{k, (i \cdot)}.
$
The following decomposition is (a variant of) Hoeffding decomposition 
\begin{align}
v_{k,(ij)}=v_{k,(i\cdot)}-v_{k, (j\cdot)}+\bar{v}_{k,(ij)}, \label{17101820}
\end{align}
where we take (\ref{17101820}) as the definition of $\bar{v}_{k,(ij)}$. 
 It is easy to check that the three parts in the RHS  are  pairwise uncorrelated. In addition, all of the three parts in the RHS of  (\ref{17101820}) are with mean 0 and variance $\frac13$, i.e.,
 \begin{align}
 \mathbb{E}v_{k,(i\cdot)}=\mathbb{E} v_{k, (j\cdot)}=\mathbb{E} \bar{v}_{k,(ij)}=0, \quad    \mathbb{E}v_{k,(i\cdot)}^2=\mathbb{E} v_{k, (j\cdot)}^2=\mathbb{E} \bar{v}_{k,(ij)}^2=\frac13.  \label{171128201}
 \end{align}  
  For brevity, we further introduce the notation 
\begin{align}
u_{k,(ij)}:= v_{k,(i\cdot)}-v_{k, (j\cdot)}.  \label{171128200}
\end{align}
Hence, we can also write $v_{k,(ij)}=u_{k,(ij)}+\bar{v}_{k,(ij)}$.

For a fixed $k\in \llbracket 1, p\rrbracket$, let $F_k$ be the common distribution of all $w_{ki}, i\in \llbracket 1, n\rrbracket$. We see that 
\begin{align}
v_{k,(i\cdot)}=\mathbb{E}(\mathbbm{1}(w_{kj}\leq w_{ki})|w_{ki})-\mathbb{E}(\mathbbm{1}(w_{kj}> w_{ki})|w_{ki})=2F_k(w_{ki})-1,  \label{17092101}
\end{align}
which is uniformly distributed on $[-1,1]$.  Hence, all $v_{k, (i\cdot)}, (k,i)\in \llbracket 1, p\rrbracket \times \llbracket 1, n\rrbracket$ are i.i.d., uniform random variables on $[-1,1]$, in light of (\ref{17092101}) and the independence of  $w_{ki}$'s. We will  call $v_{k,(i\cdot)}$ and $v_{k,(j\cdot)}$ (or together $u_{k,(ij)}$) the {\it linear} parts of $v_{k,(ij)}$, and  call $\bar{v}_{k,(ij)}$ the {\it nonlinear} part. Although the linear parts in all $v_{k,(ij)}$'s have a simple dependence structure due to the independence between $v_{k,(i\cdot)}$'s, the nonlinear parts couple $v_{k,(ij)}$'s together with certain nontrivial dependence relation.  For instance, $v_{k,(ij)}$ and $v_{k,(i\ell)}$ are correlated even when $j\neq \ell$. More specifically, it is elementary to check  
\begin{align}
\mathbb{E}v_{k,(ij)}v_{k,(i\ell)}= \mathbb{E}(v_{k,(i\cdot)})^2=\frac13. \label{17092150}
\end{align}

In the sequel, we will often separate the nonlinear part from the linear part. To this end, we introduce the following notations.  We set the $M$-dimensional row vector 
\begin{align}
\mb{v}_k:= \frac{1}{\sqrt{M}} (v_{k,(ij)})_{i<j}\equiv \frac{1}{\sqrt{M}}\big(v_{k,(12)},\ldots, v_{k, (1n)}, v_{k,(23)},\ldots, v_{k,(2n)}  \ldots,  v_{k,(n-1,n)}\big). \label{17091801}
\end{align}
Further, we set 
\begin{align}
\mb{u}_k:= \frac{1}{\sqrt{M}}(u_{k,(ij)})_{i<j},\qquad 
\bar{\mb{v}}_k:=\frac{1}{\sqrt{M}}(\bar{v}_{k,(ij)})_{i<j}. \label{17101010}
\end{align}
With the above notations, we can write 
\begin{align}
\mb{v}_k=\mb{u}_k+\bar{\mb{v}}_k,\qquad k\in \llbracket 1, p\rrbracket.   \label{171128230}
\end{align}
Note that under the null hypothesis, i.e., the components of the population vector $\mb{w}$ are independent, the random vectors $\mb{v}_1, \ldots, \mb{v}_p$ are also independent. But the components in $\mb{v}_k$ are dependent, as mentioned above (c.f. (\ref{17092150})).  We also notice that  $\mb{v}_i$ is the $i$-th row of $\Theta$ defined in (\ref{def of Theta}).  For the columns of $\Theta$, i.e.,  $\bs{\theta}_{(ij)}$'s in (\ref{17120201}), we also introduce the notations 
\begin{align*}
\bs{\theta}_{(i\cdot)}:=  \frac{1}{\sqrt{M}} (v_{1,(i\cdot)}, \ldots, v_{p,(i\cdot)})', \qquad \bar{\bs{\theta}}_{(ij)}:= \frac{1}{\sqrt{M}} (\bar{v}_{1,(ij)}, \ldots, \bar{v}_{p,(ij)})'.
\end{align*} 
Hence, we have the decomposition  for columns
\begin{align}
\bs{\theta}_{(ij)}= \bs{\theta}_{(i\cdot)}-\bs{\theta}_{(j\cdot)}+\bar{\bs{\theta}}_{(ij)}.  \label{17120205}
\end{align}
Further
note that the nonzero eigenvalues of the  matrix $K$ are the same as those of  the following  $M\times M$ matrix
\begin{align}
\mathcal{K}:=\sum_{i=1}^p \mb{v}_k'\mb{v}_k=\Theta'\Theta.  \label{17113001}
\end{align}

\subsection{Large deviation estimates for $\mathbf{v}_k$}
Set the $M\times M$ symmetric matrix
\begin{align}
\Gamma=( \chi_{(ij)(st)} )_{i<j,s<t}, \label{17102201}
\end{align} 
where $(ij)$ is the row index and $(st)$ is the column index and
\begin{align*}
\chi_{(ij)(st)}:=\frac13 \big( \delta_{is}+\delta_{jt}- \delta_{it}-\delta_{js}\big). 
\end{align*}
It is elementary to check that 
\begin{align}
\Gamma^2=\frac{n}{3}\Gamma. \label{17092601}
\end{align}
Consequently, we have  the fact 
\begin{align}
\|\Gamma\|=O(n).  \label{17110901}
\end{align}
We further set the $n\times M$ matrix 
\begin{align}
T=(t_{\ell,(ij)})_{\ell, i<j}, \qquad  t_{\ell,(ij)}:=\delta_{\ell i}-\delta_{\ell j}, \quad 1\leq \ell\leq n,  1\leq i<j\leq n,  \label{17112630}
\end{align} 
where $\ell$ is the row index and $(ij)$ is the column index. 
It is easy to check 
\begin{align}
\Gamma=\frac13T'T. \label{17120301}
\end{align}

The first proposition is on the large deviation estimates for some linear and quadratic forms of $\mb{u}_k$. 
\begin{pro} \label{lem. large deviation linear part}Let $\mb{u}_k$ be defined as in (\ref{17101010}). Let $\mathbf{a}=(a_{(ij)})_{i<j}\in \mathbb{C}^M$ be any deterministic vector, and let $B:=(b_{(ij),(s t)})_{{i<j, s<t}}\in \mathbb{C}^{M\times M}$ be any deterministic matrix. We have
\begin{align}
&\mathbb{E} \mb{u}_k B\mb{u}_k'=\frac{1}{M}\mathrm{Tr} B\Gamma, \label{17092110}\\
&\big|\mb{u}_k\mathbf{a}'\big|
\prec \sqrt{\frac{\mathbf{a}\Gamma\mathbf{a}^*}{M}}\prec \sqrt{\frac{\|\mb{a}\|^2}{n}}, \label{17092170}\\
&\Big|\mb{u}_k B\mb{u}_k'-\frac{1}{M}\mathrm{Tr} B\Gamma\Big|\prec  \sqrt{\frac{\mathrm{Tr} |B\Gamma|^2}{M^2}}.\label{17092130}
\end{align}
\end{pro}
The second proposition is about the large deviation estimates for some linear and quadratic forms of $\bar{\mb{v}}_k$ and the crossing quadratic forms of $\bar{\mb{v}}_k$ and $\mb{u}_k$. 
\begin{pro} \label{lem. large deviation nonlinear part} Let $\mb{u}_k$ and $\bar{\mb{v}}_k$ be as defined  in (\ref{17101010}). Let $\mathbf{a}=(a_{(ij)})_{i<j}\in \mathbb{C}^M$ be any deterministic vector, and let $B:=(b_{(ij),(st)})_{i<j,s<t}\in \mathbb{C}^{M\times M}$ be any  deterministic matrix. We have 
\begin{align}
&\big|\bar{\mb{v}}_k\mathbf{a}'\big|\prec \sqrt{\frac{\|\mb{a}\|^2}{M}}, \label{17101050}\\
&\Big|\mb{u}_k B\bar{\mb{v}}_k'\Big|\prec \sqrt{\frac{n}{M^2} \mr{Tr} |B|^2}+\sqrt{\frac{1}{M^2}\sum_{\ell=1}^n \Big|\sum_{j=\ell+1}^n (TB)_{j,(\ell j)}\Big|^2},   \label{17101051}\\
& \Big|\bar{\mb{v}}_k B\bar{\mb{v}}_k'-\frac{1}{3M}\mathrm{Tr} B\Big|\prec \sqrt{\frac{n}{M^2} \mr{Tr} |B|^2}. \label{17101052}
\end{align}
\end{pro}
We further set
\begin{align}
\wt{\Gamma}= \Gamma+\frac13I_M. \label{17113050}
\end{align}
From Propositions \ref{lem. large deviation linear part} and \ref{lem. large deviation nonlinear part}, we can easily get the following corollary. 
\begin{cor} \label{cor. large devi} Let $\mb{v}_k$ be as defined in (\ref{17091801}). Let $\mathbf{a}=(a_{(ij)})_{i<j}\in \mathbb{C}^M$ be any  deterministic vector, and let $B:=(b_{(ij),(st)})_{i<j,s<t}\in \mathbb{C}^{M\times M}$ be any  deterministic matrix. We have 
\begin{align}
&\big|\mb{v}_k\mathbf{a}'\big|
\prec \sqrt{\frac{\mathbf{a}\Gamma\mathbf{a}'}{M}}\prec \sqrt{\frac{\|\mb{a}\|^2}{n}},  \label{17113070}\\
&\Big|\mb{v}_k B\mb{v}_k'-\frac{1}{M}\mathrm{Tr} B\wt{\Gamma}\Big|\prec  \sqrt{\frac{\mathrm{Tr} |B|^2}{M}}.\label{17113010}
\end{align}
\end{cor}
The proofs of Propositions \ref{lem. large deviation linear part} and \ref{lem. large deviation nonlinear part} and also the proof of Corollary \ref{cor. large devi} are stated in the supplementary material \cite{BaoSupp}. 

\section{Strong local law for $K$} \label{s.local law}
{ In this section, we state a strong local law for the matrix $K$; see Proposition \ref{lem.weak law for R}. The  proof of Proposition \ref{lem.weak law for R} is 
stated in the supplementary material \cite{BaoSupp} and it heavily relies on the large deviation bounds in Corollary \ref{cor. large devi}.  To state the results, we need more notations. Recall the matrices $K$ and $\mc{K}$ defined in (\ref{17101901}) and (\ref{17113001}). 
We  denote the Green functions of $K$ and $\mathcal{K}$ by
\begin{align*}
G(z):=(K-z)^{-1}, \qquad \mathcal{G}(z):=(\mathcal{K}-z)^{-1}.  
\end{align*} 
Then, we further denote the Stieltjes transform of $K$ by 
\begin{align*}
{m} (z):=\frac{1}{p} \mr{Tr} G(z)=\frac{1}{p} \sum_{i=1}^p G_{ii}(z). 
\end{align*}

 For any $z=E+\ii\eta\in \mathbb{C}^+$, we  set the function $\underline{m} (z):\mathbb{C}^+\to \mathbb{C}^+$ as the solution to the equation 
\begin{align}
\frac23 c_n(z-\frac13)(\underline{m} (z))^2+(z-1+\frac23 c_n)\underline{m} (z)+1=0. \label{17113014}
\end{align}
It is elementary to check that $\underline{m}$ is the Stieltjes transform of $F^K_{c_n}$ (c.f. Theorem \ref{thm. global law of K}). Some properties of the function $\underline{m}$ are given in Lemma \ref{lem.properties of m}.

We then introduce the following notations
\begin{align}
&{\Lambda}_{\rm d}\equiv \Lambda_{\rm d}(z):=\max_{k}|G_{kk}(z)-\underline{m}(z)|, \qquad {\Lambda}_{\rm o}\equiv {\Lambda}_{\rm o}(z):=\max_{k\neq \ell} |G_{k\ell}(z)|, \nonumber\\ &{\Lambda}\equiv {\Lambda}(z):=|{m}(z) -\underline{m}(z)|.  \label{17102910}
\end{align}
In the sequel, we work in  the following  domain of $z$
\begin{align}
{\mathcal{D}}(\epsilon):=\big\{z=E+\ii \eta: \frac{1}{2}{\lambda}_{+,c}\leq E\leq 2{\lambda}_{+,c}, n^{-1+\epsilon}\leq \eta\leq 1 \big\}, \label{18080701}
\end{align}
where  ${\lambda}_{+,c}$ is defined in (\ref{17110101}).
 Let $\gamma_1\geq \gamma_2\geq\cdots\geq \gamma_{p\wedge n}$ be the ordered $p$-quantiles of $F_{c_n}^K$, i.e., $\gamma_j$ is the smallest real number such that  
\begin{align*}
\int_{-\infty}^{\gamma_j} {\rm d} F_{c_n}^K(x)=\frac{p-j+1}{p}, \qquad j\in \llbracket 1, n\wedge p\rrbracket.  
\end{align*}
We further define the deterministic control parameter
\begin{align}
\Psi\equiv \Psi(z):= \sqrt{\frac{\Im \underline{m}(z)}{n\eta}}+\frac{1}{n\eta}.  \label{18080702}
\end{align}
With the above notations, we can now state the following strong local law. 
\begin{pro} \label{lem.weak law for R}  Under the assumption (\ref{assump on dim}),  the following hold:

(i): (Entrywise local law) The following bounds hold uniformly on ${\mathcal{D}}(\epsilon)$
\begin{align}
{\Lambda}_{ \rm d}(z)\prec \Psi(z), \qquad {\Lambda}_{ \rm o}(z)\prec \Psi(z).  \label{17111705}
\end{align}

(ii): (Strong local law) The following bound holds uniformly on ${\mathcal{D}}(\epsilon)$
\begin{align}
{\Lambda} (z)\prec\frac{1}{n\eta}. \label{17111663}
\end{align}

(iii): (Rigidity on the right edge). For $i\in [1, \delta p]$ with any sufficiently small constant $\delta \in (0,1)$, we have 
\begin{align}
|\lambda_{i}(K)-{\gamma}_i|\prec n^{-\frac{2}{3}}i^{-\frac13}.  \label{17111701}
\end{align}
\end{pro}

\section{Decoupling} \label{s. decoupling}
In this section, we compare the Green functions of the matrix $K$ with another random matrix $\wh{K}$ which has independent linear part and ``nonlinear" part (c.f. (\ref{17110202})).  Recall  (\ref{170920100}).  We set the matrices
\begin{align}
U:=\frac{1}{\sqrt{M}}\big((v_{k,(i\cdot)}-v_{k,(j,\cdot)})\big)_{k,(ij)}, \qquad \bar{V}:=  \frac{1}{\sqrt{M}}\big(\bar{v}_{k,(ij)}\big)_{k,(ij)}\label{171130101}
\end{align} 
and let 
\begin{align*}
H:=\frac{1}{\sqrt{M}} \big( h_{k,(ij)}\big)_{k,(ij)}, \qquad k\in \llbracket 1, p\rrbracket,\qquad 1\leq i<j\leq n 
\end{align*}
be a $p\times M$ matrix, where the entries $h_{k,(ij)}$'s are i.i.d. $N(0, \frac13)$. We also set the random variables 
$
h_{k,(ij)}:=-h_{k,(ji)} 
$
if $i\geq j$, for further use. 
We  assume that $H$ is independent of $U$.  
We define the random matrices 
\begin{align}
\wh{\Theta}:=(U+H),\qquad  \wh{K}:=  \wh{\Theta}\wh{\Theta}'=(U+H)(U+H)'. \label{17110202}
\end{align}
Then we denote the Green function of $\wh{K}$ and its  normalized trace by 
\begin{align*}
\wh{G}(z):=(\wh{K}-z)^{-1}, \qquad \wh{m}(z):=\frac{1}{p} \mathrm{Tr} \wh{G}(z)
\end{align*}

In this section, we will establish the following comparison proposition. 
\begin{pro} \label{lem.17100801}Let $\varepsilon>0$ be any sufficiently small constant. Set  $\eta=n^{-\frac23-\varepsilon}$. Let $E_1, E_2\in \mathbb{R}$ satisfy $E_1<E_2$ and 
\begin{align}
|E_1|, |E_2|\leq n^{-\frac23+\varepsilon}.  \label{17100940}
\end{align}
Let $F:\mathbb{R}\to \mathbb{R}$ be a smooth function satisfying 
$
\max_{x\in \mathbb{R}} |F^{(\ell)}(x)|(|x|+1)^{-C}\leq C,  \ell=1,2,3,4,
$
for some positive constant $C$.  Then, there exists a constant $\delta>0$ such that, for sufficiently large $n$ we have 
\begin{align*}
&\Big| \mathbb{E} F\Big( n \int_{E_1}^{E_2} \Im m (x+{\lambda}_{+,c_n}+\mathrm{i}\eta){\rm d} x\Big)\nonumber\\
&\qquad\qquad-\mathbb{E} F\Big( n \int_{E_1}^{E_2} \Im \wh{m} (x+{\lambda}_{+,c_n}+\mathrm{i}\eta){\rm d} x\Big)\Big|\leq n^{-\delta}. 
\end{align*}
\end{pro}

\begin{proof}[Proof of Proposition \ref{lem.17100801}]   For simplicity, in this proof, we denote by 
\begin{align}
z\equiv z(x):= x+{\lambda}_{+,c_n}+\mathrm{i}\eta, \qquad x\in [E_1, E_2].  \label{17110234}
\end{align}
Recall the small constant $\varepsilon$ in Proposition \ref{lem.17100801}. For brevity, we will simply write $C\varepsilon$ with any positive constant (independent of $\varepsilon$)  by $\varepsilon$ in the sequel. In other words, we allow $\varepsilon$ to vary from line to line, up to  $C$.  We then construct the following sequence of the interpolations: 
$
\Theta= \Theta_0,  \ldots, \Theta_{\gamma-1}, \; \Theta_{\gamma} \ldots, \; \Theta_p=\wh{\Theta},
$
where $\Theta_\gamma$ is the matrix whose first $\gamma$ rows are the same as those of  $\wh{\Theta}$ and the remaining $p-\gamma$ rows are the same as those of  $\Theta$.  Correspondingly, we set the notations
\begin{align*}
K_\gamma= \Theta_\gamma\Theta_\gamma', \qquad G_\gamma(z):=(K_\gamma-z)^{-1}, \qquad m_\gamma:= \frac{1}{p} \mr{Tr} G_\gamma(z). 
\end{align*}
We first claim the following lemma, whose proof is stated  in the supplementary material \cite{BaoSupp}.  
\begin{lem}[Local law for $K_\gamma$] \label{lem. local law for K gamma}
All the estimates in Proposition  \ref{lem.weak law for R} hold for $K_\gamma$ for all $\gamma\in \llbracket 0, p \rrbracket$.
\end{lem}
With Lemma \ref{lem. local law for K gamma}, we proceed to the proof of Proposition \ref{lem.17100801}. 
Using the above notations, we can write  
\begin{align*}
&\mathbb{E} F\Big( n \int_{E_1}^{E_2} \Im m (z){\rm d} x\Big)-\mathbb{E} F\Big( n \int_{E_1}^{E_2} \Im \wh{m} (z){\rm d} x\Big)\nonumber\\
&=  \mathbb{E} F\Big( n \int_{E_1}^{E_2} \Im m_0 (z){\rm d} x\Big)-\mathbb{E} F\Big( n \int_{E_1}^{E_2} \Im {m}_p (z){\rm d} x\Big)\nonumber\\
&= \sum_{\gamma=1}^p \Big( \mathbb{E} F\Big( n \int_{E_1}^{E_2} \Im m_{\gamma-1} (z){\rm d} x\Big) - \mathbb{E} F\Big( n \int_{E_1}^{E_2} \Im m_\gamma (z){\rm d} x\Big)\Big).
\end{align*}
Hence, it suffices to show that for all $\gamma\in \llbracket 1, p\rrbracket$, 
\begin{align}
\bigg| \mathbb{E} F\Big( n \int_{E_1}^{E_2} \Im m_{\gamma-1} (z){\rm d} x\Big) - \mathbb{E} F\Big( n \int_{E_1}^{E_2} \Im m_\gamma (z){\rm d} x\Big)\bigg|\leq n^{-1-\delta}  \label{17100920}
\end{align}
for some positive constant $\delta$. 
For a fixed $\gamma$, we further introduce the notation $\Theta_\gamma^{(i)}$ to denote the matrix obtained from $\Theta_\gamma$ with the $i$-th row removed. Then, by definition, we have 
$\Theta_{\gamma-1}^{(\gamma)}=\Theta_\gamma^{(\gamma)}$. Correspondingly, we use the notations
\begin{align*}
K_\gamma^{(i)}:=\Theta_\gamma^{(i)} (\Theta_\gamma^{(i)})', \qquad G_\gamma^{(i)}:= (K_\gamma^{(i)}-z)^{-1}, \qquad m_\gamma^{(i)}:=\frac{1}{p} \mr{Tr} G_\gamma^{(i)}. 
\end{align*}
Also note that $m_{\gamma-1}^{(\gamma)}=m_\gamma^{(\gamma)}$.  Next, we expand both  $m_{\gamma-1}$ and $m_\gamma$  around $m_\gamma^{(\gamma)}$.
Observe that 
\begin{align}
m_{\gamma-1}-m_{\gamma}^{(\gamma)}= \frac{1}{p} \frac{1+\mb{v}_\gamma (\Theta_\gamma^{(\gamma)})' (G_\gamma^{(\gamma)})^2\Theta_\gamma^{(\gamma)} \mb{v}_\gamma'}{\mb{v}_\gamma\mb{v}_\gamma'-z-\mb{v}_\gamma (\Theta_\gamma^{(\gamma)})' G_\gamma^{(\gamma)}\Theta_\gamma^{(\gamma)} \mb{v}_\gamma'}=:  \frac{1}{p} \frac{1+\mb{v}_\gamma A_\gamma \mb{v}_\gamma'}{1-z-\mb{v}_\gamma B_\gamma \mb{v}_\gamma'}, \label{17110230}
\end{align}
where in the last step we use the trivial fact $\mb{v}_\gamma\mb{v}_\gamma'=1$. 
 Similarly,
\begin{align}
m_{\gamma}-m_{\gamma}^{(\gamma)}= \frac{1}{p} \frac{1+\hat{\mb{v}}_\gamma A_\gamma \hat{\mb{v}}_\gamma'}{\hat{\mb{v}}_\gamma\hat{\mb{v}}_\gamma'-z-\hat{\mb{v}}_\gamma B_\gamma \hat{\mb{v}}_\gamma'}, \label{17100923}
\end{align}
where we use the notation $\hat{\mb{v}}_\gamma:=\mb{u}_\gamma+\mb{h}_\gamma$ to denote the $\gamma$-th row of $\wh{\Theta}$.

We then further set 
\begin{align}
D_\gamma:= \mb{v}_\gamma B_\gamma \mb{v}_\gamma'- \frac{1}{M} \mr{Tr} B_\gamma{\Gamma}, \quad \wh{D}_\gamma:= 1-\hat{\mb{v}}_\gamma\hat{\mb{v}}_\gamma'+\hat{\mb{v}}_\gamma B_\gamma \hat{\mb{v}}_\gamma'-\frac{1}{M} \mr{Tr} B_\gamma{\Gamma}, \label{17100924}
\end{align} 
and write
\begin{align}
&D_\gamma=\Big(\mb{u}_\gamma B_\gamma \mb{u}_\gamma'- \frac{1}{M} \mr{Tr} B_\gamma{\Gamma}\Big)+\bar{\mb{v}}_\gamma B_\gamma \bar{\mb{v}}'_\gamma+ 2\mb{u}_\gamma B_\gamma \bar{\mb{v}}_\gamma'=: \mc{U}_\gamma+\mc{V}_\gamma+\mc{P}_\gamma.\nonumber\\
&\wh{D}_\gamma= \Big({\mb{u}}_\gamma B_\gamma {\mb{u}}_\gamma'-\frac{1}{M} \mr{Tr} B_\gamma{\Gamma}\Big)+{\mb{h}}_\gamma B_\gamma {\mb{h}}_\gamma'+ 2\mb{u}_\gamma B_\gamma {\mb{h}}_\gamma'\nonumber\\
&\qquad\quad+\Big(\frac23-{\mb{u}}_\gamma{\mb{u}}_\gamma'\Big)+ \Big(\frac13-{\mb{h}}_\gamma{\mb{h}}_\gamma'-2{\mb{u}}_\gamma{\mb{h}}_\gamma'\Big)\nonumber\\
&\qquad =: \mc{U}_\gamma+\hat{\mc{V}}_\gamma+\hat{\mc{P}}_\gamma+\hat{\mc{W}}_\gamma+\hat{\mc{O}}_\gamma, \label{171102101}
\end{align}
where we recall that $B_\gamma$ is (complex) symmetric. Similarly, we write 
\begin{align}
{\mb{v}}_\gamma A_\gamma {\mb{v}}_\gamma' &= {\mb{u}}_\gamma A_\gamma {\mb{u}}_\gamma'+ \bar{\mb{v}}_\gamma A_\gamma \bar{\mb{v}}_\gamma'+2{\mb{u}}_\gamma A_\gamma \bar{\mb{v}}_\gamma' =: {\mb{u}}_\gamma A_\gamma {\mb{u}}_\gamma'+ \bar{\mb{v}}_\gamma A_\gamma \bar{\mb{v}}_\gamma'+{\mc{Q}}_\gamma, \nonumber\\
\hat{\mb{v}}_\gamma A_\gamma \hat{\mb{v}}_\gamma' &= {\mb{u}}_\gamma A_\gamma {\mb{u}}_\gamma'+ {\mb{h}}_\gamma A_\gamma {\mb{h}}_\gamma'+ 2 {\mb{u}}_\gamma A_\gamma {\mb{h}}_\gamma' =: {\mb{u}}_\gamma A_\gamma {\mb{u}}_\gamma'+ {\mb{h}}_\gamma A_\gamma {\mb{h}}_\gamma'+\hat{\mc{Q}}_\gamma. \label{171128100}
\end{align}
We have the following crucial technical lemma.
\begin{lem}\label{lem. technical estimates}Let $\eta=n^{-\frac23-\varepsilon}$, and $x,x_1,x_2\in[E_1,E_2]$, where $E_1$ and $E_2$ satisfy (\ref{17100940}).  Let $z=x+\lambda_{+,c_n}+\ii \eta$ and $z_a=x_a+\lambda_{+,c_n}+\ii \eta, a=1,2$. With the above notations, we have 
\begin{align}
&|\mc{U}_\gamma(z)|\prec n^{-\frac13+\varepsilon}, \quad |\mc{V}_\gamma(z)|\prec n^{-\frac56+\varepsilon},   \quad |\hat{\mc{P}}_\gamma(z)|\prec n^{-\frac56+\varepsilon}, \nonumber\\
 &  |\hat{\mc{V}}_\gamma(z)|\prec n^{-1+\varepsilon},\quad |\hat{\mc{W}}_\gamma(z)|\prec n^{-\frac12+\varepsilon}, \quad |\hat{\mc{O}}_\gamma(z)|\prec n^{-1+\varepsilon},  \quad |\hat{\mc{Q}}_\gamma(z)|\prec n^{-\frac16+\varepsilon}, \nonumber\\
 &    |{\mb{u}}_\gamma A_\gamma(z) {\mb{u}}_\gamma'|\prec n^{\frac13+\varepsilon}, \quad |\bar{\mb{v}}_\gamma A_\gamma(z) \bar{\mb{v}}_\gamma'|\prec n^{-\frac16+\varepsilon}, \quad  |{\mb{h}}_\gamma A_\gamma(z) {\mb{h}}_\gamma'|\prec  n^{-\frac12+\varepsilon},
 \label{17112601}
 \end{align}
 and 
 \begin{align}
 |\mc{P}_\gamma(z)|\prec n^{-\frac12+\varepsilon},\quad |\mc{Q}_\gamma(z)|\prec n^{-\frac16+\varepsilon} \label{171201100}
\end{align}
In addition, we have 
\begin{align}
&  \big|\mathbb{E}\big(\mb{u}_\gamma A_\gamma (z) \mb{u}_\gamma'\hat{\mc{W}}_\gamma\big) \big|\prec n^{-\frac23+\varepsilon}, \quad \big|\mathbb{E}\big( {\mb{u}}_\gamma A_\gamma(z_1) {\mb{u}}_\gamma' \mc{P}_\gamma(z_2)\big)\big|\prec  n^{-\frac12+\varepsilon}.  \label{17112610}
\end{align} 
The above estimates still hold if we replace some or all of $z, z_1,z_2$ by their complex conjugates. 
\end{lem}
The  proof of Lemma \ref{lem. technical estimates} will be stated in the supplementary material \cite{BaoSupp}. Two key technical inputs for the proof are Propositions \ref{lem. large deviation linear part} and \ref{lem. large deviation nonlinear part}.

We proceed to the proof of  Proposition \ref{lem.17100801}, with the aid of Lemma \ref{lem. technical estimates}. 
First, using (\ref{17112601}) and (\ref{171201100}), we can write
\begin{align}
 n \int_{E_1}^{E_2} \big(m_{\gamma-1}(z)-m_{\gamma}^{(\gamma)}(z)\big){\rm d} x &= \frac{n}{p} \int_{E_1}^{E_2}  \frac{1+\mb{v}_\gamma A_\gamma \mb{v}_\gamma'}{1-z- \frac{1}{M} \mr{Tr} B_\gamma{\Gamma}-D_\gamma} {\rm d} x \nonumber\\
 &=\tau_{\gamma 0}+\tau_{\gamma 1}+\tau_{\gamma 2}+O_\prec(n^{-\frac76+\varepsilon}),\label{17100930}
\end{align}
where
\begin{align}
&\tau_{\gamma 0}:= \frac{n}{p}\int_{E_1}^{E_2} \frac{1+\mb{v}_\gamma A_\gamma \mb{v}_\gamma'}{(1-z- \frac{1}{M} \mr{Tr} B_\gamma{\Gamma})}  {\rm d} x=O_\prec(n^{-\frac13+\varepsilon}), \nonumber\\
&  \tau_{\gamma 1}:=  \frac{n}{p}\int_{E_1}^{E_2} \frac{1+\mb{u}_\gamma A_\gamma \mb{u}_\gamma'}{(1-z- \frac{1}{M} \mr{Tr} B_\gamma{\Gamma})^{2}} \big(\mc{U}_\gamma+\mc{P}_\gamma \big){\rm d} x=O_\prec(n^{-\frac23+\varepsilon}),\nonumber\\
& \tau_{\gamma 2}:=  \frac{n}{p}\int_{E_1}^{E_2} \frac{1+\mb{u}_\gamma A_\gamma \mb{u}_\gamma'}{(1-z- \frac{1}{M} \mr{Tr} B_\gamma{\Gamma})^{3}} \mc{U}_\gamma^2 {\rm d} x=O_\prec(n^{-1+\varepsilon}). \label{17110310}
\end{align}
Here we  use the fact $1/(1-z- \frac{1}{M} \mr{Tr} B_\gamma{\Gamma})\sim 1$ with high probability, which follows from  $1/(1-z- \frac{1}{M} \mr{Tr} B_\gamma{\Gamma})=\underline{m}+O_\prec(\frac{1}{n\eta})$ (c.f. Lemma \ref{lem. local law for K gamma} and an analogue of (\ref{171116112} )), and also  (\ref{17113020}). Analogously, we have 
\begin{align}
 n \int_{E_1}^{E_2} \big(m_{\gamma}(z)-m_{\gamma}^{(\gamma)}(z)\big){\rm d} x &= \frac{n}{p} \int_{E_1}^{E_2}  \frac{1+\hat{\mb{v}}_\gamma A_\gamma \hat{\mb{v}}_\gamma'}{1-z- \frac{1}{M} \mr{Tr} B_\gamma{\Gamma}-\wh{D}_\gamma} {\rm d} x \nonumber\\
 &=\hat{\tau}_{\gamma 0}+\hat{\tau}_{\gamma 1}+{\tau}_{\gamma 2}+O_\prec(n^{-\frac76+\varepsilon}),
\end{align}
where
\begin{align}
&\hat{\tau}_{\gamma 0}:= \frac{n}{p}\int_{E_1}^{E_2} \frac{1+\hat{\mb{v}}_\gamma A_\gamma \hat{\mb{v}}_\gamma'}{(1-z- \frac{1}{M} \mr{Tr} B_\gamma{\Gamma})}  {\rm d} x=O_\prec(n^{-\frac13+\varepsilon}), \nonumber\\
&  \hat{\tau}_{\gamma 1}:=  \frac{n}{p}\int_{E_1}^{E_2} \frac{1+\mb{u}_\gamma A_\gamma \mb{u}_\gamma'}{(1-z- \frac{1}{M} \mr{Tr} B_\gamma{\Gamma})^{2}} (\mc{U}_\gamma+\hat{\mc{W}}_\gamma) {\rm d} x=O_\prec(n^{-\frac23+\varepsilon}).
\label{17110311}
\end{align}

For brevity, we further introduce the notation
$
\zeta_\gamma:= n \int_{E_1}^{E_2} \Im m_{\gamma}^{(\gamma)} (z){\rm d} x. 
$
 Then we can write 
\begin{align*}
&F\Big( n \int_{E_1}^{E_2} \Im m_{\gamma-1} (z){\rm d} x\Big)= F(\zeta_\gamma)
+ F'(\zeta_\gamma) (\Im \tau_{\gamma 0}+\Im \tau_{\gamma 1}+\Im \tau_{\gamma 2})\nonumber\\
&\qquad+\frac{F^{(2)}(\zeta_\gamma)}{2}\big((\Im \tau_{\gamma 0})^2+ 2\Im \tau_{\gamma 0} \Im \tau_{\gamma 1}\big)+\frac{F^{(3)}(\zeta_\gamma)}{6} (\Im \tau_{\gamma 0})^3+O_\prec(n^{-\frac76+\varepsilon}). 
\end{align*}
Analogously, we have 
\begin{align*}
&F\Big( n \int_{E_1}^{E_2} \Im m_{\gamma} (z){\rm d} x\Big)= F(\zeta_\gamma)
+ F'(\zeta_\gamma) (\Im \hat{\tau}_{\gamma 0}+\Im \hat{\tau}_{\gamma 1}+\Im {\tau}_{\gamma 2})\nonumber\\
&\qquad+\frac{F^{(2)}(\zeta_\gamma)}{2}\big((\Im \hat{\tau}_{\gamma0 })^2+ 2\Im \hat{\tau}_{\gamma 0} \Im \hat{\tau}_{\gamma 1}\big)+\frac{F^{(3)}(\zeta_\gamma)}{6} (\Im \hat{\tau}_{\gamma 0})^3+O_\prec(n^{-\frac76+\varepsilon}). 
\end{align*}
Therefore, to establish (\ref{17100920}), it suffices to show the following
\begin{align}
&\mathbb{E} \Im {\tau}_{\gamma a}-\mathbb{E}\Im \hat{\tau}_{\gamma a}=O_\prec(n^{-1-\delta}), \qquad a=0,1\label{17110301}\\
&\mathbb{E} (\Im {\tau}_{\gamma 0})^2-\mathbb{E} (\Im \hat{\tau}_{\gamma 0})^2=O_\prec(n^{-1-\delta}),\label{17110302}\\
&\mathbb{E} \Im {\tau}_{\gamma 0} \Im {\tau}_{\gamma 1}-\mathbb{E}\Im \hat{\tau}_{\gamma 0} \Im \hat{\tau}_{\gamma 1}=O_\prec(n^{-1-\delta}),\label{17110303}\\
&\mathbb{E} (\Im {\tau}_{\gamma 0})^3-\mathbb{E} (\Im \hat{\tau}_{\gamma 0})^3=O_\prec(n^{-1-\delta}). \label{17101001}
\end{align}
We prove the above estimates one by one. First, for (\ref{17110301}) with $a=0$, we simply have 
$
\mathbb{E} \Im {\tau}_{\gamma 0}-\mathbb{E}\Im \hat{\tau}_{\gamma 0}=0,
$
 since the covariance matrix of $\mb{v}_\gamma$ and that of $\hat{\mb{v}}_\gamma$ are the same. For (\ref{17110301}) with $a=1$, the conclusion follows from the  estimates in (\ref{17112610}) and the bounds of $\mc{P}_\gamma$ and $\hat{\mc{W}}_\gamma$ in  (\ref{17112601}). 
 
 Next, we show  (\ref{17110302}). Observe that for any $\omega_1,\omega_2\in\mathbb{C}$, we can write $\Im \omega_1\Im \omega_2= \frac{1}{4}(\omega_1\bar{\omega}_2+\bar{\omega}_1\omega_2-\omega_1\omega_2-\bar{\omega}_1\bar{\omega}_2)$. According to the definitions in (\ref{17110310}) and (\ref{17110311}), and also the fact that the covariance matrix of $\mb{v}_\gamma$ and that of $\hat{\mb{v}}_\gamma$ are the same, it suffices to show 
 \begin{align}
 \mathbb{E}  \mb{v}_\gamma A_\gamma(z_1) \mb{v}_\gamma' \mb{v}_\gamma A_\gamma(z_2) \mb{v}_\gamma'
 -  \mathbb{E}  \hat{\mb{v}}_\gamma A_\gamma(z_1) \hat{\mb{v}}_\gamma' \hat{\mb{v}}_\gamma A_\gamma(z_2) \hat{\mb{v}}_\gamma'= O_\prec(n^{\frac13-\delta}),  \label{18081201}
 \end{align}
 and,  if we replace one or both of $z_1$ and $z_2$ by their complex conjugates, the analogues of (\ref{18081201}) are also true. 
Here $z_1,z_2$ satisfy the assumptions in Lemma \ref{lem. technical estimates}.  These desired estimates follow from the decompositions in (\ref{171128100}), and the bounds in (\ref{17112601}) for the terms in the decompositions. Similarly, applying the decompositions in (\ref{171128100}), and the bounds in (\ref{17112601}) again, one can show  (\ref{17110303}) and (\ref{17101001}).   We omit the details. This completes the proof of Proposition \ref{lem.17100801}. 
\end{proof}

\section{First-order approximation} \label{s.first order appro}
Recall  (\ref{171130101}). 
We first set 
\begin{align}
\qquad \wt{K}:=\frac{1}{3}I_p+UU', \quad \wt{G}(z):=(\wt{K}-z)^{-1}, \quad \wt{m}(z):=\frac{1}{p} \mr{Tr} \wt{G}(z).  \label{18080710}
\end{align}
In this section, our aim is to establish the following proposition. 
\begin{pro}\label{lem.17100802} Suppose that the  assumptions on $\eta, E_1, E_2, F$ in  Proposition \ref{lem.17100801} hold. For some constant $\delta>0$ and  sufficiently large $n$,  we have
\begin{align*}
&\Big| \mathbb{E} F\Big( n \int_{E_1}^{E_2} \Im \wh{m} (x+{\lambda}_{+,c_n}+\mathrm{i}\eta){\rm d} x\Big)\nonumber\\
&\qquad\qquad-\mathbb{E} F\Big( n \int_{E_1}^{E_2} \Im \wt{m} (x+{\lambda}_{+,c_n}+\mathrm{i}\eta){\rm d} x\Big)\Big|\leq n^{-\delta}. 
\end{align*}
\end{pro}

\begin{proof}[Proof of Proposition \ref{lem.17100802}] 
We first define the following continuous interpolation between $\wh{K}$ and $\wt{K}$  and its Green function for $t\in[0,1]$, 
\begin{align}
\wh{K}_t:= (U+tH)(U+tH)'+\frac{1}{3}(1-t^2)I_p, \quad \wh{G}_t:=(\wh{K}_t-z)^{-1}.  \label{17112902}
\end{align}
and we also denote by 
$
\wh{m}_t:=\frac{1}{p} \mathrm{Tr} \wh{G}_t .
$
Especially, we have $\wh{K}_1=\wh{K}$ and $\wh{K}_0=\wt{K}$.  
Similar to Lemma \ref{lem. local law for K gamma}, we have the following local law for  $\wh{K}_t$, whose proof is stated in the supplementary material \cite{BaoSupp}. 
\begin{lem} [Local law for $\wh{K}_t$] \label{lem. local law for K hat t}
All the estimates in Proposition  \ref{lem.weak law for R}  hold for $\wh{K}_t$ for all $t\in[0,1]$.
\end{lem}
With the aid of Lemma \ref{lem. local law for K hat t}, we now proceed to the proof of Proposition \ref{lem.17100802}. 
For brevity, we simply write $z\equiv z(x):= x+{\lambda}_{+,c_n}+\mathrm{i}\eta$ in the sequel, and further introduce the notation
\begin{align}
\Phi_t:= n \int_{E_1}^{E_2} \Im \wh{m}_t (z){\rm d} x. \label{17112901}
\end{align}
Then we can write 
\begin{align*}
&\mathbb{E} F\Big( n \int_{E_1}^{E_2} \Im \wh{m} (z){\rm d} x\Big)-\mathbb{E} F\Big( n \int_{E_1}^{E_2} \Im \wt{m} (z){\rm d} x\Big)\nonumber\\
&\qquad\qquad= \int_{0}^1  \mathbb{E}  \frac{\partial }{\partial t} F\big( \Phi_t\big)  {\rm d} t=\int_0^1 \mathbb{E} \Big(F'\big(\Phi_t\big) \frac{\partial \Phi_t}{\partial t}\Big) {\rm d} t.
\end{align*}
Our aim is to show
\begin{align*}
\Big| \frac{\partial \Phi_t}{\partial t}\Big| \prec n^{-\delta}, \qquad \forall t\in [0,1].
\end{align*}
This, together with the assumption on $F'$,  leads to the conclusions in Proposition \ref{lem.17100802}. 
From the definition in (\ref{17112901}),  we have 
\begin{align*}
\frac{\partial \Phi_t}{\partial t}= n \int_{E_1}^{E_2} \frac{\partial \Im \wh{m}_t (z)}{\partial t}{\rm d} x= \frac{n}{p} \int_{E_1}^{E_2} \frac{\partial \Im \mathrm{Tr}\wh{G}_t (z)}{\partial t}{\rm d} x. 
\end{align*}
Considering that $|E_1|, |E_2|\leq N^{-\frac23+\varepsilon}$, it suffices to show 
\begin{align}
 \Big|\frac{\partial \mathrm{Tr}\wh{G}_t (z)}{\partial t}\Big|\prec  n^{\frac{2}{3}-\delta} \label{17112903}
\end{align}
 for all $x\in [E_1,E_2]$. 
From the definitions in (\ref{17112902}), we have 
\begin{align*}
\frac{\partial \mathrm{Tr}\wh{G}_t}{\partial t}=-\mathrm{Tr}\Big(\wh{G}_t\big((HU'+UH')+2t(HH'-\frac13)  \big) \wh{G}_t\Big).
\end{align*}
Hence, for (\ref{17112903}),  it suffices to show  the following estimates hold for all $x\in [E_1,E_2]$:
\begin{align}
&\Big|\mathrm{Tr} \big(HU' \wh{G}_t^2\big)\Big|\prec n^{\frac{2}{3}-\delta}, \qquad \Big|\big(\mathrm{Tr}UH' \wh{G}_t^2\big)\Big|\prec n^{\frac{2}{3}-\delta}, \nonumber\\
 &\Big|\mathrm{Tr}\big((HH'-\frac13 ) \wh{G}_t^2\big)\Big|\prec n^{\frac{2}{3}-\delta}. \label{17100901}
\end{align}

We start with  the first estimate in (\ref{17100901}). The other two can be derived similarly. Let
\begin{align}
\mc{P}:= \mathrm{Tr}\big(HU' \wh{G}_t^2\big), \qquad \mf{m}^{(k,\ell)}:= \mc{P}^k\overline{\mc{P}}^\ell.  \label{18021820}
\end{align}  
Our aim is to establish the following recursive moment estimate: for any fixed integer $k>0$
\begin{align}
\mathbb{E} \big(\mf{m}^{(k,k)}\big)=  \mathbb{E} \big(\mathfrak{c}_1\mf{m}^{(k-1,k)}\big)+\mathbb{E} \big( \mathfrak{c}_2\mf{m}^{(k-2,k)}\big)
+\mathbb{E} \big(\mathfrak{c}_3\mf{m}^{(k-1,k-1)}\big) \label{17100902}
\end{align}
for some random quantities $\mathfrak{c}_i, i=1,2,3$ which satisfy 
\begin{align}
&|\mathfrak{c}_1|\prec n^{\frac23-\delta}, \qquad |\mathfrak{c}_2|\prec n^{\frac43-2\delta}, \qquad |\mathfrak{c}_3|\prec n^{\frac43-2\delta}, \label{18021810} \\
&\mathbb{E}|\mathfrak{c}_1|^{2k}\prec n^{2k(\frac23-\delta)}, \qquad \mathbb{E}|\mathfrak{c}_2|^k\prec n^{2k(\frac23-\delta)}, \qquad \mathbb{E}|\mathfrak{c}_3|^k\prec n^{2k(\frac23-\delta)}.  \label{18021811}
\end{align}
Assuming (\ref{17100902}), by Young's inequality, we have for any given small $\varepsilon$
\begin{align*}
\mathbb{E} \big(\mf{m}^{(k,k)}\big)\leq 3\frac{1}{2k}n^{2k\varepsilon} n^{2k(\frac23-\delta)}+3\frac{2k-1}{2k} n^{-\frac{2k\varepsilon}{2k-1}} \mathbb{E} \big(\mf{m}^{(k,k)}\big).
\end{align*}
Since $k$ can be any large (but fixed) positive integer, we can conclude the first estimate in (\ref{17100901}) by applying Markov's inequality.  The above  strategy of recursive moment estimate is inspired by a similar idea used in \cite{LS2}. 

Hence, what remains is to prove (\ref{17100902}). In the sequel, for brevity, we only keep tracking the bounds in (\ref{18021810}). Those in (\ref{18021811}) will follow easily from (\ref{18021810}),  the deterministic bounds of the entries of $G$ and $U$, together with the Gaussian tail of the entries in $H$. 
To this end, we first use the integration by parts formula for Gaussian random variable
\begin{align}
\mathbb{E} \big(\mf{m}^{(k,k)}\big) &= \mathbb{E}\Big(\mathrm{Tr}HU' \wh{G}_t^2 \mf{m}^{(k-1,k)}\Big)=\sum_{a,(ij)}\mathbb{E} \Big(h_{a, (ij)}\big(U' \wh{G}_t^2\big)_{(ij),a} \mf{m}^{(k-1,k)}\Big)\nonumber\\
&=\frac{1}{3M} \sum_{a,(ij)} \mathbb{E} \Big(\frac{\partial \big(U' \wh{G}_t^2\big)_{(ij),a} }{\partial h_{a, (ij)}}\mf{m}^{(k-1,k)}\Big)\nonumber\\
&\qquad+\frac{k-1}{3M} \sum_{a, (ij)} \mathbb{E}\Big(\big(U' \wh{G}_t^2\big)_{(ij),a} \frac{\partial \mc{P}}{\partial h_{a, (ij)}} \mf{m}^{(k-2,k)} \Big)\nonumber\\
&\qquad +\frac{k}{3M} \sum_{a, (ij)} \mathbb{E}\Big(\big(U' \wh{G}_t^2\big)_{(ij),a} \frac{\partial \overline{\mc{P}}}{\partial h_{a, (ij)}} \mf{m}^{(k-1,k-1)} \Big). \label{18021833}
\end{align}
Here we use the notation $\sum_{a,(ij)}$ to represent the sum over $a\in \llbracket 1, p\rrbracket, 1\leq i<j\leq n$. 
Hence, to establish (\ref{17100902}), it suffices  to show 
\begin{align}
&\frac{1}{M} \sum_{a, (ij)} \frac{\partial \big(U' \wh{G}_t^2\big)_{(ij),a} }{\partial h_{a, (ij)}}=O_\prec(n^{\frac23-\delta}), \nonumber\\
& \frac{1}{M} \sum_{a, (ij)}  \big(U' \wh{G}_t^2\big)_{(ij),a} \frac{\partial \mc{P}}{\partial h_{a, (ij)}}= O_\prec(n^{\frac43-2\delta}), \nonumber\\
&  \frac{1}{M} \sum_{a, (ij)}  \big(U' \wh{G}_t^2\big)_{(ij),a} \frac{\partial \overline{\mc{P}}}{\partial h_{a, (ij)}}= O_\prec(n^{\frac43-2\delta}).  \label{17100910}
\end{align}
The proofs of the last two estimates are similar. Hence, we only show the details of the proofs for  the first two estimates  above. Set $\wh{\Theta}_t:= U+tH$. It is easy to obtain from  (\ref{17112902}) that  
\begin{align*}
\frac{\partial \wh{G}_t}{\partial h_{a,(ij)}}=-t\wh{G}_t\big(E_{a,(ij)}\wh{\Theta}_t'+\wh{\Theta}_t (E_{a,(ij)})'\big)\wh{G}_t,
\end{align*}
where we use the notation $E_{a,(ij)}$ to denote the $p\times M$ matrix whose $(a,(ij))$-th entry is $1$ and all the other entries are $0$. 
Then, it is easy to check
\begin{align*}
\frac{\partial \big(U' \wh{G}_t^2\big)_{(ij),a} }{\partial h_{a, (ij)}}
&= -t\big(U'\wh{G}_t\big)_{(ij),a} \big(\wh{\Theta}_t' \wh{G}_t^2\big)_{(ij),a}-t\big(U'\wh{G}_t\wh{\Theta}_t\big)_{(ij)(ij)} \big(\wh{G}^2_t\big)_{aa}\nonumber\\
&\qquad -t \big(U' \wh{G}_t^2\big)_{(ij),a} \big(\wh{\Theta}_t'\wh{G}_t\big)_{(ij),a}-t\big(U' \wh{G}_t^2\wh{\Theta}_t\big)_{(ij)(ij)} \big(\wh{G}_t\big)_{aa},
\end{align*}  
and 
\begin{align*}
\frac{\partial \mc{P}}{\partial h_{a, (ij)}}
=& \big(U' \wh{G}_t^2\big)_{(ij),a}- t\big(\wh{\Theta}_t' \wh{G}_t HU' \wh{G}_t^2\big)_{(ij),a}- t\big(\wh{G}_t HU' \wh{G}_t^2\wh{\Theta}_t\big)_{a,(ij)}\nonumber\\
&- t\big(\wh{\Theta}_t' \wh{G}_t^2 HU' \wh{G}_t\big)_{(ij),a}- t\big(\wh{G}_t^2 HU' \wh{G}_t\wh{\Theta}_t\big)_{a,(ij)}.
\end{align*}

Consequently, we have 
\begin{align}
\frac{1}{M} \sum_{a, (ij)} \frac{\partial \big(U' \wh{G}_t^2\big)_{(ij),a} }{\partial h_{a, (ij)}}= &  
- \frac{t}{M} \mr{Tr} 
\wh{G}_t^2\wh{\Theta}_tU'\wh{G}_t- \frac{t}{M} \mr{Tr}\wh{\Theta}_tU'\wh{G}_t \mr{Tr}\wh{G}^2_t\nonumber\\
&- \frac{t}{M} \mr{Tr} 
\wh{G}_t\wh{\Theta}_tU'\wh{G}_t^2- \frac{t}{M} \mr{Tr}\wh{\Theta}_tU'\wh{G}_t^2 \mr{Tr}\wh{G}_t, \label{17110501}
\end{align}
and 
\begin{align}
 &\frac{1}{M} \sum_{a, (ij)}  \big(U' \wh{G}_t^2\big)_{(ij),a} \frac{\partial \mc{P}}{\partial h_{a, (ij)}}\nonumber\\
&=\frac{1}{M} \mr{Tr} \wh{G}_t^2UU' \wh{G}_t^2-\frac{t}{M} \mr{Tr} \wh{G}_t^2U\wh{\Theta}_t' \wh{G}_t HU' \wh{G}_t^2- \frac{t}{M} \mr{Tr} \wh{G}_t HU' \wh{G}_t^2\wh{\Theta}_tU' \wh{G}_t^2\nonumber\\
&\qquad -\frac{t}{M} \mr{Tr} \wh{G}_t^2U\wh{\Theta}_t' \wh{G}_t^2 HU' \wh{G}_t-\frac{t}{M} \mr{Tr} \wh{G}_t^2 HU' \wh{G}_t\wh{\Theta}_t U' \wh{G}_t^2. \label{17110502}
\end{align}
Now we claim that
\begin{align}
\|HU'\|\prec n^{-\frac12}, \qquad \|UU'\|\prec 1. \label{171104100}
\end{align}
To see the first estimate, we first notice that
\begin{align}
\|HU'UH'\|=\|HT' V_{\Cdot[2]}' V_{\Cdot[2]} T H'\|\prec  \frac{1}{n}\|HT' T H'\|, \label{171104102}
\end{align}
where we use the notation $V_{\Cdot[2]}$ to represent the $p\times n$ matrix with $\mb{v}_{i\Cdot[2]}$ as its $i$-th row.  In the last step, we use the fact that $V_{\Cdot[2]}'V_{\Cdot[2]}$ is a sample covariance matrix with entries (in $V_{\Cdot[2]}$) of order $\frac{1}{\sqrt{M}}\sim \frac{1}{n}$, which implies that $\|V_{\Cdot[2]}' V_{\Cdot[2]}\|\prec \frac1n$ (c.f. Proposition \ref{pro.large deviation for operator norm}).  Further, 
observe that $T'T$ is a rank $n$ matrix with $\|T'T\|=\frac13\|\Gamma\|=O(n)$. Writing the spectral decomposition as $T'T:= O_T' \Lambda_TO_T$, we have the fact that 
\begin{align}
\|HT' T H'\|\prec n\|HO_T'(I_n\oplus 0)O_T H'\|\stackrel{d}= n\|\mc{H}\mc{H}'\|, \label{171104101}
\end{align}
where $\mc{H}$ is a $p\times n$ matrix with i.i.d. $N(0,\frac1M)$ entries.  Then the first estimate in (\ref{171104100}) follows simply from the fact that $\|\mc{H}\mc{H}'\|\prec \frac{1}{n}$,  (\ref{171104101}), and (\ref{171104102}). 
The second estimate in (\ref{171104100}) is easy to see from the fact that $\|U'U\|=\|T' V_{\Cdot[2]}' V_{\Cdot[2]} T \|\prec \frac{1}{n} \|T' T \|\prec 1$.  Then, using (\ref{171104100}) to (\ref{17110501}), we have
\begin{align*}
&\bigg|\frac{1}{M} \sum_{a, (ij)} \frac{\partial \big(U' \wh{G}_t^2\big)_{(ij),a} }{\partial h_{a, (ij)}}\bigg| \prec \frac{1}{M} \mr{Tr} |\wh{G}_t|^3+\frac{1}{M} \mr{Tr} |\wh{G}_t|^2\mr{Tr}|\wh{G}_t|\nonumber\\
&\leq \frac{1}{M\eta^2} \Im \mr{Tr} \wh{G}_t+\frac{1}{M\eta} \Im \mr{Tr} \wh{G}_t \mr{Tr} |\wh{G}_t| \prec n^{\frac{1}{3}+\varepsilon}, 
\end{align*}
where in the last step we use the local laws Lemma \ref{lem. local law for K hat t} and Lemma \ref{lem.properties of m}.

Similarly, using (\ref{17110502}) and (\ref{171104100}), we have 
\begin{align}
&\bigg|\frac{1}{M} \sum_{a, (ij)}  \big(U' \wh{G}_t^2\big)_{(ij),a} \frac{\partial \mc{P}}{\partial h_{a, (ij)}}\bigg|\prec  \frac{1}{M} \mr{Tr} |\wh{G}_t|^4+\frac{1}{M\sqrt{n}} \mr{Tr} |\wh{G}_t|^5\nonumber\\
&\leq  \frac{1}{M\eta^3} \Im \mr{Tr} \wh{G}_t+\frac{1}{M\sqrt{n}\eta^4} \Im \mr{Tr} \wh{G}_t \prec n^{\frac{5}{6}+\varepsilon},  \label{171129100}
\end{align}
where again  in the last step we use the local laws Lemma \ref{lem. local law for K hat t} and Lemma \ref{lem.properties of m}. 
Hence, we conclude the proof of the first two estimates in (\ref{17100910}). The last one can be proved similarly to the second one, we thus omit the details. Therefore, we get (\ref{17100902}). Then, by Young's inequality, we can get the first estimate in (\ref{17100901}). The second estimate in (\ref{17100901}) can be proved analogously and thus we omit the details.   For the last estimate in (\ref{17100901}), we can also use the same strategy, and the details of its proof  is stated in the supplementary material \cite{BaoSupp}.  Therefore, we completed the proof of Proposition \ref{lem.17100802}. 
\end{proof}

\section{Edge universality for $K$} \label{s. universality for K}
With Propositions \ref{lem.17100801} and \ref{lem.17100802} , we can now prove Theorem \ref{thm. universality for Kendall} and Corollary \ref{cor. TW for K}. 

\begin{proof}[Proof of Theorem \ref{thm. universality for Kendall}] Using Propositions \ref{lem.17100801} and \ref{lem.17100802}, we obtain
\begin{align} 
&\Big| \mathbb{E} F\Big( n \int_{E_1}^{E_2} \Im m (x+{\lambda}_{+,c_n}+\mathrm{i}\eta){\rm d} x\Big)\nonumber\\
&\qquad\qquad-\mathbb{E} F\Big( n \int_{E_1}^{E_2} \Im \wt{m} (x+{\lambda}_{+,c_n}+\mathrm{i}\eta){\rm d} x\Big)\Big|\leq n^{-\delta},  \label{17120261}
\end{align}
where $F, E_1, E_2$ and $\eta$ satisfy the assumptions in Proposition \ref{lem.17100801}.  Similar to the proof of Theorem 1.1 in \cite{PY}, one can show by using (\ref{17120261}) and the local laws that 
\begin{align}
&\mathbb{P} \Big(n^{\frac23}(\lambda_1(K)- {\lambda}_{+,c_n})\leq s-n^{-\varepsilon}\Big)-n^{-\delta}
\leq \mathbb{P} \Big(n^{\frac23}(\lambda_1(\wt{K})- \lambda_{+,c_n})\leq s\Big)\nonumber\\
&\qquad\qquad\leq  \mathbb{P} \Big( n^{\frac23}(\lambda_1(K)- {\lambda}_{+,c_n})\leq s+n^{-\varepsilon}\Big)+n^{-\delta} \label{17120262}
\end{align} 

Further, we observe that 
$
UU'=V_{\Cdot[2]}TT'V_{\Cdot[2]}. 
$
In addition, we notice that 
$
TT'=nI_n-\mathbf{1}\mathbf{1}'.
$
Denoting by $\mc{V}:=\sqrt{\frac{3}{2}(n-1)}V_{\Cdot[2]}$, and $\Sigma=I_n-\frac{1}{n}\mathbf{1}\mathbf{1}'$, we can write 
\begin{align}
\wt{K}= UU'+\frac{1}{3}I_p=\frac{2n}{3(n-1)} \mc{V}\Sigma\mc{V}'+\frac{1}{3}I_p.  \label{171129110}
\end{align}
It is known from Theorem 2.7 of \cite{BKYY} that  the largest eigenvalues of $\mc{V}\Sigma\mc{V}'$  differ from the corresponding ones  of $\mc{V}\mc{V}'$ only by $O_\prec(\frac{1}{n})$.  This together with Theorem 1.1 in \cite{PY} leads to
\begin{align}
\mathbb{P} \Big(\frac32n^{\frac23}(\lambda_1(\wt{K})- {\lambda}_{+,c_n})  &\leq s-n^{-\varepsilon}\Big)-n^{-\delta}\leq \mathbb{P} \Big(n^{\frac23}(\lambda_1(Q)- d_{+,c_n})\leq s\Big)\nonumber\\
&\leq  \mathbb{P} \Big(\frac32 n^{\frac23}(\lambda_1(\wt{K})- {\lambda}_{+,c_n})\leq s\Big)+n^{-\delta}. \label{17120263}
\end{align}
Combining (\ref{17120262}) and (\ref{17120263}) we obtain (\ref{17120265}). This concludes the proof. 
\end{proof}

\begin{proof}[Proof of Corollary \ref{cor. TW for K}] The conclusion follows directly from Theorem  \ref{thm. universality for Kendall} and the Tracy-Widom limit for $\lambda_1(Q)$ (c.f \cite{Johnstone}). 
\end{proof}

\section*{Acknowledgement} The author would like to thank Jiang Hu and Wang Zhou for helpful discussion.

\newpage

\section*{Supplementary material}

In this supplementary material,  we provide the proofs of some Propositions, Lemmas, and also state more simulation results. In Section S.1, we state the proofs of   Proposition \ref{lem. large deviation linear part},  Proposition \ref{lem. large deviation nonlinear part}, and  Corollary \ref{cor. large devi}. In Section S.2, we state the proof of Proposition \ref{lem.weak law for R}.  Section S.3 is devoted to the proofs of Lemmas \ref{lem. local law for K gamma}, \ref{lem. technical estimates} and \ref{lem. local law for K hat t}, and also the proof of the last estimate in (\ref{17100901}). In Section S.4, we collect some basic technical tools, and in Section S.5, we present more simulation results.

\
\
\
\

\noindent {\bf S1: Proofs of the large deviation estimates}

In this section, we state the proofs of Propositions \ref{lem. large deviation linear part} and \ref{lem. large deviation nonlinear part}, and also the proof of Corollary \ref{cor. large devi}. 

We first collect some technical results on Hoeffding decomposition in the following lemma.
\begin{lem} \label{lem.technical lem on Hoeffding} With the notations introduced in (\ref{170920100}) and (\ref{17101820}), we have 
\begin{align}
&\mathbb{E}(\bar{v}_{k,(ij)}|w_{ki})=\mathbb{E}(\bar{v}_{k,(ij)}|w_{kj})=0, \qquad i\neq j,  \label{17111970}\\
& \mathbb{E}\big(\bar{v}_{k,(ij)}^2|w_{ki}\big)=  \mathbb{E}\big(\bar{v}_{k,(ij)}^2|w_{kj}\big)=\frac13, \qquad i\neq j,
 \label{17102801} \\
 & \mathbb{E} (v_{k,(i\cdot)}\bar{v}_{k,(ij)}|w_{kj})= \frac{1}{2}\big(\frac13-v_{k,(j\cdot)}^2\big), \qquad i\neq j.  \label{171128153}
\end{align}
\end{lem}

\begin{proof}[Proof of  Lemma \ref{lem.technical lem on Hoeffding}] First, (\ref{17111970}) follows easily from the definitions in (\ref{17101820}) and (\ref{170920100}), and also the fact (\ref{17092101}).

Next, we prove  (\ref{17102801}).  First, by the trivial fact $|v_{k,(ij)}|=1$ and (\ref{17092150}), we have 
\begin{align}
\mathbb{E} \big(\bar{v}_{k,(ij)}^2|w_{kj}\big)&=\mathbb{E}\big( (v_{k,(ij)}-v_{k,(i\cdot)}+v_{k,(j\cdot)})^2|w_{kj}\big)\nonumber\\
&= \frac43-v_{k,(j\cdot)}^2-2\mathbb{E} \big(v_{k,(ij)}v_{k,(i\cdot)} |w_{kj}\big). \label{17102804}
\end{align}
From the definition (\ref{17091820}), we further observe that 
\begin{align}
&\mathbb{E} \big(v_{k,(ij)}v_{k,(i\cdot)} |w_{kj}\big) = \mathbb{E}\big(v_{k,(i\cdot)}\mathbbm{1}(w_{ki}>w_{kj})|w_{kj}\big)-\mathbb{E}\big(v_{k,(i\cdot)}\mathbbm{1}(w_{ki}<w_{kj})|w_{kj}\big)\nonumber\\
&=  \mathbb{E}\big(v_{k,(i\cdot)}\mathbbm{1}(v_{k,(i\cdot)}>v_{k, (j\cdot)})|w_{kj}\big)-\mathbb{E}\big(v_{k,(i\cdot)}\mathbbm{1}(v_{k, (i\cdot)}<v_{k, (j\cdot)})|w_{kj}\big)\nonumber\\
&=\frac12(1-v_{k, (j\cdot)}^2), \label{17102803}
\end{align}
where in the first step above we use the fact (\ref{17092101}) and the monotonicity of $F_k$, and in the second step we  use the fact that $v_{k,(i\cdot)}$ is uniformly distributed on $[-1,1]$.  Plugging (\ref{17102803}) into (\ref{17102804}) yields  (\ref{17102801}).

Further, using Hoeffding decomposition again, we observe that  
\begin{align*}
&\mathbb{E} \big(v_{k,(ij)}v_{k,(i\cdot)} |w_{kj}\big)= \mathbb{E} \big((v_{k,(i\cdot)}-v_{k,(j\cdot)}+\bar{v}_{k,(ij)})v_{k,(i\cdot)} |w_{kj}\big)\nonumber\\
&=  \mathbb{E} v_{k,(i\cdot)}^2+\mathbb{E} \big(\bar{v}_{k,(ij)})v_{k,(i\cdot)} |w_{kj}\big)=\frac13+\mathbb{E} \big(\bar{v}_{k,(ij)}v_{k,(i\cdot)} |w_{kj}\big),
\end{align*}
which together with (\ref{17102803}) leads to (\ref{171128153}).  This completes the proof of Lemma \ref{lem.technical lem on Hoeffding}. 
\end{proof}

\begin{proof}[Proof of Proposition \ref{lem. large deviation linear part}] First, according to the definitions in (\ref{171128200}), (\ref{17101010}) and (\ref{17112630}), 
we can write
\begin{align}
\mb{u}_k=\mb{v}_{k,\Cdot[2]} T, \label{17102202}
\end{align}
where we introduce the notation 
\begin{align}
\mb{v}_{k,\Cdot[2]}:= {\frac{1}{\sqrt{M}}}(v_{k,(1\cdot)}, \ldots, v_{k,(n\cdot)})\in \mathbb{R}^n.
\end{align}
Then, using (\ref{171128201}), it is easy to see that 
\begin{align*}
\mathbb{E} \mb{u}'_k\mb{u}_k=\mathbb{E}  T' \mb{v}_{k,\Cdot[2]}'\mb{v}_{k,\Cdot[2]}T=\frac{1}{3M}T'T= \frac{1}{M}\Gamma, 
\end{align*}
where we use (\ref{17120301}). 
Consequently, (\ref{17092110}) follows  from 
\begin{align*}
&\mathbb{E}\mb{u}_k B\mb{u}_k'=\text{Tr} B \mathbb{E} \mb{u}_k'\mb{u}_k
=\frac{1}{M}\text{Tr} B\Gamma.
\end{align*}
Further, using (\ref{17102202}) again, we can write 
\begin{align*}
\mb{u}_k\mathbf{a}'= \mb{v}_{k,\Cdot[2]} T\mb{a}', \qquad   \mb{u}_k B\mb{u}_k'= \mb{v}_{k,\Cdot[2]} TBT'   \mb{v}_{k,\Cdot[2]}'. 
\end{align*}
Using the randomness of $\mb{v}_{k,\Cdot[2]}$, we can get  (\ref{17092170}) and (\ref{17092130}) from the large deviation estimate of random vector with independent entries (c.f. Corollary B.3 of  \cite{EYY12} for instance), and also the fact (\ref{17110901}). 

This completes the proof of Proposition \ref{lem. large deviation linear part}. 
\end{proof}

\begin{proof}[Proof of Proposition \ref{lem. large deviation nonlinear part}] In this proof, we fix a $k\in \llbracket 1, p\rrbracket$.  Recall  the definitions in (\ref{17091820}), (\ref{170920100}) and (\ref{17101820}).  We first define the filtration 
\begin{align}
\mathcal{F}_0=\emptyset, \qquad \mathcal{F}_\ell:=\sigma(w_{k1}, \ldots, w_{k\ell}), \qquad \ell\in \llbracket1, n\rrbracket,  \label{17092201}
\end{align}
where we omit the dependence on $k$ from the above notations. 
We first prove (\ref{17101050}).
Define the martingale difference
\begin{align}
\mathcal{M}_\ell:=&\mathbb{E}\big(\bar{\mb{v}}_k\mathbf{a}'|\mathcal{F}_\ell \big)-\mathbb{E}\big(\bar{\mb{v}}_k\mathbf{a}'|\mathcal{F}_{\ell-1} \big). \label{17092172}
\end{align}
Using (\ref{17111970}), it is easy to check
\begin{align*}
\mathcal{M}_\ell=& \frac{1}{\sqrt{M}}\sum_{i=1}^{\ell-1} a_{(i\ell)}\Big(\mathbb{E}\big( \bar{v}_{k,(i\ell)}|\mathcal{F}_\ell \big)-\mathbb{E}\big(\bar{v}_{k,(i\ell)}|\mathcal{F}_{\ell-1} \big)\Big)\nonumber\\
&+ \frac{1}{\sqrt{M}}\sum_{j=\ell+1}^n a_{(\ell j)}\Big(\mathbb{E}\big( \bar{v}_{k,(\ell j)}|\mathcal{F}_\ell \big)-\mathbb{E}(\bar{v}_{k,(\ell j)}|\mathcal{F}_{\ell-1} \big)\Big)= \frac{1}{\sqrt{M}}\sum_{i=1}^{\ell-1} a_{(i\ell)} \bar{v}_{k,(i\ell)}.
\end{align*}
Further, we define the following filtration  for a given $\ell$,
\begin{align}
\mc{F}_{\gamma, \ell}=\sigma(w_{k1}, \ldots, w_{k\gamma}, w_{k\ell}), \qquad \gamma\in \llbracket1,  \ell-1\rrbracket. \label{17112710}
\end{align}

Observe that for each given $\ell$, the sequence $\{\frac{1}{\sqrt{M}}\sum_{i=1}^{\gamma} a_{(i\ell)} \bar{v}_{k,(i\ell)}\}_ {\gamma=1}^{ \ell-1}$ itself is a martingale w.r.t. the filtration $\{\mc{F}_{\gamma, \ell}\}_{\gamma=1}^{\ell-1}$, according to the fact (\ref{17111970}). Using Burkholder inequality and the boundedness of $\bar{v}_{k,(i\ell)}$'s, we have  for any integer $ q\geq 2$
\begin{align*}
\mathbb{E}|\mathcal{M}_\ell|^q\leq (Cq)^{\frac{3q}{2}} \Big(\frac{1}{M}\sum_{i=1}^{\ell-1} |a_{(i\ell)}|^2\Big)^{\frac{q}{2}}.
\end{align*}
Hence, we have
\begin{align}
|\mathcal{M}_\ell|\prec \sqrt{ \frac{1}{{M}}\sum_{i=1}^{\ell-1} |a_{(i\ell)}|^2}. \label{17092180}
\end{align}
Then, using Burkholder inequality again, we have 
\begin{align}
\mathbb{E}\Big|\sum_{\ell} \mathcal{M}_\ell\Big|^q\leq (Cq)^{\frac{3q}{2}} \mathbb{E} \Big(\sum_{\ell} |\mathcal{M}_\ell|^2\Big)^{\frac{q}{2}}. \label{17101060} 
\end{align}
From  (\ref{17092180}), we see that 
\begin{align}
\sum_{\ell=1}^n |\mathcal{M}_\ell|^2\prec \frac{1}{M} \sum_{i<j} |a_{(ij)}|^2=\frac{\|\mb{a}\|^2}{M}. \label{17101061}
\end{align}
Also notice that by the deterministic boundedness of $\bar{v}_{k,(ij)}$ and Cauchy-Schwarz, we can  get the deterministic bound $|\mc{M}_\ell|\leq C\sqrt{\sum_{i=1}^{\ell-1} |a_{(i\ell)}|^2/n} $. Plugging in this deterministic bound together with the stochastic bound (\ref{17101061}) to (\ref{17101060}), in light of Definition \ref{definition of stochastic domination}, we can conclude
\begin{align*}
\mathbb{E}\Big|\sum_{\ell} \mathcal{M}_\ell\Big|^q\prec (Cq)^{\frac{3q}{2}}  \Big(\frac{\|\mb{a}\|^2}{M}\Big)^{\frac{q}{2}}.
\end{align*}
Then, by Markov inequality, we obtain  (\ref{17101050}). 

Observe that normally the stochastic bound like (\ref{17101061}) cannot  directly imply the bound for  moments such as  the RHS  of (\ref{17101060}). But one can indeed do so  if there is  also a crude but deterministic bound for the random variable. This fact has been discussed below Definition \ref{definition of stochastic domination}.
And this will be always the case in the remaining proof. Hence, without further justification, we will regard the stochastic bounds in the sequel as deterministic and plug them into the  moment estimates directly.

Next, we prove (\ref{17101051}). Recall the filtration (\ref{17092201}).  We set the martingale difference
  \begin{align}
  \mc{L}_\ell:= \mathbb{E} \big( \mb{u}_k B\bar{\mb{v}}_k'|\mathcal{F}_\ell \big)-  \mathbb{E} \big( \mb{u}_k B\bar{\mb{v}}_k'|\mathcal{F}_{\ell-1} \big). \label{17102230}
  \end{align}
Using (\ref{17102202}), we have 
\begin{align}
 \mb{u}_k B\bar{\mb{v}}_k'= \mb{v}_{k,\Cdot[2]} TB \bar{\mb{v}}_k'= \frac{1}{M}\sum_{a}\sum_{i<j} (TB)_{a,(ij)} v_{k,(a\cdot)} \bar{v}_{k,(ij)}.   \label{17112701}
\end{align}
According to (\ref{17112701}) and the definition in (\ref{17102230}), by  Lemma \ref{lem.technical lem on Hoeffding}, it is not difficult to derive
\begin{align*}
\mc{L}_\ell =&  \frac{1}{M}\sum_{a}\sum_{i<j} (TB)_{a,(ij)} \Big(\mathbb{E} \big(v_{k,(a\cdot)} \bar{v}_{k,(ij)}| \mathcal{F}_\ell\big)- \mathbb{E} \big(v_{k,(a\cdot)} \bar{v}_{k,(ij)}| \mathcal{F}_{\ell-1}\big)\Big)\nonumber\\
=& \mc{L}_{\ell1}+ \mc{L}_{\ell2}+ \mc{L}_{\ell3}+\mc{L}_{\ell4}, 
\end{align*}
where 
\begin{align*}
&\mc{L}_{\ell1}:=  \frac{1}{M}\sum_{i=1}^{\ell-1} (TB)_{\ell,(i\ell)} \Big(v_{k,(\ell\cdot)} \bar{v}_{k,(i\ell)}- \frac12(v_{k,(i\cdot)}^2-\frac{1}{3})\Big),\nonumber\\
& \mc{L}_{\ell2}:= \frac{1}{M}\sum_{j>i=1}^{\ell-1} (TB)_{\ell,(ij)} v_{k,(\ell\cdot)} \bar{v}_{k,(ij)},\nonumber\\
 &  \mc{L}_{\ell3}:= \frac{1}{2M}\sum_{j=\ell+1}^n (TB)_{j,(\ell j)} (v_{k,(\ell\cdot)}^2-\frac{1}{3}),\nonumber\\
 & \mc{L}_{\ell4}:= \frac{1}{M}\sum_{a=1}^{\ell-1}\sum_{i=1}^{\ell-1} (TB)_{a,(i\ell)} v_{k,(a\cdot)} \bar{v}_{k,(i\ell)}.
\end{align*}
Here we use the notation $\sum_{j>i=1}^{\ell-1}$ to represent the double sum $\sum_{i=1}^{\ell-1}\sum_{j=i+1}^{\ell-1}$ for short. 
Using (\ref{17101050}) to the sum $\sum_{i=1}^{\ell-1} (TB)_{\ell,(i\ell)} \bar{v}_{k,(i\ell)}$ and the large deviation for the linear form of i.i.d. random variables (c.f. Corollary B.3 of  \cite{EYY12} for instance) to the sum $\sum_{i=1}^{\ell-1} (TB)_{\ell,(i\ell)}(v_{k,(i\cdot)}^2-\frac{1}{3})$, we get
\begin{align}
|\mc{L}_{\ell1}|\prec  \frac{1}{M}\sqrt{\sum_{i=1}^{\ell-1} \big|(TB)_{\ell,(i\ell)}\big|^2}.  \label{17112730}
\end{align}
Again, using (\ref{17101050}) to $\sum_{j>i=1}^{\ell-1} (TB)_{\ell,(ij)}\bar{v}_{k,(ij)}$, and also using the boundedness of  $v_{k,(\ell\cdot)}$, we can analogously get 
\begin{align}
|\mc{L}_{\ell2}|\prec  \frac{1}{M}\sqrt{\sum_{j>i=1}^{\ell-1} \big|(TB)_{\ell,(ij)}\big|^2}. \label{17112731}
\end{align}
For $\mc{L}_{\ell 3}$, we use the boundedness of $v_{k,(\ell\cdot)}$, and get 
\begin{align}
|\mc{L}_{\ell3} |\leq \frac{1}{M} \Big|\sum_{j=\ell+1}^n (TB)_{j,(\ell j)}\Big|. \label{17112732}
\end{align}
To bound $\mc{L}_{\ell 4}$, we do another martingale decomposition. Recall the filtration defined in (\ref{17112710}).  We define 
\begin{align}
\mc{N}_{\gamma, \ell}=\mathbb{E} \big(\mc{L}_{\ell 4}|\mc{F}_{\gamma, \ell}\big)-  \mathbb{E} \big(\mc{L}_{\ell 4}|\mc{F}_{\gamma-1, \ell}\big), \qquad \gamma\in \llbracket 1, \ell-1\rrbracket. \label{17112715}
\end{align}  

In light of the definition (\ref{17112715}), and (\ref{17111970}), it is not difficult to check that 
\begin{align*}
\mc{N}_{\gamma, \ell}
=&  \frac{1}{M} (TB)_{\gamma,(\gamma\ell)} \Big(v_{k,(\gamma\cdot)} \bar{v}_{k,(\gamma\ell)}- \frac{1}{2}(\frac13-v_{k,(\ell, \cdot)}^2)\big)\Big)\nonumber\\
&+ \frac{1}{M}\sum_{i=1}^{\gamma-1} (TB)_{\gamma,(i\ell)} v_{k,(\gamma \cdot)} \bar{v}_{k,(i\ell)}+ \frac{1}{M}\sum_{a=1}^{\gamma-1} (TB)_{a,(\gamma\ell)} v_{k,(a\cdot)} \bar{v}_{k,(\gamma\ell)}
\end{align*}
Using  (\ref{17101050}) to the sum  $\sum_{i=1}^{\gamma-1} (TB)_{\gamma,(i\ell)}  \bar{v}_{k,(i\ell)}$ and the large deviation for the linear form i.i.d. random variables (c.f. Corollary B.3 of  \cite{EYY12} for instance) to the sum $\sum_{a=1}^{\gamma-1} (TB)_{a,(\gamma\ell)} v_{k,(a\cdot)}$, we can conclude the bound 
\begin{align*}
|\mc{N}_{\gamma, \ell}|\prec  \frac{1}{M} |(TB)_{\gamma,(\gamma\ell)}|+ \frac{1}{M}\sqrt{\sum_{i=1}^{\gamma-1} \big|(TB)_{\gamma,(i\ell)}\big|^2}+  \frac{1}{M}\sqrt{\sum_{a=1}^{\gamma-1} \big|(TB)_{a,(\gamma\ell)}\big|^2}. 
\end{align*}
Since  $\mc{L}_{\ell 4}=\sum_{\gamma=1}^{\ell-1} \mc{N}_{\gamma,\ell}$ is a martingale, using Burkholder inequality we have  
\begin{align*}
\mathbb{E}|\mc{L}_{\ell 4}|^{q}\leq (Cq)^{\frac{3q}{2}} \mathbb{E}\Big(\sum_{\gamma=1}^{\ell-1} |\mc{N}_{\gamma,\ell}|^2\Big)^{\frac{q}{2}}\prec  (Cq)^{\frac{3q}{2}} \Big(\frac{1}{M^2}\sum_{a=1}^{\ell-1} \sum_{i=1}^{\ell-1} |(TB)_{a,(i\ell)}|^2 \Big)^{\frac{q}{2}}.
\end{align*}
By Markov inequality, we then have 
\begin{align}
|\mc{L}_{\ell 4}|\prec \sqrt{\frac{1}{M^2}\sum_{a=1}^{\ell-1} \sum_{i=1}^{\ell-1} |(TB)_{a,(i\ell)}|^2}.  \label{17112733}
\end{align}

Now, further, since $ \mb{u}_k B\bar{\mb{v}}_k'= \sum_{\ell=1}^n \mc{L}_\ell$ is a martingale, we can again use the Burkholder inequality to get 
\begin{align}
\mathbb{E} |\mb{u}_k B\bar{\mb{v}}_k'|^q\leq (Cq)^{\frac{3q}{2}}  \mathbb{E}\Big( \sum_{\ell=1}^n |\mc{L}_{\ell}|^2\Big)^{\frac{q}{2}}. \label{17112751}
\end{align}
From  (\ref{17112730}), (\ref{17112731}), (\ref{17112732}) and (\ref{17112733}), we have 
\begin{align}
\sum_{\ell=1}^n \mc{L}_{\ell}^2\prec  \frac{1}{M^2} \mr{Tr} (TB)(TB)^*+ \frac{1}{M^2}\sum_{\ell=1}^n \Big|\sum_{j=\ell+1}^n (TB)_{j,(\ell j)}\Big|^2\nonumber\\
\prec   \frac{n}{M^2} \mr{Tr} |B|^2+\frac{1}{M^2}\sum_{\ell=1}^n \Big|\sum_{j=\ell+1}^n (TB)_{j,(\ell j)}\Big|^2.  \label{17112750}
\end{align}
Plugging (\ref{17112750}) into (\ref{17112751}) and using Markov inequality we conclude  (\ref{17101051}).

Next, we prove (\ref{17101052}). We first observe that $\mathbb{E} (\bar{\mb{v}}_k B\bar{\mb{v}}_k')=\frac{1}{3M} \mr{Tr} B$, in light of (\ref{171128201}) and (\ref{17111970}). We then decompose the quadratic form  into four parts
\begin{align}
&\bar{\mb{v}}_k B\bar{\mb{v}}_k'-\mathbb{E} (\bar{\mb{v}}_k B\bar{\mb{v}}_k')= \frac{1}{M}\sum_{i<j} b_{(ij)(ij)}((\bar{v}_{k,(ij)})^2-\frac13)\nonumber\\
&\qquad\qquad+ \frac{1}{M}\sum_{i<j}\sum_{t}\mathbbm{1}(j\neq t) b_{(ij)(it)} \bar{v}_{k,(ij)} \bar{v}_{k, (it)}\nonumber\\
&\qquad\qquad+ \frac{1}{M}\sum_{ s<j }\sum_i \mathbbm{1}(i\neq s)  b_{(ij)(sj)}\bar{v}_{k,(ij)} \bar{v}_{k, (sj)} \nonumber\\
&\qquad\qquad + \frac{1}{M}\sum_{i<j}\sum_{ s<t} \mathbbm{1}(j\neq t) \mathbbm{1}(i\neq s)  b_{(ij)(st)} \bar{v}_{k,(ij)} \bar{v}_{k, (st)}\nonumber\\
&\qquad\qquad=: Z_1+Z_2+Z_3+Z_4. \label{17102805}
\end{align}
In the sequel, we estimate $Z_i, i=1, 2,3,4$  one by one.   With the aid of (\ref{17102801}), we first estimate  $Z_1$. We recall the filtration $\mathcal{F}_\ell$  in (\ref{17092201}), and define 
\begin{align*}
\zeta_{1\ell}:=& \mathbb{E}\big(Z_1|\mc{F}_\ell\big)-\mathbb{E}\big(Z_1|\mc{F}_{\ell-1}\big) \nonumber\\
=&\frac{1}{M}\sum_{j=\ell+1}^n b_{(\ell j)(\ell j)}\Big(\mathbb{E}\big((\bar{v}_{k,(\ell j)})^2|\mathcal{F}_\ell\big)-\mathbb{E}\big((\bar{v}_{k,(\ell j)})^2|\mathcal{F}_{\ell-1}\big) \Big)\nonumber\\
&+\frac{1}{M}\sum_{i=1}^{\ell-1 } b_{(i\ell)(i\ell)}\Big(\mathbb{E}\big((\bar{v}_{k,(i\ell)})^2|\mathcal{F}_\ell\big)-\mathbb{E}\big((\bar{v}_{k,(i\ell)})^2|\mathcal{F}_{\ell-1}\big) \Big)\nonumber\\
=& \frac{1}{M}\sum_{i=1}^{\ell-1 } b_{(i\ell)(i\ell)}\Big((\bar{v}_{k,(i\ell)})^2-\frac13\big),
\end{align*}  
where in the second step we use (\ref{17102801}). Observe that $\{b_{(i\ell)(i\ell)}\Big((\bar{v}_{k,(i\ell)})^2-\frac13\big)\}_{i=1}^{\ell-1}$ is a martingale difference sequence w.r.t. the filtration $\{\mc{F}_{i,\ell}\}_{i=1}^{\ell-1}$ for any given $\ell$, by the fact (\ref{17102801}). Hence, we have   
\begin{align*}
|\zeta_{1\ell}|\prec  \sqrt{\frac{1}{M^2}\sum_{i=1}^{\ell-1 } |b_{(i\ell)(i\ell)}|^2}. 
\end{align*}
which further implies
\begin{align*}
\sum_{\ell}|\zeta_{1\ell}|^2\prec  \frac{1}{M^2}\sum_{i<\ell}  |b_{(i\ell)(i\ell)}|^2. 
\end{align*}
Similar to the proofs for (\ref{17101050}) and (\ref{17101051}), we can then use  Burkholder inequality to conclude
\begin{align}
|Z_{1}|=\big|\sum_{\ell=0}^n \zeta_{1\ell}\big|\prec \sqrt{ \frac{1}{M^2}\sum_{i<\ell} | b_{(i\ell)(i\ell)}|^2}.  \label{17102830}
\end{align}

Next, we show the estimate of $Z_2$.  By definition, we can write 
\begin{align}
Z_2= \frac{1}{M}\sum_i\sum_{{j, t =i+1}}^n \mathbbm{1}(j\neq t) b_{(ij)(it)} \bar{v}_{k,(ij)} \bar{v}_{k, (it)}=:  \sum_iZ_2^{(i)}. \label{17102811}
\end{align}
In the following, we fix an $i$, and estimate one summand $Z_2^{(i)}$. We introduce the filtration 
\begin{align*}
\mathcal{F}^{(i)}_{\ell}:=\sigma (w_{k i}, w_{k,i+1}, \ldots, w_{k,\ell}), \qquad i+1\leq \ell\leq n. 
\end{align*}
Now, we define the martingale difference for $\ell\in\llbracket i+1,n\rrbracket$
\begin{align}
\zeta_{2\ell}^{(i)}:=& \frac{1}{M}\sum_{{j, t =i+1}}^n \mathbbm{1}(j\neq t) b_{(ij)(it)}\Big( \mathbb{E}\big(\bar{v}_{k,(ij)} \bar{v}_{k, (it)}|\mc{F}_\ell^{(i)}\big)- \mathbb{E}\big(\bar{v}_{k,(ij)} \bar{v}_{k, (it)}|\mc{F}_{\ell-1}^{(i)}\big)\Big)\nonumber\\
=&  \frac{1}{M}\bar{v}_{k,(i,\ell)}\Big(\sum_{{ t =i+1}}^{\ell-1} b_{(i\ell)(it)} \bar{v}_{k, (it)}+ \sum_{{ j =i+1}}^{\ell-1}  b_{(ij)(i\ell)}\bar{v}_{k,(ij)} \Big), \label{17112780}
\end{align}
where the second step follows from  (\ref{17111970}). 
Applying (\ref{17101050}), we have 
\begin{align*}
&\Big|\sum_{{ t =i+1}}^{\ell-1} b_{(i\ell)(it)} \bar{v}_{k, (it)}\Big|\prec  \sqrt{\sum_{{ t =i+1}}^{\ell-1} |b_{(i\ell)(it)}|^2}, \nonumber\\
& \Big| \sum_{{ j =i+1}}^{\ell-1}  b_{(ij)(i\ell)}\bar{v}_{k,(ij)} \Big|\prec \sqrt{\sum_{{ j =i+1}}^{\ell-1}  |b_{(ij)(i\ell)}|^2}.
\end{align*}
Then it is elementary to show that 
\begin{align*}
\sum_{\ell} |\zeta_{2\ell}^{(i)}|^2 \prec \frac{1}{M^2} \sum_{j,t=i+1}^{n}  |b_{(ij)(it)}|^2.
\end{align*}
Further, by Burkholder inequality, we get 
\begin{align}
|Z_{2}^{(i)}|=\big|\sum_{\ell} \zeta_{2\ell}^{(i)}\big|\prec  \frac{1}{M}\sqrt{ \sum_{j,t=i+1}^{n}  |b_{(ij)(it)}|^2}.  \label{17102810}
\end{align}
Plugging (\ref{17102810}) into (\ref{17102811}) and using Cauchy-Schwarz inequality, we obtain
\begin{align}
|Z_2|\prec \sqrt{\frac{n}{M^2}\sum_i   \sum_{j,t=i+1}^{n}  |b_{(ij)(it)}|^2}. \label{17102831}
\end{align}
Similarly, we can show 
\begin{align}
|Z_3|\prec  \sqrt{\frac{n}{M^2} \sum_j \sum_{i, s=1 }^{j-1} |b_{(ij)(sj)}|^2}. \label{17102832}
\end{align}

Finally, we estimate  $Z_4$.   We define the martingale difference sequence 
\begin{align*}
\zeta_{4\ell}:= \mathbb{E}\big(Z_4|\mc{F}_\ell\big)-\mathbb{E}\big(Z_4|\mc{F}_{\ell-1}\big).  
\end{align*}
Similarly to (\ref{17112780}), one can use (\ref{17111970}) to  derive that 
\begin{align}
\zeta_{4\ell}=& \frac{1}{M}\sum_{i=1}^{\ell-1}\sum_{t>s=1}^{\ell-1} b_{(i\ell)(st)}\bar{v}_{k,(i\ell)}\bar{v}_{k,(st)}+\frac{1}{M}\sum_{j>i=1}^{\ell-1}\sum_{s=1}^{\ell-1} b_{(ij)(s\ell)} \bar{v}_{k,(ij)}\bar{v}_{k,(s\ell)}.  \label{17102253}
\end{align}
  The estimate of the two terms in the RHS of (\ref{17102253}) can be done similarly.  Hence, we only show the details for the first term in the sequel.  Applying (\ref{17101050}), we  have 
  \begin{align}
  \Big|\frac{1}{\sqrt{M}}\sum_{t=1}^{\ell-1}\sum_{s=1}^{t-1} b_{(i\ell)(st)}\bar{v}_{k,(st)}\Big|\prec\sqrt{\frac{1}{M} \sum_{t=1}^{\ell-1}\sum_{s=1}^{t-1} |b_{(i\ell)(st)}|^2 }.
  \end{align}
Therefore, using Cauchy-Schwarz, we have 
  \begin{align*}
  \Big|\frac{1}{M}\sum_{i=1}^{\ell-1}\sum_{t>s=1}^{\ell-1} b_{(i\ell)(st)}\bar{v}_{k,(i\ell)}\bar{v}_{k,(st)}\Big|\prec \sqrt{\frac{n}{M^2}\sum_{i=1}^{\ell-1}\sum_{t>s=1}^{\ell-1}  |b_{(i\ell)(st)}|^2 }. 
  \end{align*}
  The estimate for the second term in the RHS of (\ref{17102253}) is similar. Consequently, we have 
  \begin{align*}
  \sum_{\ell} |\zeta_{4\ell}|^2\prec  \frac{n}{M^2}\sum_{i<\ell}  \sum_{s<t}  |b_{(i\ell)(st)}|^2.
  \end{align*}
  Therefore, by Burkholder inequality, we get 
  \begin{align}
  |Z_4|=|\sum_\ell \zeta_{4\ell}|\prec \sqrt{\frac{n}{M^2}\sum_{i<\ell}  \sum_{s<t}  |b_{(i\ell)(st)}|^2}. \label{17102833} 
  \end{align}
  Combining (\ref{17102830}), (\ref{17102831}), (\ref{17102832}) and (\ref{17102833}) finally yields (\ref{17101052}).
  
Hence, we conclude the proof of Proposition \ref{lem. large deviation nonlinear part}. 

\end{proof}

\begin{proof}[Proof of Corollary \ref{cor. large devi}] The results in Corollary \ref{cor. large devi} follow from Propositions \ref{lem. large deviation linear part} and \ref{lem. large deviation nonlinear part}, (\ref{171128230}), and also the fact 
\begin{align}
&\frac{1}{M^2}\sum_{\ell=1}^n \Big|\sum_{j=\ell+1}^n (TB)_{j,(\ell j)}\Big|^2\leq \frac{n}{M^2} \sum_{\ell<j}  \big|(TB)_{j,(\ell j)}^2\big|\nonumber\\
&\leq  \frac{n}{M^2} \mr{Tr} B^*T'TB=  \frac{n}{3M^2} \mr{Tr} B^*\Gamma B\leq \frac{C}{M} \mr{Tr} |B|^2.   \label{17113040}
\end{align}
This completes the proof of Corollary \ref{cor. large devi}. 
\end{proof}

\noindent {\bf S2: Proof of the strong local law}

In this section, we state the proof of Proposition \ref{lem.weak law for R}. 
We first  introduce the notation
\begin{align*}
{\Lambda}_{\rm d}^c\equiv {\Lambda}_{\rm d}^c(z):= \max_{k}|G_{kk}(z)-{m}(z)|.
\end{align*}
Let $\Theta^{(i)}$ be the submatrix of $\Theta$ with the $i$-th row $\mb{v}_i$ removed. We also denote by  $K^{(i)}=\Theta^{(i)}(\Theta^{(i)})'$ and $\mathcal{K}^{(i)}=(\Theta^{(i)})'\Theta^{(i)}$ the submatrices. 
Correspondingly, we further denote by $G^{(i)}(z):= (K^{(i)}-z)^{-1}$ and $\mathcal{G}^{(i)}(z):=(\mathcal{K}^{(i)}-z)^{-1}$ their Green functions. 
Analogously, we use the notation $\Theta^{(ij)}$ to denote the submatrix of $\Theta$ with both the $i$-th and $j$-th rows removed for $i\neq j$. Correspondingly, we can define the notations $K^{(ij)}$, $\mc{K}^{(ij)}$, $G^{(ij)}$ and $\mc{G}^{(ij)}$. 
We also use ${m} ^{(i)}(z)$ and  ${m} ^{(ij)}(z)$ to represent the Stieltijes transforms of $K^{(i)}$ and $K^{(ij)}$, respectively.

\begin{proof}[Proof of Proposition \ref{lem.weak law for R}]  With the aid of the large deviation estimates in Corollary \ref{cor. large devi}, the proof of Proposition \ref{lem.weak law for R} can be done with the aid of the general proof strategy in \cite{PY}. Nevertheless, due to the different dependence structure within the rows of $\Theta$, the proof still differs  in  many technical details.  Hence, in the sequel, we  state the proof in a sketchy way with a highlight on the  parts different from \cite{PY}. In addition, as mentioned above, the statements in \cite{PY} are given in a more quantitative way, especially on the control of the high probability of events. Here, instead, we employ the notation $\prec$ defined in  
Definition \ref{definition of stochastic domination} for the high probability estimates.  But this difference is not essential for the proof. 

 We first fix a $z\in {\mc{D}}(\varepsilon)$ and assume that the following a priori bounds hold
\begin{align}
{\Lambda}_{ \rm d}(z)\prec n^{-\frac{\varepsilon}{10}}, \qquad {\Lambda}_{ \rm o}(z)\prec n^{-\frac{\varepsilon}{10}}.  \label{17102936}
\end{align}
Under the additional assumption (\ref{17102936}), we also have 
\begin{align}
G_{ii}(z)\sim 1, \quad m(z)\sim 1 \label{171130100}
\end{align}
with high probability, in light of  (\ref{17113020}).  We then further define a stochastic control parameter
\begin{align}
\Pi(z):=\sqrt{\frac{\Im \underline{m} (z)+ {\Lambda} (z)}{n\eta}}+\frac{1}{n\eta}.  \label{17111530}
\end{align}

Our first task is to show that 
\begin{align}
{\Lambda}_{\rm d}^c (z)\prec \Pi(z), \qquad {\Lambda}_{\rm o} (z)\prec \Pi(z)  \label{17111401}
\end{align}
under the additional assumption (\ref{17102936}). 

By Schur complement,  we have 
\begin{align}
G_{kk}=\frac{1}{\mb{v}_k\mb{v}_k'- z-\mb{v}_k (\Theta^{(k)})'G^{(k)} \Theta^{(k)}\mb{v}_k'}=:-\frac{1}{1- z-\mb{v}_k B^{(k)}\mb{v}_k'}, \label{17091701}
\end{align}
where in the last step we use the fact $\mb{v}_k\mb{v}_k'=1$ and introduce the notation $B^{(k)}=(\Theta^{(k)})'G^{(k)} \Theta^{(k)}$.  Applying (\ref{17113010}), we have 
\begin{align}
&\Big|\mb{v}_kB^{(k)}\mb{v}_k'- \frac{1}{M} \mr{Tr}  B^{(k)}\wt{\Gamma}\Big|
\prec  \sqrt{\frac{\mathrm{Tr} |B^{(k)}|^2}{M}}. \label{17102931}
\end{align}
   Further, we observe  that 
\begin{align}
B^{(k)}=(\Theta^{(k)})'  \Theta^{(k)}\mc{G}^{(k)}=\mc{K}^{(k)} \mc{G}^{(k)}=I_M+z\mc{G}^{(k)}.  \label{171029401} 
\end{align}
Hence, we have 
\begin{align}
\mathrm{Tr} |B^{(k)}|^2&= \sum_{i=1}^{p-1}\Big|1+\frac{z}{\lambda_i^{(k)}-z}\Big|^2=\sum_{i=1}^{p-1}\Big(1+\frac{z}{\lambda_i^{(k)}-z}+\frac{\bar{z}}{\lambda_i^{(k)}-\bar{z}}+\frac{|z|^2}{|\lambda_i^{(k)}-z|^2}\Big)\nonumber\\
 &=(p-1)\Big(1+ z {m} ^{(k)}+\bar{z}\overline{{m} ^{(k)}}(z)+\frac{|z|^2}{\eta} \Im {m} ^{(k)}(z)\Big),  \label{17102930}
\end{align}
where we use $\lambda_i^{(k)}, i=1, \ldots, p-1$ to denote the $p-1$ nontrivial eigenvalues of  $\mc{K}^{(k)}$, which are also the eigenvalues of $K^{(k)}$.  Plugging (\ref{17102930}) into (\ref{17102931}) yields
\begin{align}
\Big|\mb{v}_k B^{(k)}\mb{v}_k'- \frac{1}{M} \mr{Tr}  B^{(k)}\wt{\Gamma}\Big|\prec\sqrt{\frac{\Im {m} ^{(k)}(z) }{n\eta }+\frac{1}{n}}\prec \Pi(z), \label{17102940}
\end{align}
where in the last step we use the fact $\mr{Tr} G^{(k)}=\mr{Tr} G+O(\frac{1}{\eta})$ (c.f. Lemma \ref{17101810}), and also (\ref{171130100}). 
We can then conclude from  (\ref{17091701}) and (\ref{17102940}) that 
\begin{align}
G_{kk}= \frac{1}{1-z-\frac{1}{M} \mr{Tr} B^{(k)} \wt{\Gamma}+O_\prec(\Pi)}.  \label{17103110}
\end{align}
Let $\sum_{\ell}^{(k)}$ denote the sum over $\ell\in \llbracket 1, p\rrbracket\setminus\{k\}$. We can further write 
\begin{align}
\frac{1}{M} \mr{Tr}  B^{(k)}\wt{\Gamma}&=\frac{1}{M} \mr{Tr}  (\Theta^{(k)})'\Theta^{(k)}\mc{G}^{(k)} \wt{\Gamma}\nonumber\\
&= \frac{1}{M}\sum_{\ell}^{(k)}\mb{v}_{\ell}\mc{G}^{(k)} \wt{\Gamma}\mb{v}_{\ell}'=\frac{1}{M}\sum_{\ell}^{(k)} \frac{\mb{v}_\ell \mathcal{G}^{(k\ell)}\wt{\Gamma}\mb{v}_\ell '}{1+\mb{v}_\ell  \mathcal{G}^{(k\ell)}\mb{v}_\ell '}\nonumber\\
&=\frac{1}{M}\sum_{\ell}^{(k)} \frac{\mb{v}_\ell B^{(k\ell)}\wt{\Gamma}\mb{v}_\ell '-\mb{v}_\ell\wt{\Gamma}\mb{v}_\ell '}{z+\mb{v}_\ell B^{(k\ell)}\mb{v}_\ell '- 1}, \label{17102946}
\end{align}
where we  use Sherman-Morrison formula in the third step, and introduce the matrix $B^{(k\ell)}=(\Theta^{(k\ell)})' G^{(k\ell)}\Theta^{(k\ell)}$ which satisfies  the identity 
\begin{align}
z\mathcal{G}^{(k\ell)}=B^{(k\ell)}-I_M. \label{17111510}
\end{align}  
 Similarly to (\ref{17102931}), we can again apply  (\ref{17113010}) to get 
\begin{align}
&\Big|\mb{v}_\ell B^{(k\ell)}\wt{\Gamma}\mb{v}_\ell '-\frac{1}{M} \mr{Tr}  B^{(k\ell)}\wt{\Gamma}^2\Big|\prec  \sqrt{\frac{\mathrm{Tr} |B^{(k\ell)}\wt{\Gamma}|^2}{M}}\nonumber\\
&\prec  \sqrt{{\mathrm{Tr} |B^{(k\ell)}|^2}}\prec n\sqrt{\frac{\Im {m} ^{(k\ell)}(z) }{n\eta }+\frac{1}{n}}\prec n\Pi(z),
\label{17102944}
\end{align}
where the last two steps can be shown similarly to  (\ref{17102940}). Using (\ref{17113010})  with $B=\wt{\Gamma}$, we have 
\begin{align}
|\mb{v}_\ell\wt{\Gamma}\mb{v}_\ell '-\frac{1}{M}\mr{Tr} \wt{\Gamma}^2|\prec \sqrt{\frac{\mathrm{Tr} \wt{\Gamma}^2}{M}}. \label{17113058}
\end{align}
Observe from (\ref{17092601}) and (\ref{17113050}) that 
\begin{align}
  \wt{\Gamma}^2= \frac{1}{3}(n+2)\Gamma+\frac{1}{9}I_M.  \label{17113057}
\end{align}
In addition, from the definition of $\Gamma$ in (\ref{17102201}) we see that $\mr{Tr} \Gamma=\frac23 M$.  Plugging this fact together with (\ref{17113057}) into (\ref{17113058}) yields the bound  
\begin{align}
|\mb{v}_\ell\wt{\Gamma}\mb{v}_\ell '-\frac{1}{M}\mr{Tr} \wt{\Gamma}^2| \prec  \sqrt{n}.  \label{17102943}
\end{align}
Then, plugging the estimates (\ref{17102940}), (\ref{17102944}) and (\ref{17102943}) into  (\ref{17102946}) yields the estimate 
\begin{align}
\frac{1}{M} \mr{Tr}  B^{(k)}\wt{\Gamma}=  \frac{2}{n-1}\sum_{\ell}^{(k)} \frac{\frac{1}{Mn} \mr{Tr}  B^{(k\ell)}\wt{\Gamma}^2-\frac{1}{Mn} \mr{Tr} \wt{\Gamma}^2+O_\prec(\Pi)}{z-1+\frac{1}{M} \mr{Tr}  B^{(k\ell)}\wt{\Gamma}+O_\prec(\Pi)}. \label{17103101}
\end{align}
It is elementary to check from (\ref{17092601}) and (\ref{17113050}) that 
\begin{align*}
\wt{\Gamma}^2=-\frac{n+1}{9} I_M+\frac{n+2}{3}\wt{\Gamma}.
\end{align*}

For brevity, we further denote by 
\begin{align}
 &{m}^{(\Bbbk)}_\Gamma:=\frac{1}{M} \mr{Tr}  B^{(\Bbbk)}\wt{\Gamma}= z\frac{1}{M} \mr{Tr}  \mathcal{G}^{(\Bbbk)}\wt{\Gamma}-\frac{1}{M} \mr{Tr}\wt{\Gamma}, \nonumber\\
 & {m} ^{(\Bbbk)}_I:= \frac{1}{M} \mr{Tr}  B^{(\Bbbk)}=z\frac{1}{M} \mr{Tr}  \mathcal{G}^{(\Bbbk)}-1, \label{18021101}
\end{align}
where $\Bbbk=\emptyset, \{k\}$, or $\{k,\ell\}$. 
Consequently, we can rewrite (\ref{17103101}) as
\begin{align}
{m} ^{(k)}_{\Gamma} &=  \frac{2}{n-1}\sum_{\ell}^{(k)} \frac{-\frac{n+1}{9n}{m} ^{(k\ell)}_I+\frac{n+2}{3n} {m} ^{(k\ell)}_\Gamma-\frac{2n+5}{9n}+O_\prec(\Pi)}{z-1+{m} ^{(k\ell)}_\Gamma+O_\prec(\Pi)}. \label{18021110}
\end{align}
From (\ref{18021101}), (\ref{091002}) and the fact $\|\wt{\Gamma}\|=O(n)$ (c.f. (\ref{17110901})), we also have 
\begin{align*}
\|{m} ^{(\Bbbk)}_{\Gamma} -{m}_{\Gamma} \|=O(\frac{1}{n\eta}), \qquad \|{m} ^{(\Bbbk)}_{I} -{m}_{I} \|=O(\frac{1}{M\eta}), \qquad \Bbbk=\{k\}, \text{ or } \{k\ell\}.
\end{align*}
This together with  the fact $m_I=\frac{1}{M}\mr{Tr}B=O_\prec(\frac{1}{n})$, (\ref{18021110}) and (\ref{17103110}) further implies that 
\begin{align}
{m} _\Gamma =   \frac{\frac{2}{3}c_n {m} _\Gamma-\frac{4}{9}c_n+O_\prec(\Pi)}{z-1+{m} _\Gamma+O_\prec(\Pi)},
\label{17111571}
\end{align} 
and 
\begin{align}
G_{kk}=  \frac{1}{1-z-m_\Gamma+O_\prec(\Pi)}. \label{18021111}
\end{align}
Then, (\ref{18021111}) and the a priori bound (\ref{171130100}) implies that 
\begin{align}
1-z-m_\Gamma\sim 1, \qquad m_\Gamma\sim 1 \label{18021114}
\end{align}
with high probability.  Plugging (\ref{18021114}) back into (\ref{17111571}) and (\ref{18021114}), we arrive at the equations
\begin{align}
{m} _\Gamma^2+(z-1-\frac23c_n) {m} _\Gamma+\frac49 c_n=O_\prec(\Pi) \label{17103111}
\end{align}
and 
\begin{align}
{m} =\frac{1}{1-z-{m} _\Gamma}+O_\prec(\Pi). \label{17103112}
\end{align}
Substituting (\ref{17103112}) back into (\ref{18021111}) and using (\ref{18021114}) give the first estimate in (\ref{17111401}). In addition, from (\ref{17103111}) and (\ref{17103112}), we can also get the following  equation  for ${m} $:
\begin{align}
\frac23 c_n(z-\frac13){m} ^2+(z-1+\frac23 c_n){m} +1=O_\prec(\Pi).  \label{17111560}
\end{align}

Next, we prove the second estimate in (\ref{17111401}).  To this end, we need Lemma \ref{lem.resolvent identity}. First, combining (\ref{17111402}) with (\ref{17111403}) yields
\begin{align}
G_{ij}=  z \big( G_{ii}(z)G_{jj}(z)-G_{ji}(z)G_{ij}(z)\big) \mb{v}_i \mc{G}^{(ij)}(z) \mb{v}_j', \qquad i\neq j.  \label{17111540}
\end{align}
According to (\ref{17111510}), we can write 
\begin{align}
\mb{v}_i \mc{G}^{(ij)}(z) \mb{v}_j'= z^{-1}\mb{v}_i B^{(ij)}(z) \mb{v}_j'- z^{-1}  \mb{v}_i \mb{v}_j'.  \label{17111515}
\end{align}
Also observe that $\mb{v}_i$ and $\mb{v}_j$ are independent if $i\neq j$. Hence, using (\ref{17113070}) twice we get 
\begin{align}
\big|\mb{v}_i B^{(ij)}(z) \mb{v}_j'| &\prec  \sqrt{\frac{\|B^{(ij)}(z) \mb{v}_j'\|^2}{n}}= \sqrt{\frac{\sum_{k} (\mb{e}_kB^{(ij)}(z) \mb{v}_j')^2}{n}}\nonumber\\
& \prec \sqrt{\frac{\sum_{k,\ell} (\mb{e}_kB^{(ij)}(z) \mb{e}_\ell)^2}{n^2}}\prec \sqrt{\frac{\mr{Tr} (B^{(ij)})^2}{M}}\prec \Pi(z),
\label{17111536}
\end{align}
where the last step follows from the last line of (\ref{17102944}). 
For the second term in the RHS of (\ref{17111515}), using (\ref{17113070}) we have 
\begin{align}
 |\mb{v}_i \mb{v}_j'|\prec \sqrt{\frac{\|\mb{v}_j\|^2}{n}}=\frac{1}{\sqrt{n}}\prec \Pi(z),  \label{17111537}
\end{align}
where the last step follows from the definition of $\Pi(z)$  (c.f. (\ref{17111530})) and the fact that $\Im \underline{m} (z)\gtrsim \eta$ (c.f. (\ref{17113021})).
Plugging (\ref{17111536}) and (\ref{17111537}) into (\ref{17111515}) yields the bound  $|\mb{v}_i \mc{G}^{(ij)}(z) \mb{v}_j'|\prec \Pi(z)$. This together with  (\ref{17111540}), the a priori bounds in (\ref{17102936}) and also (\ref{171130100}), further implies  (\ref{17111401}). 

Next, we show that (\ref{17111560}) can be improved to 
\begin{align}
\frac23 c_n(z-\frac13){m} ^2+(z-1+\frac23 c_n){m} +1=O_\prec(\wh{\Pi}^2) \label{171116113} 
\end{align}
for any control parameter $\wh{\Pi}\equiv \wh{\Pi}(z)$ which satisfies $\Pi(z)\prec \wh{\Pi}(z)$.

To this end, roughly speaking, we need to  improve the error term in both (\ref{17111571}) and 
(\ref{17103112}) from $\Pi$ to $\Pi^2$.  This is achieved through a general fluctuation averaging mechanism in \cite{PY} (see Lemmas 7.3 and 7.4 therein).  We first introduce the following notations
\begin{align*}
&\mc{Z}_{1,k}:= \mb{v}_k B^{(k)}\mb{v}_k'- \frac{1}{M} \mr{Tr}  B^{(k)}\wt{\Gamma}, \nonumber\\
&\mc{Z}_{2,k}:= \frac{1}{n}\mb{v}_kB^{(k)}\wt{\Gamma}\mb{v}_k '- \frac{1}{Mn} \mr{Tr}  B^{(k)}\wt{\Gamma}^2,\nonumber\\
&\mc{Z}_{3, \ell}:=\frac{1}{n}\mb{v}_k\wt{\Gamma}\mb{v}_k '-\frac{1}{Mn}\mr{Tr} \wt{\Gamma}^2.\nonumber\\
\end{align*}
 We have the following 
fluctuation averaging estimates. 
\begin{lem}\label{lem. fluctuation averaging} Suppose that  the a priori bound (\ref{17102936}) holds. Let $\wh{\Pi}\equiv \wh{\Pi}(z)$ be any deterministic control parameter which satisfies $\Pi(z)\prec \wh{\Pi}(z)$. We have 
\begin{align}
\frac{1}{p}\sum_k \mc{Z}_{a,k}=O_\prec(\wh{\Pi}^2), \qquad a=1,2,3. \label{17111601}
\end{align}
\end{lem}
\begin{proof}[Proof of Lemma \ref{lem. fluctuation averaging}] First, the proof of (\ref{17111601}) for $a=3$ is elementary, since it follows from the large deviation of the sum of independent variables directly (c.f. Corollary B.3 of  \cite{EYY12} for instance). 

The proof of (\ref{17111601}) for $a=1,2$ can be done very similarly to the counterpart in \cite{PY}. Hence, we only sketch some necessary changes below, without repeating the tedious argument. For $a=1$, by the identity $z\mathcal{G}^{(k)}=B^{(k)}-I_M$, and (\ref{17111601}) for $a=3$, it suffices to show that 
\begin{align}
\frac{1}{p}\sum_{k} \big(\mb{v}_k \mc{G}^{(k)}\mb{v}_k'- \frac{1}{M} \mr{Tr}  \mc{G}^{(k)}\wt{\Gamma}\big)=O_\prec(\wh{\Pi}^2). \label{171202110}
\end{align}
By  (\ref{171202100}), it suffices to show that 
\begin{align}
\frac{1}{p}\sum_{k}(\mathrm{Id}-\mathbb{E}_k) (\frac{1}{G_{kk}})=O_\prec(\wh{\Pi}^2),  \label{171202103}
\end{align}
where we use $\mathbb{E}_k$ to denote the expectation w.r.t. $\mb{v}_k$.  The proof of (\ref{171202103}) can be done in the same way as  that for Lemma 7.4 in \cite{PY}, by keeping using the expansion in (\ref{17111403}) and the smallness of the off-diagonal entries $(G_{ij})$'s (c.f. (\ref{17111705})). We thus omit the details.

For $a=2$, similarly to (\ref{171202110}), one can instead prove 
\begin{align}
\frac{1}{p}\sum_{k} \big(\frac{1}{n}\mb{v}_k \mc{G}^{(k)}\wt{\Gamma}\mb{v}_k'- \frac{1}{Mn} \mr{Tr}  \mc{G}^{(k)}\wt{\Gamma}^2\big)
=\frac{1}{p}\sum_{k} (\mr{Id}-\mathbb{E}_k)\big(\frac{1}{n}\mb{v}_k \mc{G}^{(k)}\wt{\Gamma}\mb{v}_k'\big)=O_\prec(\wh{\Pi}^2). \label{171202130}
\end{align}
We observe that 
\begin{align}
\frac{1}{n}\mb{v}_k \mc{G}^{(k)}\wt{\Gamma}\mb{v}_k' &= \frac{1}{n}\mb{v}_k \mc{G}^{(k\ell)}\wt{\Gamma}\mb{v}_k'-\frac{1}{n}\frac{\mb{v}_k \mc{G}^{(k\ell)}\mb{v}_\ell'\mb{v}_\ell\mc{G}^{(k\ell)}\wt{\Gamma}\mb{v}_k'}{1+\mb{v}_\ell\mc{G}^{(k\ell)}\mb{v}_\ell'}\nonumber\\
&= \frac{1}{n} \mb{v}_k \mc{G}^{(k\ell)}\wt{\Gamma}\mb{v}_k'+\frac{G_{k\ell}}{G_{kk}} \big(\frac{1}{n}\mb{v}_\ell\mc{G}^{(k\ell)}\wt{\Gamma}\mb{v}_k'\big),\label{171202136}
\end{align}
where in the second step we use (\ref{17111402}) and (\ref{171202100}).  Using the expansion (\ref{171202136}) instead of (\ref{17111403}) and using the smallness of both of the off diagonal entries $G_{ij}$'s and also the smallness of the factor of the form $\frac{1}{n}\mb{v}_\ell\mc{G}^{(k\ell)}\wt{\Gamma}\mb{v}_k'$ with $\ell\neq k$, one can prove (\ref{171202130}) similarly to (\ref{171202110}). We thus omit the details. 

This completes the proof of Lemma \ref{lem. fluctuation averaging}. 
\end{proof}

Now, with the improved bounds in (\ref{17111601}), we proceed to the proof of  (\ref{17111663}).  We first rewrite (\ref{17091701}) as
\begin{align}
G_{kk}=-\frac{1}{1- z-\mb{v}_k B^{(k)}\mb{v}_k'} &= -\frac{1}{1- z-\frac{1}{M} \mr{Tr} B^{(k)}\wt{\Gamma}-\mc{Z}_{1,k}}. \label{171116101}
\end{align}
Note that 
\begin{align}
&\Big|\frac{1}{M} \mr{Tr} B^{(k)}\wt{\Gamma}- \frac{1}{M} \mr{Tr} B\wt{\Gamma}\Big| =\Big|z\frac{1}{M} \frac{\mb{v}_k \mc{G}^{(k)}\wt{\Gamma}\mc{G}^{(k)}\mb{v}_k'}{1+\mb{v}_k \mc{G}^{(k)}\mb{v}_k'}\Big|\nonumber\\
&\prec \frac{1}{n} \frac{\mb{v}_k |\mc{G}^{(k)}|^2 \mb{v}_k'}{|1+\mb{v}_k \mc{G}^{(k)}\mb{v}_k'|}= \frac{1}{n\eta} \frac{\Im  \mb{v}_k \mc{G}^{(k)} \mb{v}_k'}{|1+\mb{v}_k \mc{G}^{(k)}\mb{v}_k'|}. \label{17111698}
\end{align}
From (\ref{17102931}) and (\ref{17102940}), we also have 
\begin{align*}
\mb{v}_k \mc{G}^{(k)} \mb{v}_k' &=  \frac{1}{M} \mr{Tr}  \mc{G}^{(k)}\wt{\Gamma}+O_\prec(\Pi)=z^{-1}\frac{1}{M} \mr{Tr} B^{(k)}\wt{\Gamma}-z^{-1}+O_\prec(\Pi)\nonumber\\
&= z^{-1} m_\Gamma-z^{-1}+ O_\prec(\Pi). 
\end{align*}
Hence, we have 
\begin{align*}
\frac{1}{|1+\mb{v}_k \mc{G}^{(k)}\mb{v}_k'|}=\frac{|z|}{|1-z-m_\Gamma+O_\prec(\Pi)|}=|zm(z)|+O_\prec(\Pi)\prec 1,
\end{align*}
where we use (\ref{17103112}) and (\ref{171130100}). 
Moreover, we also have 
\begin{align}
\Im \mb{v}_k \mc{G}^{(k)} \mb{v}_k'=z^{-1} \Im \frac{1}{M} \mr{Tr} B^{(k)}\wt{\Gamma}+O_\prec(\eta)+O_\prec(\Pi).  \label{171111691}
\end{align}
Further, from (\ref{17103112}) we also have 
\begin{align}
\big|\Im \frac{1}{M} \mr{Tr} B^{(k)}\wt{\Gamma}\big|= \big|\Im \frac{1}{M} \mr{Tr} B\wt{\Gamma}+O_\prec(\Pi)\big|\prec   \Im {m} +\eta+\Pi. \label{17111690}
\end{align}
Substituting (\ref{17111690}) into (\ref{171111691}) yields
\begin{align}
|\Im \mb{v}_k \mc{G}^{(k)} \mb{v}_k'| \prec \Im {m} +\eta+\Pi\prec \Im \underline{m} +\Lambda+\Pi, \label{17111697}
\end{align}
where we also use the fact $\Im \underline{m} \gtrsim \eta$.  Plugging (\ref{17111697}) into (\ref{17111698}) we get 
\begin{align}
\Big|\frac{1}{M} \mr{Tr} B^{(k)}\wt{\Gamma}- \frac{1}{M} \mr{Tr} B\wt{\Gamma}\Big|\prec \Pi^2.  \label{171116100}
\end{align}
Hence, from (\ref{171116101}) and (\ref{171116100}) we get 
\begin{align*}
G_{kk}&=-\frac{1}{1- z-\frac{1}{M} \mr{Tr} B\wt{\Gamma}-\mc{Z}_{1,k}+O_\prec(\Pi^2)}=  \frac{1}{-1+ z+{m} _{\Gamma}+\mc{Z}_{1,k}}+O_\prec(\Pi^2).
\end{align*}
Then taking the average of $G_{kk}$ over $k$ and using (\ref{17111601}) for $a=1$, we obtain 
\begin{align}
{m} = \frac{1}{-1+ z+{m} _{\Gamma}}+ O_\prec(\wh{\Pi}^2). \label{171116112} 
\end{align}
Similarly, applying (\ref{17111601}) we can also improve (\ref{17103111}) to  
\begin{align}
{m} _\Gamma^2+(z-1-\frac23c_n) {m} _\Gamma+\frac49 c_n=O_\prec(\wh{\Pi}^2). \label{171116111} 
\end{align}
Combining (\ref{171116112}) and  (\ref{171116111}) we can further get (\ref{171116113}).

Now, we obtain (\ref{17111401}) and (\ref{171116113}) with the aid of the additional input (\ref{17102936}), for a fixed $z\in \mc{D}(\varepsilon)$.  To prove  (\ref{17111705}) and (\ref{17111663}), one needs to go through a standard continuity argument, starting from $\eta\geq 1$ and reducing $\eta$ to $\eta=n^{-1+\varepsilon}$ step by step, with a step size $n^{-3}$ (say). The whole continuity argument is completely the same as the counterpart of the sample covariance matrices in \cite{PY}, although the notation $\prec$ was not used therein.  We thus omit this argument and conclude (\ref{17111401}) and (\ref{171116113}).

Finally, for (\ref{17111701}), it is well understood now (c.f. \cite{PY})  that (\ref{17111701}) will follow from (\ref{17111663}) and (\ref{171116113}), if one can additionally show a crude upper bound 
\begin{align}
\lambda_1(K)\prec 1.  \label{17111720}
\end{align}
A proof of (\ref{17111720}) is given at the end of this section. We remark here in  \cite{PY}, a slightly stronger crude upper bound was used, namely, with high probability the largest eigenvalue is bounded by some large (but independent of $n$) positive constant $C$. In order to use such a bound, one need to extend the local law to a larger domain to include $E=C$, where $E=\Re z$. Here, in  (\ref{17111720}), we have a weaker crude upper bound, namely, with high probability, $\lambda_1(K)\leq n^{{\epsilon}}$ for any tiny constant ${\epsilon}>0$. In order to use such a bound to further get (\ref{17111701}), we need to extend our local law  from ${\mathcal{D}}(\epsilon)$ to a larger domain: $\wt{\mathcal{D}}(\epsilon):=\big\{z=E+\ii \eta: \frac{1}{2}{\lambda}_{+,c}\leq E\leq n^{\frac{\epsilon}{10}}, n^{-1+\epsilon}\leq \eta\leq 1 \big\}$ (say). For sufficiently small $\epsilon$, the proof of the local law, i.e., Proposition \ref{lem.weak law for R} (i), (ii), on $\wt{\mathcal{D}}(\epsilon)$, does not require any essential change on the proof on the smaller domain  ${\mathcal{D}}(\epsilon)$. Therefore, we complete the proof of Proposition \ref{lem.weak law for R}.

\end{proof}

In the sequel, we prove the estimate (\ref{17111720}). 
\begin{proof}[Proof of (\ref{17111720})]
We first write 
\begin{align*}
K=UU'+U\bar{V}'+\bar{V}U'+\bar{V}\bar{V}',
\end{align*}
according to Hoeffding decomposition, where $U$ and $\bar{V}$ are defined in (\ref{171130101}).  From (\ref{171104100}), we known that $\|U\|\prec 1$. Hence, it suffices to show that $\|\bar{V}\bar{V}'\|\prec 1$ which 
is equivalent to $\|\bar{V}\|\prec 1$.  To this end, we observe that 
\begin{align}
&(\bar{V}\bar{V}')_{kk}=\frac{1}{M}\sum_{i<j} \bar{v}_{k,(ij)}^2=O(1), \nonumber\\
&(\bar{V}\bar{V}')_{k\ell}=  \frac{1}{M}\sum_{i<j} \bar{v}_{k,(ij)} \bar{v}_{\ell,(ij)}\prec \sqrt{\frac{\|\bar{\mathbf{v}}_\ell\|^2}{M}}=O(\frac{1}{n}), \qquad k\neq \ell \label{18021850}
\end{align}
where in the second inequality we use (\ref{17101050}). Hence, $\bar{V}\bar{V}$ is a $p\times p$ matrix whose diagonal entries are order $1$ and the off-diagonal entries are $O_\prec(\frac{1}{n})$. For a rectangular matrix $A=(a_{ij})_{\mathsf{N},\mathsf{M}}$, let $\|A\|_1=\max_{1\leq j\leq \mathsf{M}} \sum_{i=1}^{\mathsf{N}}|a_{ij}|$ and $\|A\|_\infty=\max_{1\leq i\leq \mathsf{N}} \sum_{j=1}^{\mathsf{M}}|a_{ij}|$ be its 
$1$-norm and $\infty$-norm, respectively.   Then by H\"{o}lder's inequality for the matrix norm $\|A\|\leq \sqrt{\|A\|_1\|A\|_\infty}$, we can get  from (\ref{18021850}) the bound $\|\bar{V}\bar{V}'\|\prec 1$.  This concludes the proof.
\end{proof}

\noindent {\bf S3: Proofs of some other lemmas}

In this section, we state the proofs of Lemmas \ref{lem. local law for K gamma}, \ref{lem. technical estimates} and \ref{lem. local law for K hat t}. We also state the proof of the last estimate in (\ref{17100901}) at the end of this section.

\begin{proof}[Proof of Lemma \ref{lem. local law for K gamma}] The proof of  Proposition  \ref{lem.weak law for R} only relies on the large deviation results in Propositions \ref{lem. large deviation linear part} and \ref{lem. large deviation nonlinear part}. It suffices to check that Proposition \ref{lem. large deviation nonlinear part} still holds if one replaces $\bar{\mb{v}}_k$ by $\mathbf{h}_k$, where $\mathbf{h}_k$ represents the $k$-th row of $H$.  In light of (\ref{17102202}) and the fact that $\mb{h}_k$ has $i.i.d.$ normal entries, it is easy to check that the results in Proposition \ref{lem. large deviation nonlinear part} are still valid for $\mb{h}_k$ instead of $\bar{\mb{v}}_k$, using the large deviation estimates for independent random variables ((c.f. Corollary B.3 of  \cite{EYY12} for instance)). Actually, the counterparts of (\ref{17101051}) and (\ref{17101052}) are even sharper in the case of $\mb{h}_k$ instead of $\bar{\mb{v}}_k$. Hence, we  complete the proof of  Lemma \ref{lem. local law for K gamma}. 
\end{proof}

\begin{proof}[Proof of Lemma \ref{lem. technical estimates}]
Recall the definition of $A_\gamma$ and $B_\gamma$ from (\ref{17110230}), and also set 
\begin{align*}
\mc{K}_\gamma^{(\gamma)}= (\Theta_\gamma^{(\gamma)})'(\Theta_\gamma^{(\gamma)}), \qquad \mc{G}_\gamma^{(\gamma)}=(\mc{K}_\gamma^{(\gamma)}-z)^{-1}. 
\end{align*}   
Similarly to (\ref{17102930}), we have 
\begin{align}
&\mr{Tr} |B_\gamma|^2=\mr{Tr}|I+z\mc{G}_\gamma^{(\gamma)}|^2\nonumber\\
&\qquad= (p-1) \Big(1+ z m^{(\gamma)}_\gamma+\bar{z}\overline{m^{(\gamma)}_\gamma}(z)+\frac{|z|^2}{\eta} \Im m^{(\gamma)}_\gamma(z)\Big)=O_\prec(n^{\frac43+\varepsilon}), \label{17110241}
\end{align}
where the last step follows from  Lemma \ref{lem. local law for K gamma},  Lemma \ref{lem.properties of m} and the fact $|m_\gamma^{(\gamma)}-m_\gamma|\leq \frac{1}{n\eta}$.   
Similarly, we have 
\begin{align}
\mr{Tr} |A_\gamma|^2= \mr{Tr} |(\Theta_\gamma^{(\gamma)})' (G_\gamma^{(\gamma)})^2\Theta_\gamma^{(\gamma)}|^2=\mr{Tr} \big| \mc{G}_\gamma^{(\gamma)}+z(\mc{G}_\gamma^{(\gamma)})^2\big|^2\nonumber\\
\leq \frac{1}{\eta^2} \mr{Tr} \big| I+z\mc{G}_\gamma^{(\gamma)}\big|^2=O_\prec(n^{\frac83+\varepsilon}), \label{17110246}
\end{align}
where in the last step we use (\ref{17110241}).  In addition, we also have 
\begin{align}
&\mr{Tr} B_\gamma=\mr{Tr}(I+z\mc{G}_\gamma^{(\gamma)})=(p-1)(1+zm_\gamma^{(\gamma)})=O_\prec(n),\nonumber\\
&\mr{Tr} A_\gamma= \mr{Tr}(\mc{G}_\gamma^{(\gamma)}+z(\mc{G}_\gamma^{(\gamma)})^2)=O_\prec(n^{\frac{4}{3}+\varepsilon}), \nonumber\\
& \mr{Tr} |A_\gamma|= \mr{Tr}|\mc{G}_\gamma^{(\gamma)}+z(\mc{G}_\gamma^{(\gamma)})^2|=O_\prec(n^{\frac{4}{3}+\varepsilon}).\label{17110250}  
\end{align}
From the local law in Lemma \ref{lem. local law for K gamma}, it is easy to show that 
\begin{align}
\mr{Tr} |B_\gamma|=\mr{Tr}|I+z\mc{G}_\gamma^{(\gamma)}|\prec n. \label{18021801}
\end{align}
For instance, we refer to Lemma 3.10 of  \cite{BPZunpub} and its proof  for a detailed argument on the derivation of the above bound from the local law.

Applying Propositions \ref{lem. large deviation linear part} and \ref{lem. large deviation nonlinear part}, (\ref{17110241}), (\ref{17110246}), (\ref{17110250}) and the fact $\|\Gamma\|= O(n)$, we see that 
\begin{align*}
&|\mc{U}_\gamma|\prec \sqrt{\frac{\mathrm{Tr} |B_\gamma|^2}{M}}\prec    n^{-\frac13+\varepsilon}, \quad |\mc{V}_\gamma|\prec \frac{1}{M} |\mr{Tr } B_\gamma| +\sqrt{\frac{n}{M^2} \mr{Tr} |B_\gamma|^2} \prec n^{-\frac56+\varepsilon}.
\end{align*}
Then, using the large deviation of the independent random variables ((c.f. Corollary B.3 of  \cite{EYY12} for instance)), it is easy to check 
\begin{align*}
& |\hat{\mc{P}}_\gamma(z)|\prec  \sqrt{\frac{\mathrm{Tr} |TB_\gamma|^2}{M^2}}=  \sqrt{\frac{\mathrm{Tr} B_\gamma\Gamma B_\gamma^*}{M^2}}\prec  \sqrt{\frac{n\mathrm{Tr} |B_\gamma|^2 }{M^2}}\prec n^{-\frac56+\varepsilon}, \nonumber\\
 &|\hat{\mc{V}}_\gamma|\prec \frac{1}{M} |\mr{Tr } B_\gamma| +\sqrt{\frac{1}{M^2} \mr{Tr} |B_\gamma|^2} \prec n^{-1+\varepsilon},\nonumber\\
& |\hat{\mc{W}}_\gamma| \prec n^{-\frac12+\varepsilon}, \qquad   |\hat{\mc{O}}_\gamma|\prec n^{-1+\varepsilon}, 
\end{align*}
Hence, the first seven estimates in (\ref{17112601}) are proved.  Analogously, we can prove the last three estimates in (\ref{17112601}) by using Propositions \ref{lem. large deviation linear part} and \ref{lem. large deviation nonlinear part},  (\ref{17110246}), the last two estimates in (\ref{17110250}). For instance, 
 from Proposition \ref{lem. large deviation linear part}, (\ref{17110250}) and (\ref{17110246}), we have  the bound 
\begin{align}
|{\mb{u}}_\gamma A_\gamma {\mb{u}}_\gamma'|\prec \frac{1}{M} \big|\mr{Tr} A_\gamma\Gamma\big|+ \sqrt{\frac{\mathrm{Tr} |A_\gamma\Gamma|^2}{M^2}}\prec \frac{n}{M} \mr{Tr} |A_\gamma|+ \sqrt{\frac{\mathrm{Tr} |A_\gamma|^2}{M}} \prec n^{\frac13+\varepsilon}. \label{17110361}
\end{align}
We omit the details of the  estimates for the last two estimates in (\ref{17112601}).  They can be obtained similarly. 

 Next, we prove (\ref{171201100}).  By the large deviation inequality (\ref{17101051}), we have 
\begin{align*}
&|\mc{P}_\gamma|\prec  \sqrt{\frac{n}{M^2} \mr{Tr} |B_\gamma|^2}+\sqrt{\frac{1}{M^2}\sum_{\ell=1}^n \Big|\sum_{j=\ell+1}^n (TB_\gamma)_{j,(\ell j)}\Big|^2},\nonumber\\
&|\mc{Q}_\gamma|\prec   \sqrt{\frac{n}{M^2} \mr{Tr} |A_\gamma|^2}+\sqrt{\frac{1}{M^2}\sum_{\ell=1}^n \Big|\sum_{j=\ell+1}^n (TA_\gamma)_{j,(\ell j)}\Big|^2}.
\end{align*}
Applying (\ref{17110241}) and (\ref{17110246}) we obtain
\begin{align*}
\sqrt{\frac{n}{M^2} \mr{Tr} |B_\gamma|^2}\prec n^{-\frac56+\varepsilon}, \qquad \sqrt{\frac{n}{M^2} \mr{Tr} |A_\gamma|^2}\prec  n^{-\frac16+\varepsilon}. 
\end{align*}
To show the last two estimates in (\ref{17110250}), we will prove the bound
\begin{align}
\frac{1}{M^2}\sum_{\ell=1}^n \Big|\sum_{j=\ell+1}^n (TB_\gamma)_{j,(\ell j)}\Big|^2\prec n^{-1}, \label{17112801}\\
\frac{1}{M^2}\sum_{\ell=1}^n \Big|\sum_{j=\ell+1}^n (TA_\gamma)_{j,(\ell j)}\Big|^2\prec n^{-\frac13+\varepsilon}.  \label{17112802}
\end{align}

The proofs of (\ref{17112801}) and (\ref{17112802}) can be done in the same way. We thus present the details for the proof of (\ref{17112801}) only. Recall the definition of $B_\gamma=(\Theta_\gamma^{(\gamma)})' G_\gamma^{(\gamma)}\Theta_\gamma^{(\gamma)}$ from (\ref{17110230}). To ease the presentation, in the sequel, we work with $\Theta_\gamma$ and $G_\gamma$ instead of the minors $\Theta_\gamma^{(\gamma)}$ and $G_\gamma^{(\gamma)}$, and prove 
\begin{align}
\frac{1}{M^2}\sum_{\ell=1}^n \Big|\sum_{j=\ell+1}^n (T(\Theta_\gamma)' G_\gamma\Theta_\gamma)_{j,(\ell j)}\Big|^2\prec n^{-1} \label{17112820}
\end{align}
instead of (\ref{17112801}). Further, we only show the details for the proof of (\ref{17112820}) for $\gamma=0$  to ease the presentation. The extension to general $\gamma$ will be explained at the end. Observe that $\Theta_0=\Theta$. 
 We first notice from (\ref{17112630}) that 
\begin{align}
\mathbf{e}_i T\Theta'= \sum_{\beta>i}\mb{\theta}_{(i\beta)}'-\sum_{\alpha<i} \mb{\theta}_{(\alpha i)}'= \sum_{\alpha}\mb{\theta}_{(i\alpha)}', \label{17112635}
\end{align}
where we use the fact $\mb{\theta}_{(\alpha i)}= -\mb{\theta}_{(i\alpha)}$.  Here we use $\mathbf{e}_i$ to represent the $n$-dimensional row vector whose $i$th coordinate is $1$ and the others are $0$. 
Hence, we can write
\begin{align}
\sum_{j=\ell+1}^n(T\Theta' G\Theta)_{j,(\ell j)}=\sum_{j=\ell+1}^n\sum_{\alpha}  \mb{\theta}_{(j\alpha)}' G  \mb{\theta}_{(\ell j)}= \mr{Tr} G  \Big(\sum_{j=\ell+1}^n\mb{\theta}_{(\ell j)} \Big(\sum_{\alpha}  \mb{\theta}_{(j\alpha)}'\Big)\Big).\label{17112860}
\end{align}
Using the decomposition in (\ref{17120205}), we can write
\begin{align}
\sum_{j=\ell+1}^n\mb{\theta}_{(\ell j)}  \big(\sum_{\alpha}\mb{\theta}_{(j\alpha)}'\big)= &\mb{\theta}_{(\ell \cdot)} \big(\sum_{j=\ell +1}^n\sum_{\alpha}\mb{\theta}_{(j\alpha)}'\big)-\sum_{j=\ell+1}^n  \mb{\theta}_{(j\cdot)}\big(\sum_{\alpha}\mb{\theta}_{(j\alpha)}'\big)\nonumber\\
& +  \sum_{j=\ell+1}^n\bar{\mb{\theta}}_{(\ell j)}  \big(\sum_{\alpha}\mb{\theta}_{(j\alpha)}'\big). \label{171126117}
\end{align}
Therefore, to show (\ref{17112820}) with $\gamma=0$, it suffices to prove 
\begin{align}
&\frac{1}{M^2}\sum_{\ell}  \Big| \big(\sum_{j=\ell +1}^n\sum_{\alpha}\mb{\theta}_{(j\alpha)}'\big)G \mb{\theta}_{(\ell \cdot)} \Big|^2\prec n^{-1},\label{17120101}\\
& \frac{1}{M^2}\sum_{\ell}  \Big| \mr{Tr} G \sum_{j=\ell+1}^n  \mb{\theta}_{(j\cdot)}\big(\sum_{\alpha}\mb{\theta}_{(j\alpha)}'\big)\Big|^2\prec n^{-1}, \label{17120102}\\
&   \frac{1}{M^2}\sum_{\ell}  \Big|\mr{Tr}  G \sum_{j=\ell+1}^n\bar{\mb{\theta}}_{(\ell j)}  \big(\sum_{\alpha}\mb{\theta}_{(j\alpha)}'\big)\Big|^2\prec n^{-1}.   \label{17120103}
\end{align}

For (\ref{17120101}), by (\ref{17113070}), we have  $\|\sum_{j=\ell+1}^n\sum_{\alpha}\mb{\theta}_{(j\alpha)}\|_{\infty}\prec\sqrt{n}$, and thus  $\|\sum_{j=\ell+1}^n\sum_{\alpha}\mb{\theta}_{(j\alpha)}\|\prec n$. 
This further implies 
\begin{align*}
\frac{1}{M^2}\sum_{\ell}  \Big| \big(\sum_{j=\ell +1}^n\sum_{\alpha}\mb{\theta}_{(j\alpha)}'\big)G \mb{\theta}_{(\ell \cdot)} \Big|^2 \prec  \frac{1}{n^2}\sum_{\ell} \|G \mb{\theta}_{(\ell \cdot)}\|^2=  \frac{1}{n^2} \mr{Tr}G \big(\sum_{\ell}\mb{\theta}_{(\ell \cdot)}\mb{\theta}_{(\ell \cdot)}'\big)G^*.
\end{align*}
Now, note that $\sum_{\ell}\mb{\theta}_{(\ell \cdot)}\mb{\theta}_{(\ell \cdot)}'$ is a sample covariance matrix with mean zero and variance $\frac{1}{3M}$ entries $\frac{1}{\sqrt{M}}v_{k,(i\cdot)}$'s. Then from Proposition \ref{pro.large deviation for operator norm}, it is easy to check
$
\|\sum_{\ell}\mb{\theta}_{(\ell \cdot)}\mb{\theta}_{(\ell \cdot)}'\|\prec \frac{1}{n}. 
$ 
Hence, 
\begin{align*}
\frac{1}{M^2}\sum_{\ell}  \Big| \big(\sum_{j=\ell +1}^n\sum_{\alpha}\mb{\theta}_{(j\alpha)}'\big)G \mb{\theta}_{(\ell \cdot)} \Big|^2\prec \frac{1}{n^3} \mr{Tr}|G|^2 \prec n^{-\frac53+\varepsilon}. 
\end{align*}

For (\ref{17120102}), we further write
\begin{align}
\sum_{j=\ell+1}^n  \mb{\theta}_{(j\cdot)}\big(\sum_{\alpha}\mb{\theta}_{(j\alpha)}'\big)=n\sum_{j=\ell+1}^n  \mb{\theta}_{(j\cdot)}\mb{\theta}_{(j\cdot)}'- \big(\sum_{j=\ell+1}^n  \mb{\theta}_{(j\cdot)}\big) \big(\sum_{\alpha}\mb{\theta}_{(\alpha\cdot)}'\big)\nonumber\\
+\sum_{j=\ell+1}^n  \mb{\theta}_{(j\cdot)}\big(\sum_{\alpha}\bar{\mb{\theta}}_{(j\alpha)}'\big), \label{17120110}
\end{align}

Again, from Proposition \ref{pro.large deviation for operator norm}, we can check  $\|n\sum_{j=\ell+1}^n  \mb{\theta}_{(j\cdot)}\mb{\theta}_{(j\cdot)}'\|\prec 1$. In addition, according to the large deviation of the sum of independent random variables,   it is easy to see that $\|\sum_{j=\ell+1}^n  \mb{\theta}_{(j\cdot)}\|_{\infty}=O_\prec(\frac{1}{\sqrt{n}})$ and $\|\sum_{\alpha}\mb{\theta}_{(\alpha\cdot)}\|_{\infty}=O_\prec(\frac{1}{\sqrt{n}})$. Consequently, we have  the bounds $\|\sum_{j=\ell+1}^n  \mb{\theta}_{(j\cdot)}\|\prec 1$ and  $\|\sum_{\alpha}\mb{\theta}_{(\alpha\cdot)}\|\prec 1$.  For the last term in the RHS of (\ref{17120110}), we write 
\begin{align*}
\sum_{j=\ell+1}^n  \mb{\theta}_{(j\cdot)}\big(\sum_{\alpha}\bar{\mb{\theta}}_{(j\alpha)}'\big)&= \big(\mb{\theta}_{(\ell+1,\cdot)}, \cdots, \mb{\theta}_{(n\cdot)}\big) \big(\sum_{\alpha}\bar{\mb{\theta}}_{(\ell+1,\alpha)}, \cdots, \sum_{\alpha}\bar{\mb{\theta}}_{(n,\alpha)}\big)'\nonumber\\
&=: \Theta_{\ell \Cdot[2]} \bar{\Theta}_{\ell+}'.
\end{align*}
Using Proposition \ref{pro.large deviation for operator norm} again, we have 
\begin{align}
|\Theta_{\ell \Cdot[2]}|\prec \frac{1}{\sqrt{n}}.  \label{17120120}
\end{align}
In addition, we have 
\begin{align}
\|\bar{\Theta}_{\ell+}\|= \sqrt{\|\bar{\Theta}_{\ell+}\bar{\Theta}_{\ell+}'\|}=\sqrt{\Big{\|} \sum_{i=\ell+1}^n \big( \sum_{\alpha}\bar{\mb{\theta}}_{(i,\alpha)}\big)\big( \sum_{\alpha}\bar{\mb{\theta}}_{(i,\alpha)}\big)' \Big{\|}}\prec \sqrt{n}, \label{17120131}
\end{align}
where in the last step we use the fact 
\begin{align}
\big{\|} \sum_{\alpha}\bar{\mb{\theta}}_{(i\alpha)}\big{\|}\prec 1, \label{17112680}
\end{align}
which follows from  (\ref{17101050}).  Therefore, we conclude 
\begin{align}
\Big{\|}\sum_{j=\ell+1}^n  \mb{\theta}_{(j\cdot)}\big(\sum_{\alpha}\mb{\theta}_{(j\alpha)}'\big)\Big{\|}\prec 1.  \label{17120170}
\end{align}
This implies 
\begin{align*}
\frac{1}{M^2}\sum_{\ell}  \Big| \mr{Tr} G \sum_{j=\ell+1}^n  \mb{\theta}_{(j\cdot)}\big(\sum_{\alpha}\mb{\theta}_{(j\alpha)}'\big)\Big|^2\prec \frac{n}{M^2} (\mr{Tr} |G|)^2\prec n^{-1},
\end{align*}
which proves (\ref{17120102}).  Here in the last step we use the fact $\text{Tr} |G|\prec n$ whose proof is analogous to (\ref{18021801}).  Again, we refer to Lemma 3.10 of  \cite{BPZunpub} and its proof  for a similar derivation of such bound from the local law.
For (\ref{17120103}),  we write 
\begin{align}
\sum_{j=\ell+1}^n\bar{\mb{\theta}}_{(\ell j)}  \big(\sum_{\alpha}\mb{\theta}_{(j\alpha)}'\big)=  &n\sum_{j=\ell+1}^n\bar{\mb{\theta}}_{(\ell j)} \mb{\theta}_{(j\cdot)}'- \sum_{j=\ell+1}^n\bar{\mb{\theta}}_{(\ell j)}  \big(\sum_{\alpha}\mb{\theta}_{(\alpha\cdot)}'\big)\nonumber\\
&\qquad\qquad+\sum_{j=\ell+1}^n\bar{\mb{\theta}}_{(\ell j)}  \big(\sum_{\alpha}\bar{\mb{\theta}}_{(j\alpha)}'\big). \label{171126160}
\end{align}
For the first term in the RHS of (\ref{171126160}),  we  have
\begin{align*}
\sum_{j=\ell+1}^n\bar{\mb{\theta}}_{(\ell j)} \mb{\theta}_{(j\cdot)}'= \Big( \bar{\mb{\theta}}_{(\ell,\ell+1)}, \cdots, \bar{\mb{\theta}}_{(\ell n)}\Big) \Big(\mb{\theta}_{(\ell+1, \cdot)}, \cdots, \mb{\theta}_{(n\cdot)} \Big)'=: \bar{\Theta}_{\ell}  \Theta_{\ell\Cdot[2]}'.
\end{align*}  
Conditioning on the randomness of $w_{k\ell}$ for all $k\in\llbracket 1, p\rrbracket$ and a fixed $\ell$, the random matrix $\bar{\Theta}_\ell$ is also a mean $0$ data matrix with (conditionally) independent entries. Hence, conditioning on $w_{k\ell}$ for all $k\in\llbracket 1, p\rrbracket$  and a fixed $\ell$, the matrix $\bar{\Theta}_\ell\bar{\Theta}_\ell'$ is again a sample covariance matrix. From Proposition \ref{pro.large deviation for operator norm}, we have
\begin{align}
\|\bar{\Theta}_\ell\|\prec \frac{1}{\sqrt{n}}. \label{171126190}
\end{align}
This together with (\ref{17120120}) yields 
$
{\|}n\sum_{j=\ell+1}^n\bar{\mb{\theta}}_{(\ell j)} \mb{\theta}_{(j\cdot)}'{\|}\prec 1. 
$
Further, by (\ref{17101050}) one can check that $\|\sum_{j=\ell+1}^n\bar{\mb{\theta}}_{(\ell j)}\|_\infty\prec \frac{1}{\sqrt{n}}$. Then the second term in the RHS of (\ref{171126160}) can be bounded by the facts $\|\sum_{j=\ell+1}^n\bar{\mb{\theta}}_{(\ell j)}\|\prec 1$  and   $\|\sum_{\alpha}\mb{\theta}_{(\alpha\cdot)}\|\prec 1$. For the last term in the RHS of  (\ref{171126160}), we observe that
\begin{align*}
\|\sum_{j=\ell+1}^n\bar{\mb{\theta}}_{(\ell j)}  \big(\sum_{\alpha}\bar{\mb{\theta}}_{(j\alpha)}'\big)\|=  \|\bar{\Theta}_{\ell} \bar{\Theta}_{\ell+}' \|\prec 1,
\end{align*} 
where in the last step we use (\ref{17120131}) and (\ref{171126190}). Therefore, we have 
\begin{align}
\|\sum_{j=\ell+1}^n\bar{\mb{\theta}}_{(\ell j)}  \big(\sum_{\alpha}\mb{\theta}_{(j\alpha)}'\big)\|\prec 1. \label{17120176}
\end{align}
This implies 
\begin{align*}
\frac{1}{M^2}\sum_{\ell}  \Big|\mr{Tr}  G \sum_{j=\ell+1}^n\bar{\mb{\theta}}_{(\ell j)}  \big(\sum_{\alpha}\mb{\theta}_{(j\alpha)}'\big)\Big|^2\prec \frac{n}{M^2}  (\mr{Tr}  |G|)^2\prec \frac{1}{n}.
\end{align*}
Again, in the last step above we use the fact $\mr{Tr}  |G| \prec n$. 
This proves  (\ref{17120103}). Hence, we complete the proof of  (\ref{17112820}) for $\gamma=0$.

 For $\gamma>0$, we denote by $\mb{\theta}_{(ij)}^\gamma$ the $(ij)$-th column of the matrix $\Theta_\gamma$, i.e. the $k$-th component of  $\sqrt{M}\mb{\theta}_{(ij)}^\gamma$ is $ v_{k,(ij)}$ if $k\leq \gamma$, and is $ (v_{k,(i\cdot)}-v_{k,(j\cdot)}+h_{k,(ij)})$ otherwise. We then further denote by $\bar{\mb{\theta}}_{(ij)}^{\gamma}$ the random vector whose $k$-th component is $\frac{1}{\sqrt{M}}\bar{v}_{k,(ij)}$ if $k\leq \gamma$ and is $\frac{1}{\sqrt{M}} h_{k,(ij)}$ otherwise.  Replacing $\bar{\mb{\theta}}_{(ij)}$ by $\bar{\mb{\theta}}_{(ij)}^\gamma$ in the above discussion, we can prove (\ref{17112820}) for general $\gamma$ similarly. 
Performing the proof with the minors $\Theta_\gamma^{(\gamma)}$ and $G_\gamma^{(\gamma)}$ instead of $\Theta_\gamma$ and $G_\gamma$, we can conclude (\ref{17112801}).  Similarly, we can prove (\ref{17112802}). We omit the details.  This completes the proof of (\ref{171201100}).

Next, we show the estimates in (\ref{17112610}).  For the first estimate in (\ref{17112610}), we have 
 \begin{align*}
\big|\mathbb{E}\mb{u}_\gamma A_\gamma \mb{u}_\gamma'\hat{\mc{W}}_\gamma \big|&=  \big|\mr{Cov}\big(\mb{u}_\gamma A_\gamma \mb{u}_\gamma', {\mb{u}}_\gamma{\mb{u}}_\gamma'\big)\big|= \big| \mr{Cov} \big(\mb{v}_{\gamma,\Cdot[2]}T A_\gamma T'\mb{v}_{\gamma,\Cdot[2]}', {\mb{v}}_{\gamma,\Cdot[2]}TT'{\mb{v}}_{\gamma,\Cdot[2]}'\big)\big|\nonumber\\
&\prec \Big|\frac{1}{M^2} \mr{Tr} T A_\gamma T' TT' \Big|+\Big|\frac{1}{M^2}\sum_{i=1}^n  (T A_\gamma T')_{ii} (TT')_{ii} \Big|\nonumber\\
&\prec\Big|\frac{1}{M^2} \mr{Tr} A_\gamma \Gamma^2 \Big|+\Big|\frac{n}{M^2}\mr{Tr} A_\gamma \Gamma\Big|\prec \frac{1}{M} \mr{Tr} |A_\gamma| \prec n^{-\frac23+\varepsilon},
 \end{align*}
 where we use the identity (\ref{17102202}), the facts $T'T=\Gamma$, $(TT')_{ii}=n-1$ and (\ref{17110250}). 

For the second estimate in (\ref{17112610}), we have 
\begin{align*}
&\mathbb{E}\big( {\mb{u}}_\gamma A_\gamma(z_1) {\mb{u}}_\gamma' {\mb{u}}_\gamma B_\gamma(z_2) \bar{\mb{v}}_\gamma'\big)= \mathbb{E}\big( {\mb{v}}_{\gamma,\Cdot[2]} TA_\gamma(z_1)T' {\mb{v}}_{\gamma,\Cdot[2]}' {\mb{v}}_{\gamma,\Cdot[2]} T B_\gamma(z_2) \bar{\mb{v}}_\gamma'\big)\nonumber\\
&= \frac{1}{M^2}\sum_{a,b,c} \sum_{i<j} \mathbb{E}\Big((TA_\gamma(z_1)T')_{ab}  (T B_\gamma(z_2))_{c, (ij)}\Big)  \mathbb{E} \big(v_{\gamma, (a\cdot)}v_{\gamma, (b\cdot)} v_{\gamma, (c\cdot)}  \bar{v}_{\gamma,(ij)}\big) 
\end{align*}
Due to (\ref{17111970}) and the fact that $v_{k,(i\cdot)}$'s are all centered and i.i.d, we have  $\mathbb{E} v_{\gamma, (a\cdot)}v_{\gamma, (b\cdot)} v_{\gamma, (c\cdot)}  \bar{v}_{\gamma,(ij)}\neq 0$ only when two of $a,b,c$ are $i$ and one is $j$, or two of them are $j$ and one is $i$. We only show the details for the estimates in the following case: $a=b=i$, $c=j$. All the other cases can be done analogously. More specifically, we will show in details the following estimate
\begin{align}
&\Big|\frac{1}{M^2} \sum_{i<j} (TA_\gamma(z_1)T')_{ii}  (T B_\gamma(z_2))_{j, (ij)}\Big|\prec n^{-\frac12+\varepsilon}.\label{171128106}
\end{align}
 Recall the definitions $A_\gamma=(\Theta_\gamma^{(\gamma)})' (G_\gamma^{(\gamma)})^2\Theta_\gamma^{(\gamma)}$ and $B_\gamma=(\Theta_\gamma^{(\gamma)})' G_\gamma^{(\gamma)}\Theta_\gamma^{(\gamma)}$ from (\ref{17110230}).  Similarly to the strategy we used in the proof of (\ref{17112801}), to ease the presentation, we only show the details of the proof with $\Theta_\gamma^{(\gamma)}$ and $G_\gamma^{(\gamma)}$ replaced by $\Theta$ and $G$, respectively, i.e., we will prove the estimate
\begin{align}
&\Big|\frac{1}{M^2} \sum_{i<j} (T\Theta' (G(z_1))^2\Theta T')_{ii}  (T \Theta' G(z_2)\Theta)_{j, (ij)}\Big|\prec n^{-\frac12+\varepsilon}. \label{171128103}
\end{align}
 Using (\ref{17112635}), we can then write 
\begin{align}
&\frac{1}{M^2} \sum_{i<j} (T\Theta' (G(z_1))^2\Theta T')_{ii}  (T \Theta' G(z_2)\Theta)_{j, (ij)}\nonumber\\
&=\frac{1}{M^2} \sum_{i<j} \big(\sum_{\alpha}\mb{\theta}_{(i\alpha)}'\big)G^2(z_1) \big(\sum_{\alpha}\mb{\theta}_{(i\alpha)}\big)  \big(\sum_{\alpha}\mb{\theta}_{(j\alpha)}'\big)G(z_2)\mb{\theta}_{(ij)}\nonumber\\
&= \frac{1}{M^2}  \mr{Tr} \bigg(G^2(z_1) \sum_i \big(\sum_{\alpha}\mb{\theta}_{(i\alpha)}\big)\big(\sum_{\alpha}\mb{\theta}_{(i\alpha)}'\big)  \Big( \sum_{j=i+1}^n\big(\sum_{\alpha}\mb{\theta}_{(j\alpha)}'\big)G(z_2)\mb{\theta}_{(ij)}\Big)\bigg). \label{17112695}
\end{align}
Now, we claim that 
\begin{align}
\Big\|\sum_i \big(\sum_{\alpha}\mb{\theta}_{(i\alpha)}\big)\big(\sum_{\alpha}\mb{\theta}_{(i\alpha)}'\big) \Big\|\prec n, \label{17112650}
\end{align}
and 
\begin{align}
\Big|\sum_{j=i+1}^n\big(\sum_{\alpha}\mb{\theta}_{(j\alpha)}'\big)G(z_2)\mb{\theta}_{(ij)}\Big|\prec n^{\frac{7}{6}+\varepsilon}.  \label{17120160}
\end{align}
Then, using (\ref{17112650}) and (\ref{17120160}) to (\ref{17112695}), we conclude 
\begin{align*}
&\Big| \frac{1}{M^2} \sum_{i<j} (T\Theta' (G(z_1))^2\Theta T')_{ii}  (T \Theta' G(z_2)\Theta)_{j, (ij)}\Big|\nonumber\\
&\leq \frac{1}{M^2}  \mr{Tr} |G(z_1)|^2 \max_{i}\Big| \sum_{j=i+1}^n\big(\sum_{\alpha}\mb{\theta}_{(j\alpha)}'\big)G(z_2)\mb{\theta}_{(ij)}\Big|\nonumber\\
&\qquad\times \Big\|\sum_i \big(\sum_{\alpha}\mb{\theta}_{(i\alpha)}\big)\big(\sum_{\alpha}\mb{\theta}_{(i\alpha)}'\big)\Big\|  \prec n^{-\frac{11}{6}+\varepsilon} \mr{Tr} |G(z_1)|^2\prec n^{-\frac12+\varepsilon},  
\end{align*}
where in the last step we use the fact 
\begin{align*}
\mr{Tr} |G(z_1)|^2=\frac{1}{\eta} \Im \mr{Tr} G(z_1)= \frac{p}{\eta} \Im m(z_1)\prec n^{\frac43+\varepsilon},
\end{align*}
which follows from Lemma \ref{lem. local law for K gamma},  Lemma \ref{lem.properties of m} and the assumption on $z_1$ in Lemma \ref{lem. technical estimates}.  This proves (\ref{171128103}). The proof of (\ref{171128106}) can be done similarly. 
Therefore, what remains is to prove (\ref{17112650}) and (\ref{17120160}). We start with (\ref{17112650}).
Again, using the decomposition in  (\ref{17120205}), we can write
\begin{align}
&\sum_i \big(\sum_{\alpha}\mb{\theta}_{(i\alpha)}\big)\big(\sum_{\alpha}\mb{\theta}_{(i\alpha)}'\big)=n^2\sum_i \mb{\theta}_{(i\cdot)}\mb{\theta}_{(i\cdot)}'\nonumber\\
&\qquad+n\sum_i \mb{\theta}_{(i\cdot)}\Big(- \sum_\alpha \mb{\theta}_{(\alpha\cdot)}+ \sum_\alpha \bar{\mb{\theta}}_{(i \alpha)}\Big)'\nonumber\\
&\qquad+ n\sum_i \Big(-\sum_\alpha \mb{\theta}_{(\alpha\cdot)}+\sum_\alpha \bar{\mb{\theta}}_{(i \alpha)}\Big)\mb{\theta}_{(i\cdot)}' \nonumber\\
&\qquad+ \sum_i\Big(\sum_\alpha \mb{\theta}_{(\alpha\cdot)}- \sum_\alpha \bar{\mb{\theta}}_{(i \alpha)}\Big) \Big( \sum_\alpha \mb{\theta}_{(\alpha\cdot)}- \sum_\alpha \bar{\mb{\theta}}_{(i \alpha)}\Big)'. \label{17112671}
\end{align}
First, using the fact 
$
\|\sum_i \mb{\theta}_{(i\cdot)}\|\prec 1
$
together with (\ref{17112680}), we have
\begin{align}
\|(\sum_i \mb{\theta}_{(i\cdot)})(\sum_i \mb{\theta}_{(i\cdot)})'\|\prec 1,\qquad \| (\sum_\alpha \mb{\theta}_{(\alpha\cdot)})(\sum_\alpha \bar{\mb{\theta}}_{(i \alpha)})'\|\prec 1.  \label{17112670}
\end{align}
Plugging (\ref{17112670}), (\ref{17112680}) and the fact $\|\sum_{\ell}\mb{\theta}_{(\ell \cdot)}\mb{\theta}_{(\ell \cdot)}'\|\prec \frac{1}{n}$ into (\ref{17112671}) yields
\begin{align*}
&\sum_i \big(\sum_{\alpha}\mb{\theta}_{(i\alpha)}\big)\big(\sum_{\alpha}\mb{\theta}_{(i\alpha)}'\big)\nonumber\\
&= n\sum_i \mb{\theta}_{(i\cdot)}\Big( \sum_\alpha \bar{\mb{\theta}}_{(i \alpha)}\Big)' + n\sum_i \Big( \sum_\alpha \bar{\mb{\theta}}_{(i \alpha)}\Big)\mb{\theta}_{(i\cdot)}'+O_\prec(n),
\end{align*}
where the error term $O_\prec(n)$ represents some matrix with operator norm stochastically dominated by $n$.  Further, we write
\begin{align}
\sum_i \mb{\theta}_{(i\cdot)}\Big(\sum_\alpha \bar{\mb{\theta}}_{(i \alpha)}\Big)'= \Big(\mb{\theta}_{(1\cdot)}, \cdots, \mb{\theta}_{(n\cdot)} \Big)  \Big( \sum_\alpha \bar{\mb{\theta}}_{(1 \alpha)}, \cdots,  \sum_\alpha \bar{\mb{\theta}}_{(n \alpha)}\Big)'=: \Theta_{\Cdot[2]} \bar{\Theta}'_+. \label{171126110}
\end{align}
Observe that
\begin{align}
\|\sum_i \mb{\theta}_{(i\cdot)}\Big( \sum_\alpha \bar{\mb{\theta}}_{(i \alpha)}\Big)'\|=\|\Theta_{\Cdot[2]} \bar{\Theta}_+\|\leq \|\Theta_{\Cdot[2]}\| \|\bar{\Theta}_+\|\prec 1, \label{171126120}
\end{align}
where we use the large deviation for the largest eigenvalue of the sample covariance matrices again to conclude $\|\Theta_{\Cdot[2]}\|\prec \frac{1}{\sqrt{n}}$, and use (\ref{17112680}) to conclude that 
\begin{align}
\|\bar{\Theta}_+\|=\sqrt{\|\bar{\Theta}_+\bar{\Theta}_+'\|}=\sqrt{\|\sum_{i} (\sum_\alpha \bar{\mb{\theta}}_{(i \alpha)})(\sum_\alpha \bar{\mb{\theta}}_{(i \alpha)})' \|}\prec \sqrt{n}.  \label{171126200}
\end{align} 
Hence, we complete the proof of  (\ref{17112650}).

Next, we prove (\ref{17120160}).  Note that 
\begin{align}
&\Big|\sum_{j=i+1}^n\big(\sum_{\alpha}\mb{\theta}_{(j\alpha)}'\big)G(z_2)\mb{\theta}_{(ij)}\Big| \leq  \Big| \big(\sum_{j=i+1}^n\sum_{\alpha}\mb{\theta}_{(j\alpha)}'\big)G(z_2)\mb{\theta}_{(i\cdot)}\Big|\nonumber\\
&\qquad\qquad +\Big|\mr{Tr}G(z_2) \sum_{j=i+1}^n\mb{\theta}_{(j\cdot)} \big(\sum_{\alpha}\mb{\theta}_{(j\alpha)}'\big)\Big|+ \Big| \mr{Tr}G(z_2) \sum_{j=i+1}^n\bar{\mb{\theta}}_{(ij)} \big(\sum_{\alpha}\mb{\theta}_{(j\alpha)}'\big)\Big|. \label{17120180}
\end{align}
From (\ref{17113070}), we can get $\|\sum_{j=i+1}^n\sum_{\alpha}\mb{\theta}_{(j\alpha)}\|_{\infty}\prec \sqrt{n}$. Therefore, we have the bound $\|\sum_{j=i+1}^n\sum_{\alpha}\mb{\theta}_{(j\alpha)}\|\prec n$, which together with 
$ \|\mb{\theta}_{(i\cdot)}\|\prec \frac{1}{\sqrt{n}}$
implies
\begin{align}
\Big| \big(\sum_{j=i+1}^n\sum_{\alpha}\mb{\theta}_{(j\alpha)}'\big)G(z_2)\mb{\theta}_{(i\cdot)}\Big|\prec \sqrt{n}\|G(z_2)\|\leq\sqrt{n}\eta^{-1}= n^{\frac{7}{6}+\varepsilon}. \label{17120177}
\end{align}
Next, using (\ref{17120170}), we have 
\begin{align}
\Big|\mr{Tr}G(z_2) \sum_{j=i+1}^n\mb{\theta}_{(j\cdot)} \big(\sum_{\alpha}\mb{\theta}_{(j\alpha)}'\big)\Big|\prec \mr{Tr}|G(z_2)|\prec n.  \label{17120178}
\end{align}
Similarly, applying (\ref{17120176}), we have 
\begin{align}
\Big| \mr{Tr}G(z_2) \sum_{j=i+1}^n\bar{\mb{\theta}}_{(ij)} \big(\sum_{\alpha}\mb{\theta}_{(j\alpha)}'\big)\Big|\prec \mr{Tr}|G(z_2)|\prec n.  \label{17120179}
\end{align}
Combining (\ref{17120180})-(\ref{17120179}), we obtain (\ref{17120160}). 

Notice that in the proof above, we only used the local law and the crude bound $\|G(z)\|\leq \frac{1}{\eta}$. These technical inputs still work when we replace $z$, $z_1$ and $z_2$ by their complex conjugates.   Hence, the above proof still works if we replace some or all of $z, z_1,z_2$ by their complex conjugates.
This completes the proof of Lemma \ref{lem. technical estimates}. 
\end{proof}

\begin{proof}[Proof of Lemma \ref{lem. local law for K hat t}] Observe that $\wh{K}_t$ is a shift of  the matrix $(U+tH)(U+tH)'$. Hence, it suffices to show the local law for 
the latter. In addition, the matrix $(U+tH)(U+tH)'$ share the same structure with $\wh{K}$. Again, the proof of the local law of $(U+tH)(U+tH)'$ only relies on the large deviation estimates for linear and quadratic forms of the rows of $U$ and $H$. We omit the details and conclude the proof. 
\end{proof}

\begin{proof}[Proof of the last estimate in (\ref{17100901})]
 Similarly to (\ref{18021820}), we set 
\begin{align*}
\mc{Q}:= \mathrm{Tr}\big((HH'-\frac13 ) \wh{G}_t^2, \qquad \mf{n}^{(k,\ell)}:= \mc{Q}^k\overline{\mc{Q}}^\ell.
\end{align*}
Analogously to (\ref{18021833}), we have 
\begin{align*}
\mathbb{E} \big(\mf{n}^{(k,k)}\big) &= \mathbb{E}\Big(\mathrm{Tr}HH' \wh{G}_t^2 \mf{n}^{(k-1,k)}\Big)- \frac13\mathbb{E}\Big(\mathrm{Tr}\wh{G}_t^2 \mf{n}^{(k-1,k)}\Big)\nonumber\\
&=\sum_{a,(ij)}\mathbb{E} \Big(h_{a, (ij)}\big(H' \wh{G}_t^2\big)_{(ij),a} \mf{n}^{(k-1,k)}\Big)- \frac13\mathbb{E}\Big(\mathrm{Tr}\wh{G}_t^2 \mf{n}^{(k-1,k)}\Big)\nonumber\\
&=\frac{1}{3M} \sum_{a,(ij)} \mathbb{E} \Big(\Big(H'\frac{ \partial \wh{G}_t^2 }{\partial h_{a, (ij)}} \Big)_{(ij),a} \mf{n}^{(k-1,k)}\Big)\nonumber\\
&\qquad+\frac{k-1}{3M} \sum_{a, (ij)} \mathbb{E}\Big(\big(H' \wh{G}_t^2\big)_{(ij),a} \frac{\partial \mc{Q}}{\partial h_{a, (ij)}} \mf{n}^{(k-2,k)} \Big)\nonumber\\
&\qquad +\frac{k}{3M} \sum_{a, (ij)} \mathbb{E}\Big(\big(H' \wh{G}_t^2\big)_{(ij),a} \frac{\partial \overline{\mc{Q}}}{\partial h_{a, (ij)}} \mf{n}^{(k-1,k-1)} \Big).
\end{align*}
The remaining estimates  can be done similarly to those for the terms in the RHS of (\ref{18021833}). The main difference is: instead of the matrix $HU'$ in those terms with five $\wh{G}_t$ factors in (\ref{17110502}), we will have  the matrix $HH'-\frac{1}{3}I_M$.  Note that the factor $\frac{1}{\sqrt{n}}$ in the term $\frac{1}{M\sqrt{n}}\mr{Tr} |\wh{G}_t|^5$ in the first inequality of (\ref{171129100}) comes from  the first bound in (\ref{171104100}). We observe that the same bound  holds  for the matrix $HH'-\frac{1}{3}I_M$ as well, according to Proposition \ref{pro.large deviation for operator norm}. The rest of the proof is similar to that for the first estimate in (\ref{17100901}). We thus omit the details. 
\end{proof}

\noindent {\bf S4:  Some basic tools}

In this section, we collect some basic technical tools. 

\noindent $\bullet$ {\it Rank-one perturbation formula}

At various places, we use the following fundamental perturbation formula: for $\bs{\alpha},\bs{\beta}\in\C^N$ and an invertible $D\in M_N(\C)$, we have
\begin{align}
\big(D+\bs{\alpha}\bs{\beta}^*\big)^{-1}=D^{-1}-\frac{D^{-1}\bs{\alpha}\bs{\beta}^*D^{-1}}{1+\bs{\beta}^*D^{-1}\bs{\alpha}}\,, \label{091002Kevin}
\end{align}
as can be checked readily. A standard application of~\eqref{091002Kevin} is recorded in the following lemma. 
\begin{lem} \label{17101810}
Let $D\in M_N(\C)$ be Hermitian and let $Q\in M_N(\C)$ be arbitrary. Then, for any finite-rank Hermitian matrix $R\in M_N(\C)$ and $z=E+\ii\eta\in\C^+\,$, we have
\begin{align}
\left|\mr{Tr} \left(Q\big(D+R-z\big)^{-1}\right)-\mr{Tr} \left(Q(D-z)^{-1}\right)\right| &\leq \frac{\mathrm{rank}(R)\|Q\|}{\eta}.
 \label{091002}
\end{align}
\end{lem}
\begin{proof}
Let $z\in\C^+$ and $\bs{\alpha}\in\C^N$. Then from~\eqref{091002Kevin} we have
\begin{align}\label{091002000}
\mr{Tr} \Big(Q\big(D\pm \bs{\alpha}\bs{\alpha}^*-z\big)^{-1}\Big)-\mr{Tr} \Big(Q(D-z)^{-1} \Big)=\pm\frac{\bs{\alpha}^*(D-z)^{-1}Q(D-z)^{-1}\bs{\alpha}}{1\pm\bs{\alpha}^*(D-z)^{-1}\bs{\alpha}}\,.
\end{align}
We can thus estimate
\begin{align}
&\Big|\mr{Tr} \Big(Q\big(D\pm\bs{\alpha}\bs{\alpha}^*-z\big)^{-1}\Big)-\mr{Tr} \Big(Q(D-z)^{-1} \Big) \Big|\nonumber\\
&\leq {\|Q\|}\frac{\|(D-z)^{-1}\bs{\alpha}\|^2}{\big|1\pm \bs{\alpha}^*(D-z)^{-1}\bs{\alpha}\big|}= \frac{\|Q\|}{\eta}\frac{\bs{\alpha}^* \Im (D-z)^{-1}\bs{\alpha}}{\big|1\pm\bs{\alpha}^*(D-z)^{-1}\bs{\alpha}\big|}\leq \frac{\|Q\|}{\eta}\,. 
 \label{0910021111}
\end{align}
Since $R=R^*\in M_N(\C)$ has finite rank, we can write $R$ as a finite sum of rank-one Hermitian matrices of the form $\pm \bs{\alpha}\bs{\alpha}^*$. Thus iterating ~\eqref{0910021111} we get~\eqref{091002}.
\end{proof}

\noindent $\bullet$ {\it Resolvent identities}

The following lemma can be proved via elementary linear algebra; see Lemma 3.2 of  \cite{EYY}  for instance.
\begin{lem}[Resolvent identities] \label{lem.resolvent identity} We have the following identities
\begin{align}
&G_{ij}(z)= zG_{ii}(z) G_{jj}^{(i)}(z) \mb{v}_i \mc{G}^{(ij)}(z) \mb{v}_j', \qquad i\neq j \label{17111402} \\
& G_{ij}(z)= G_{ij}^{(k)}(z)+\frac{G_{ik}(z)G_{kj}(z)}{G_{kk}(z)}, \qquad i,j\neq k,  \label{17111403}\\
&  \mb{v}_i \mc{G}^{(i)}(z) \mb{v}_i'=-\frac{1}{zG_{ii}(z)}-1, \label{171202100}
\end{align}
\end{lem}

\noindent $\bullet$ {\it Properties of $\underline{m}$}

In the following lemma, we collect some basic properties of the function $\underline{m}(z):\mathbb{C}^+\to\mathbb{C}^+$ defined in (\ref{17113014}).   Let 
$
\kappa\equiv \kappa(E):=|E-\lambda_{+,c_n}|.
$
\begin{lem}\label{lem.properties of m}  For any $z\in E+\ii\eta\in \mc{D}(\varepsilon)$, we have 
\begin{align}
&|\underline{m}(z)|\sim 1,\label{17113020}\\
&\Im \underline{m}(z)\sim \left\{
\begin{array}{ccc}
\sqrt{\kappa+\eta}, & \text{if } E\leq \lambda_{+,c_n}\\\\
\frac{\eta}{\sqrt{\kappa+\eta}}, & \text{if } E\geq \lambda_{+, c_n} 
\end{array}
\right. \label{17113021}
\end{align}
\end{lem}

\noindent $\bullet$ {\it Operator norm of sample covariance matrices}

Here we record a well-known bound on the operator norm (largest eigenvalue) of sample covariance matrix. We refer to Theorem 2.10 of \cite{BEKYY} for instance.  
\begin{pro}[Theorem 2.10, \cite{BEKYY}]\label{pro.large deviation for operator norm} Let $X=(x_{ij})\in \mathbb{C}^{\mathsf{M}\times \mathsf{N}}$ be a random matrix with independent entries. Suppose that $\mathbb{E}x_{ij}=0$, $\mathbb{E}|x_{ij}|^2=\frac{1}{{\mathsf{N}}}$ and $\mathbb{E}|\sqrt{\mathsf{N}} x_{ij}|^q\leq C_q$ for some positive constant $C_q$ for all $i,j$ and given positive integer $q$.   Further, assume that $\mathsf{M}\equiv \mathsf{M}(\mathsf{N})$ satisfies $\mathsf{N}^{1/C}\leq \mathsf{M}\leq \mathsf{N}^C$ for some positive constant $C$.  Then we have 
\begin{align*}
\big|\|XX^*\|-\big(1+\sqrt{\frac{\mathsf{M}}{\mathsf{N}}}\big)^2\big|\prec \sqrt{\frac{\mathsf{M}}{\mathsf{N}}} (\min\{\mathsf{M}, \mathsf{N}\})^{-\frac23}. 
\end{align*}
\end{pro}

\
\
\
\

\noindent {\bf S5: More simulation study} In this section, we present more simulation results.
In Tables \ref{results2} and \ref{results3}, we state the  results of sizes and powers under the choices of $n=300$ and $n=900$, respectively, and four different values of $p$ can be chosen for each $n$.  Again, the simulation results are based on 1000 replications.
\begin{table}[!htbp]
\renewcommand{\arraystretch}{0.95}
\begin{center}
 \begin{tabular}{ccccccccccccccccccc} \hline
$p$  &$T_1$&$T_2$&$T_3$&$T_4$&$T_5$&$T_6$&$T_7$&  &$T_2$ &$T_4$&$T_6$&$T_7$&&$T_2$ &$T_4$ &$T_6$&$T_7$\\ \hline\\
   &\multicolumn{7}{c}{$\mathrm{H}_{0,1}$}&  & \multicolumn{4}{c}{$\mathrm{H}_{0,2}$} & & \multicolumn{4}{c}{$\mathrm{H}_{0,3}$} \\ \cline{2-8} \cline{10-13} \cline{15-18}\\
 100  &3.7    &6.1    &3.3   &3.8   &1      &1.7   &1.7&   &4.8   &3.5    &1.4   &1.7&   &5.5     &3.7  &1.4   &1.6\\
  200 &1.9    &4.5    &4.3   &4.1   &1.6   &1.8   &1.8&   &4.3   &4.1    &1.5   &1.8&   &6.2    &4      &1.6   &2.1\\
 280  &1.9    &5.7    &3.8   &4      &1.3   &1.5   &2.1&   &5.2   &3.1    &1.6   &2.1&   &5.5    &2.6   &2      &2.6\\
 400  &1       &5.1    &2.3   &3.6   &1.2   &1.3   &2.1&   &5.5   &3.7    &2.2   &3&      &5       &2.8   &2.8   &4.3\\ \\
  &\multicolumn{7}{c}{$\mathrm{H}_{a,1-1}$}&  & \multicolumn{4}{c}{$\mathrm{H}_{a,2-1}$} & & \multicolumn{4}{c}{$\mathrm{H}_{a,3-1}$} \\ \cline{2-8} \cline{10-13} \cline{15-18} \\
 100  &88.8  &91.3  &100   &100  &93.4     &94.1 &100 &  &92.5      &100    &91.7   &100   &&90.9  &100  &92.9  &100\\
  200 &31.8     &43     &100   &100  &24.5  &26.5 &92.8&  &40.3      &100    &27      &94.2  &&41.4  &100  &24.9  &93.8 \\ \\ 
  &\multicolumn{7}{c}{$\mathrm{H}_{a,1-2}$}  && \multicolumn{4}{c}{$\mathrm{H}_{a,2-2}$} & & \multicolumn{4}{c}{$\mathrm{H}_{a,3-2}$} \\ \cline{2-8} \cline{10-13} \cline{15-18}\\
 100  &86.6    &96.2   &6.3    &7.3    &99.8  &99.2   &99.2&   &55.7   &6.5  &57.3   &58   &&99.8   &7.2   &100   &100\\
200 &36.8     &66.7   &3       &3.7   &91.3  &86  &87.1&         &71.6   &5.9  &74.6   &75.2   &&88.3  &4.6   &97.3  &97.7\\
 280  &19.3   &47.5   &3.8  &4.7   &71.1  &63.2  &66.3&        &77.9   &4.5  &80.1   &80.9   &&68.2    &3.5   &88.1  &89.9\\
 400  &7.1    &25.4   &4.2      &4.6   &40.2  &33     &37&        &82.6   &5.5  &86.8    &86.9    &&44.6     &4.2   &68.9  &72.3\\ \hline
  \end{tabular}
  \caption{
\footnotesize{The sizes and powers (percentage) of $T_1$ to $T_7$ under different hypotheses and dimension $p$. Here we chose sample size $n=300$, $\delta=1$, $\tau_1=\tau_3=\frac32$ and $\tau_2=\frac{1}{40}$.  \label{results2}}
}
\end{center}
\end{table}

\begin{table}[!htbp]
\renewcommand{\arraystretch}{0.95}
\begin{center}
 \begin{tabular}{ccccccccccccccccccc} \hline
$p$  &$T_1$&$T_2$&$T_3$&$T_4$&$T_5$&$T_6$&$T_7$&  &$T_2$ &$T_4$&$T_6$&$T_7$&&$T_2$ &$T_4$ &$T_6$&$T_7$\\ \hline\\
   &\multicolumn{7}{c}{$\mathrm{H}_{0,1}$}&  & \multicolumn{4}{c}{$\mathrm{H}_{0,2}$} & & \multicolumn{4}{c}{$\mathrm{H}_{0,3}$} \\ \cline{2-8} \cline{10-13} \cline{15-18}\\
 300  &4.9    &6.3    &4.2   &3.8   &2.2      &2.5   &2.8&   &5   &4.1           &2.7   &3.3&   &5.2     &4.8  &2   &2.3\\
  600 &2.2    &4.7    &3.3    &3.5   &1        &2.1   &2.4&   &5.7   &4          &1.8   &2.5&   &3.8       &4.2  &2.5   &2.8\\
 840  &2         &5.6    &4.1   &4.6   &2.7     &2.8   &3.7&   &4.3   &4.7       &2.5   &2.9&   &6         &4.8   &2.1  &2.3\\
 1200  &1.2    &5.2    &3      &3.4      &2.5   &2.2   &2.9&   &4.4   &3.6      &3.2   &4.1&   &5.6     &2.9   &2.6  &3.2\\ \\
  &\multicolumn{7}{c}{$\mathrm{H}_{a,1-1}$}&  & \multicolumn{4}{c}{$\mathrm{H}_{a,2-1}$} & & \multicolumn{4}{c}{$\mathrm{H}_{a,3-1}$} \\ \cline{2-8} \cline{10-13} \cline{15-18} \\
 300  &88.9  &92.6        &100   &100  &100  &100   &100 &  &93.7      &100    &100  &100   &&90.3  &100  &100  &100\\
  600 &31.5     &42.7     &100   &100  &49.2  &48    &100&  &41.9      &100    &45.6   &100  &&42.1  &100  &47.5  &100 \\ \\ 
  &\multicolumn{7}{c}{$\mathrm{H}_{a,1-2}$}  && \multicolumn{4}{c}{$\mathrm{H}_{a,2-2}$} & & \multicolumn{4}{c}{$\mathrm{H}_{a,3-2}$} \\ \cline{2-8} \cline{10-13} \cline{15-18}\\
 300  &91.2    &99.7   &4.6    & 4.5   &100   &100   &100&      &99.5   &11.3   &99.8   &99.8   &&100      &6.5   &100   &100\\
600 &37.3     &73.4   &4.8  &4.9       &100   &99.8  &99.9&     &100  &11.8     &100   &100     &&93.8     &6.1   &100  &100\\
 840  &18.7   &48.2   &3.5  &4.1       &95.8   &91.8  &92.2&    &100  &12        &100   &100     &&76.5     &3.4   &99.9  &99.9\\
 1200  &6.9    &27   &4.3      &4.9     &68.5   &56.4     &59.4&  &100   &12.2   &100    &100    &&45.2     &4.1   &94.9  &95.2\\ \hline
  \end{tabular} 
  \caption{
\footnotesize{The sizes and powers (percentage) of $T_1$ to $T_7$ under different hypotheses and dimension $p$. Here we chose sample size $n=900$, $\delta=1$, $\tau_1=\tau_3=\frac32$ and $\tau_2=\frac{1}{40}$. \label{results3}}
}
\end{center}
\end{table}

\newpage
In the sequel, we consider another type of alternative: uncorrelated but dependent data. More specifically, we consider the following 

\vspace{1ex}
\noindent $\bullet$ $\mathrm{H}_{a,4}$: Let $\{x_i\}_{i=1}^p$ be i.i.d $N(0,1)$. Let $w_1=x_1^2-1$
and $w_2=-x_1^4+6x_1^2-3$. Similarly, set $w_3=x_2^2-1$ and $w_4=-x_2^4+6x_2^2-3$.
Finally, let $w_i=x_i$ for all $i\geq 5$.  

It is easy to check that $w_i$'s are uncorrelated but dependent. However, the dependence structure is rather local, i.e., only $w_i$ is dependent of $w_{i+1}$ for $i=1,2$.    The following table summarizes the powers of 7 statistics under the alternative $\mathrm{H}_{a,4}$. 

\begin{table}[!htbp]
\begin{center}
 \begin{tabular}{ccccccccccccccccccc} \hline
$(p,n)$  &&$T_1$&$T_2$&$T_3$&$T_4$&$T_5$&$T_6$&$T_7$ \\ \hline
 (100,300) & &10.4  &77        &52.6    &100  &5  &65.7   &100 &  \\
  (200,300) &&4.3     &30.9     &51      &100  &1.5  &14.4    &67.9&  \\ 
 (200,600)  &&5.3    &77.2   &53.2      &100  &3.5   &89.9   &100&   \\
(400,600) &&3.2     &32.1   &50.9       &100   &1.8   &18.7  &88.2&    \\
(300,900)  &&5.8   &78.1   &48.5        &100   &2.8   &96.9  &100&  \\
 (600,900)  &&3.4    &33.1   &48.1      &100    &1.4   &25.2     &96.2&  \\ \hline
  \end{tabular} 
  \caption{
\footnotesize{The  powers (percentage) of $T_1$ to $T_7$ under $\mathrm{H}_{a,4}$ \label{results4}}
}
\end{center}
\end{table}

Notice that under $\mathrm{H}_{a,4}$, the performance of all parametric statistics constructed from the matrix $R$, i.e., $T_1$, $T_3$ and $T_5$ perform poorly. Among all nonparametric  statistics, $T_4$ and $T_7$ outperform the others.

In Fig \ref{fig1}-\ref{fig6}, we plot the curves for powers of 7 statistics under 6 alternatives, with various choices of  parameters $\delta$, $\tau_1$, $\tau_2$ and $\tau_3$. The simulation was done with the choice $(p,n)=(400,600)$ and 1000 replications. More specifically, in Fig \ref{fig1}, \ref{fig3}, \ref{fig5}, the $x$-axis represents the value of $\delta$, and in Fig \ref{fig2}, \ref{fig4}, \ref{fig6}, the $x$-axis represents the value of $\tau_1$, $\tau_2$ and $\tau_3$, respectively. In all figures, the $y$-axis represents the power.  We use different colors for different statistics: red ($T_1$), green ($T_2$), cyan ($T_3$), blue ($T_4$), magenta ($T_5$), yellow ($T_6$), black ($T_7$).  

\begin{figure}[h]
           \begin{floatrow}
             \ffigbox{\includegraphics[width=6cm, height=5cm]{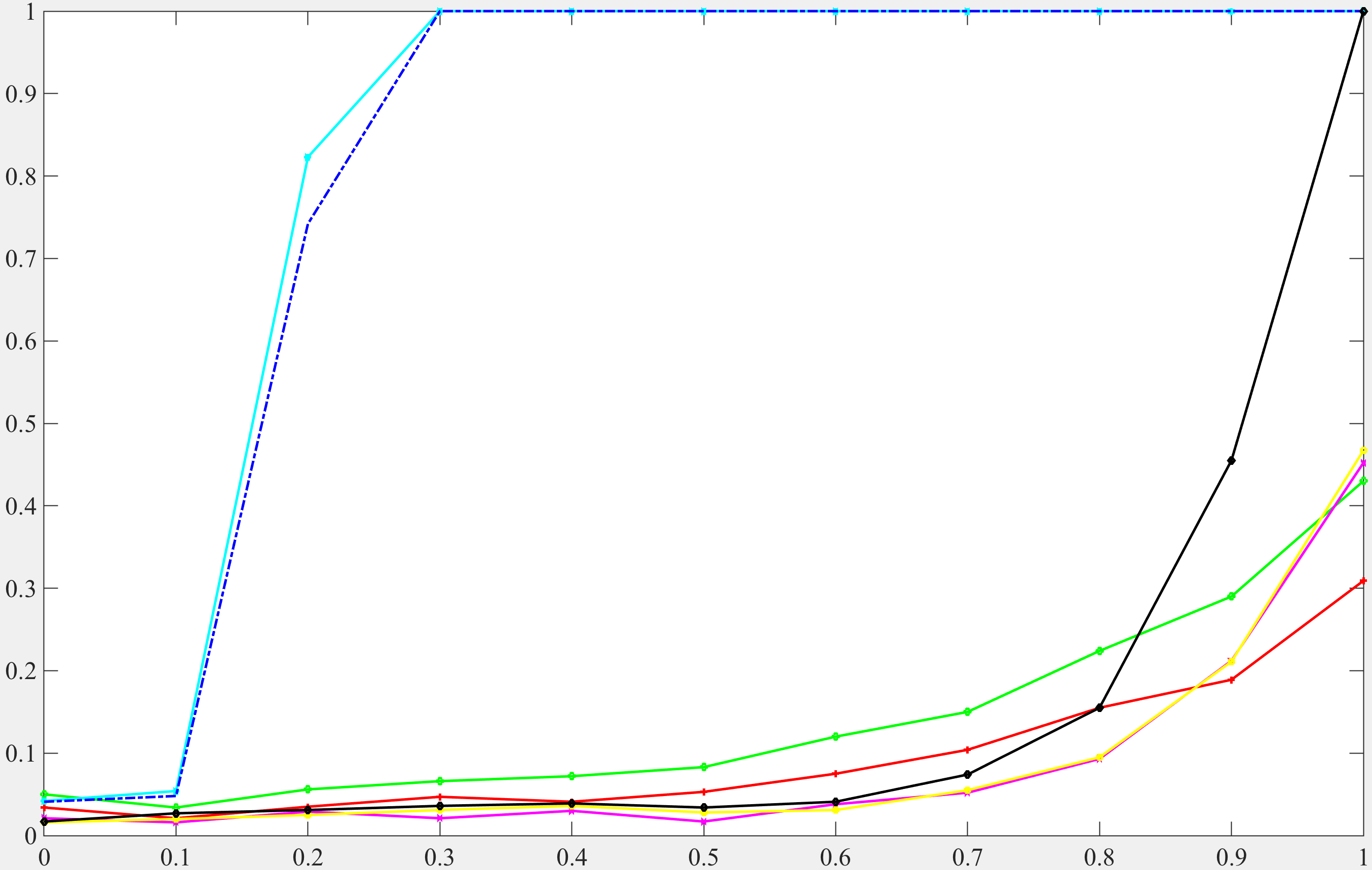}}{\caption{{\footnotesize{Powers under $\mathrm{H}_{a,1-1}$}}}\label{fig1}}
             \ffigbox{\includegraphics[width=6cm, height=5cm]{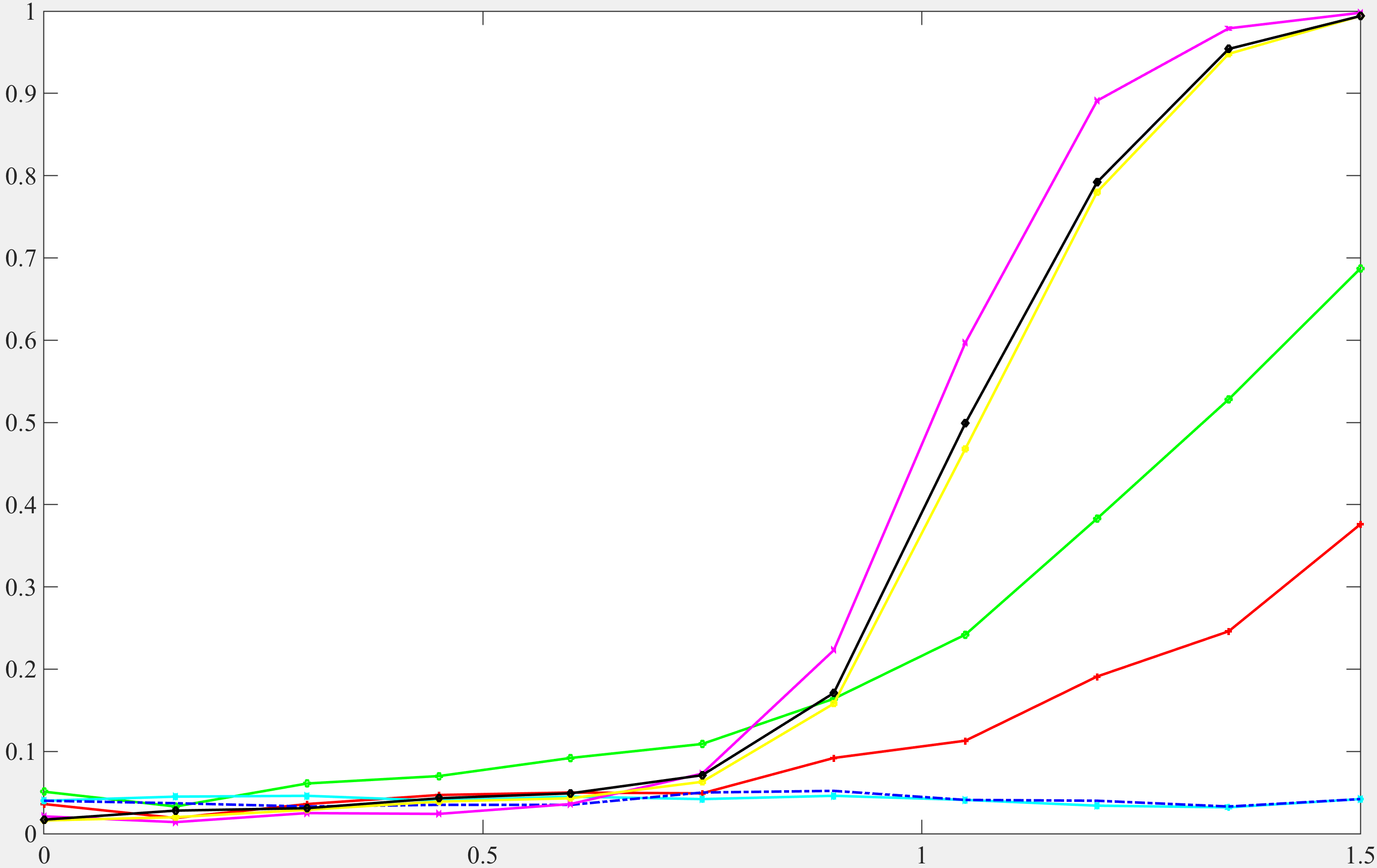}}{\caption{{\footnotesize{Powers under $\mathrm{H}_{a,1-2}$}}}\label{fig2}}
           \end{floatrow}
        \end{figure}

\begin{figure}[h]
           \begin{floatrow}
             \ffigbox{\includegraphics[width=6cm, height=5cm]{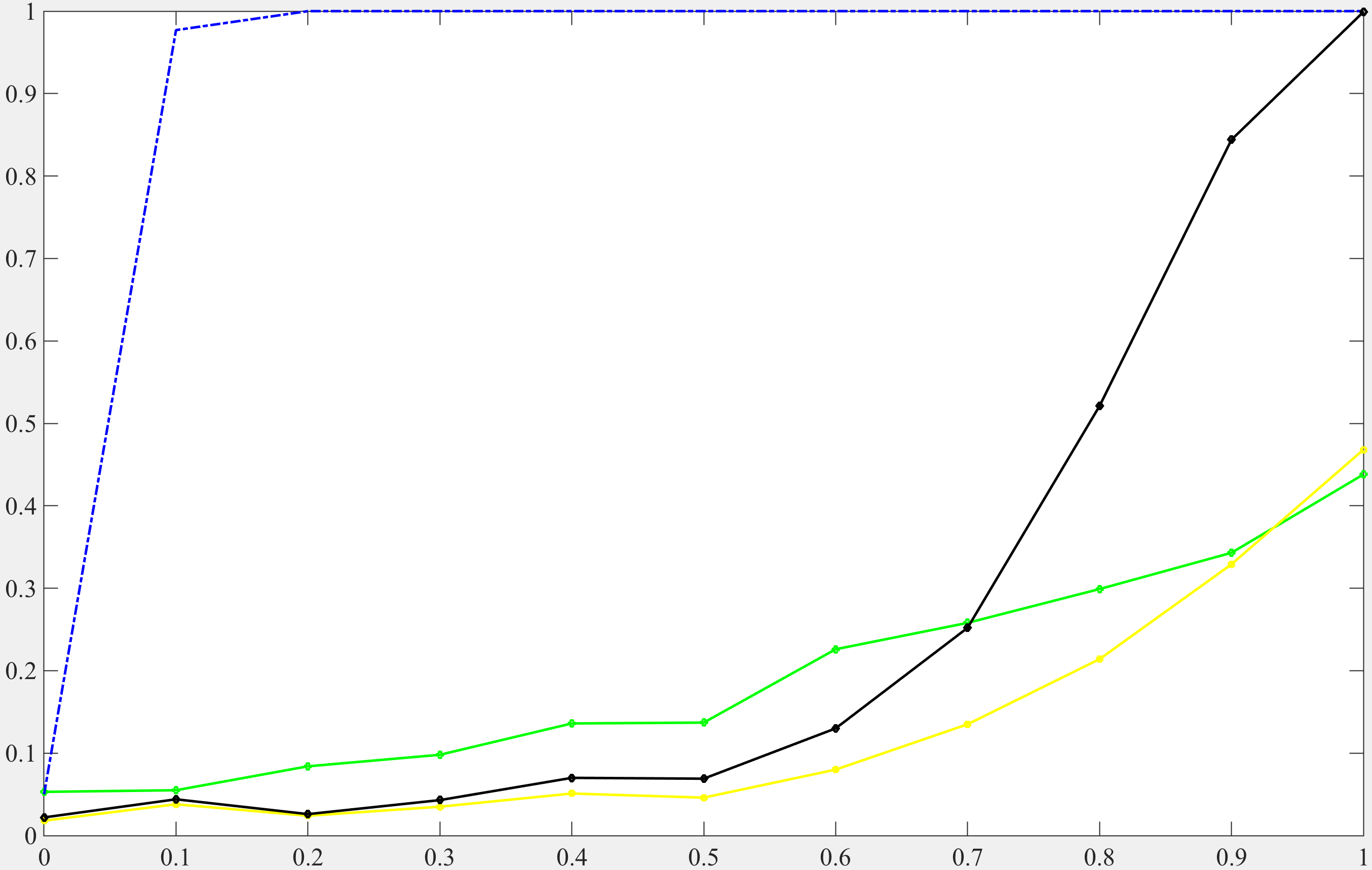}}{\caption{{\footnotesize{Powers under $\mathrm{H}_{a,2-1}$} }}\label{fig3}}
             \ffigbox{\includegraphics[width=6cm, height=5cm]{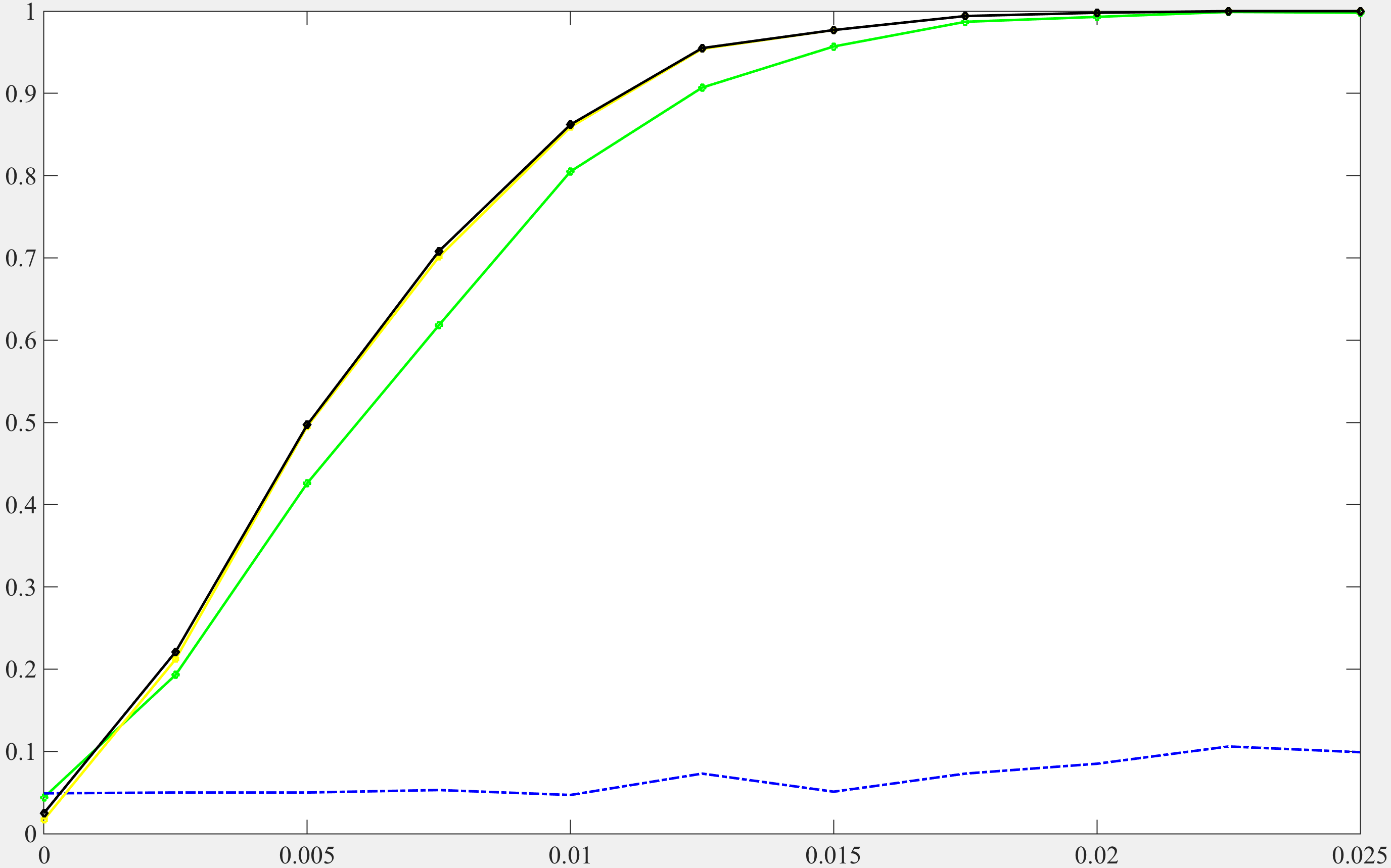}}{\caption{{\footnotesize{Powers under $\mathrm{H}_{a,2-2}$}}}\label{fig4}}
           \end{floatrow}
        \end{figure}

\begin{figure}[h]
           \begin{floatrow}
             \ffigbox{\includegraphics[width=6cm, height=5cm]{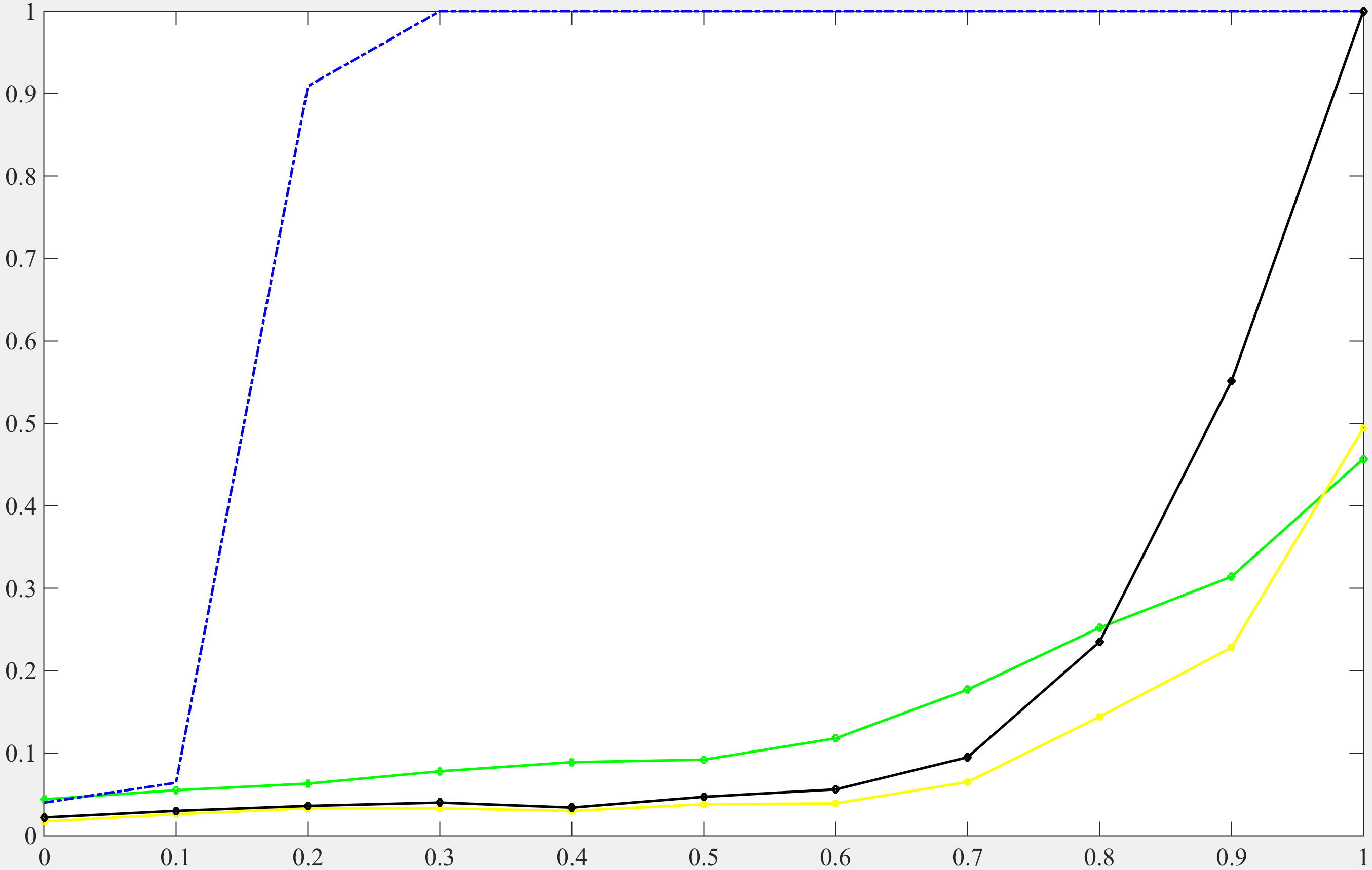}}{\caption{{\footnotesize{Powers under $\mathrm{H}_{a,3-1}$} }}\label{fig5}}
             \ffigbox{\includegraphics[width=6cm, height=5cm]{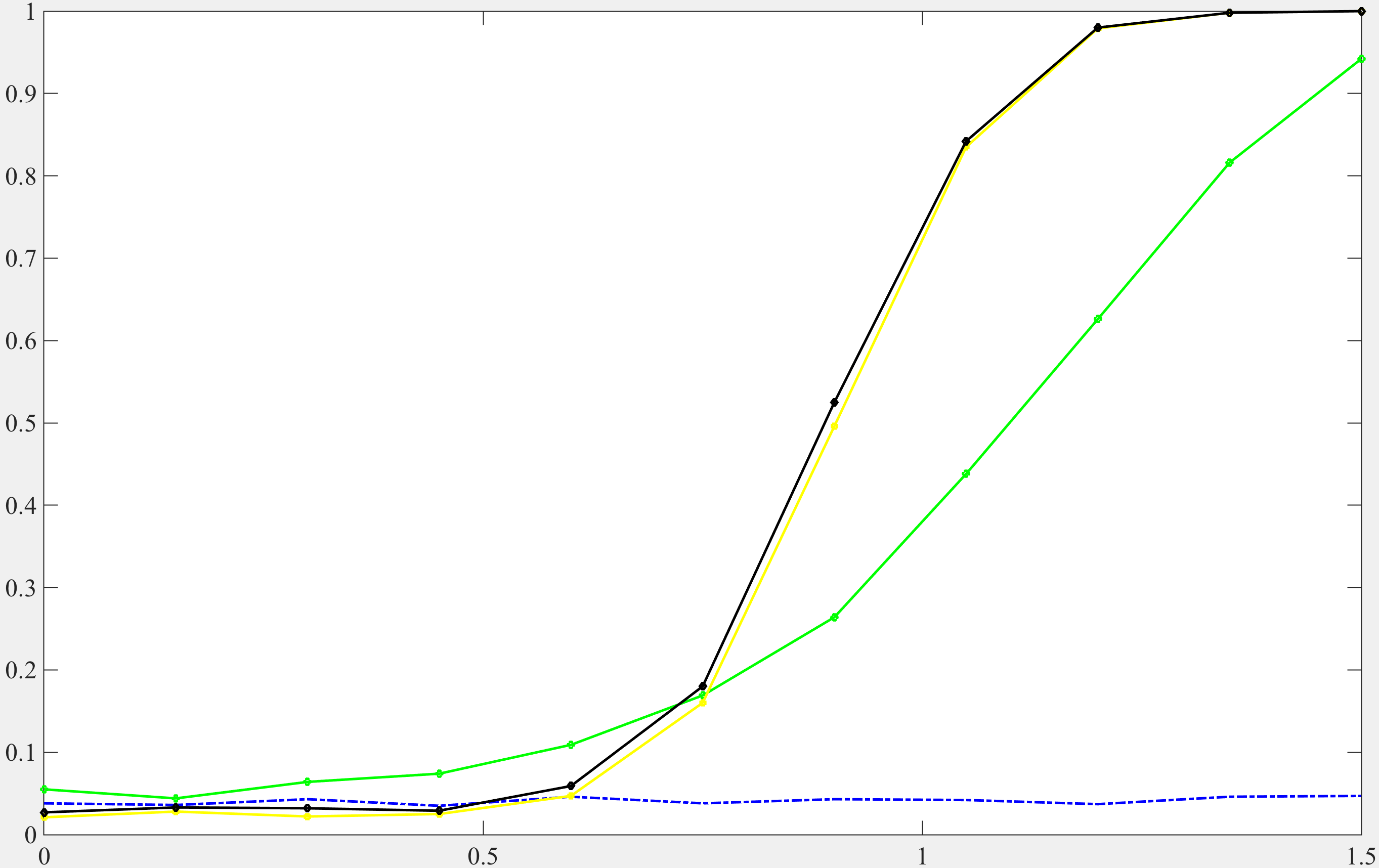}}{\caption{{\footnotesize{Powers under $\mathrm{H}_{a,3-2}$}}}\label{fig6}}
           \end{floatrow}
        \end{figure}

\end{document}